\documentclass[12pt]{article}

\usepackage{amsmath, amsfonts,amssymb, enumerate,  MnSymbol}
\usepackage{mathrsfs}  
\usepackage{graphicx}
\usepackage{subfig}
\usepackage[numbers, sort]{natbib}
\usepackage{hyperref}
\usepackage{stmaryrd}
\usepackage{xy}
\usepackage{mathtools}
\usepackage{xcolor}
\usepackage{tikz}
\usepackage{manfnt}
\usepackage{ntheorem}
\usepackage{accents}
\usepackage{yfonts}
\usepackage{pifont}
\usepackage{makecell}
\usepackage[table]{xcolor}

\usepackage{caption}
\usepackage{algorithm}
\usepackage[noend]{algpseudocode}
\algblock{Input}{EndInput}
\algnotext{EndInput}
\algblock{Output}{EndOutput}
\algnotext{EndOutput}

\usetikzlibrary{topaths,calc}

\input xy
\xyoption{all}

\colorlet{shade}{gray!40}

\numberwithin{equation}{section}

\newtheorem{theorem}{Theorem}[section]

\newtheorem{proposition}[theorem]{Proposition}
\newtheorem{lemma}[theorem]{Lemma}
\newtheorem{corollary}[theorem]{Corollary}
\newtheorem{rmk}[theorem]{Remark}
\newtheorem{definition}{Definition}[section]
\newtheorem{exmp}{Example}[section]

\newenvironment{proof}[1][Proof]{\begin{trivlist}
\item[\hskip \labelsep {\bfseries #1}]}{\end{trivlist}}

\newenvironment{remark}[1][Remark]{\begin{trivlist}
\item[\hskip \labelsep {\bfseries #1}]}{\end{trivlist}}

\newcommand{\qed}{\nobreak \ifvmode \relax \else
      \ifdim\lastskip<1.5em \hskip-\lastskip
      \hskip1.5em plus0em minus0.5em \fi \nobreak
      \vrule height0.75em width0.5em depth0.25em\fi}

\newcommand{\xmark}{\ding{55}}
\newcommand{\cmark}{\ding{51}}

\begin{document}

\title{Inductive construction of path homology chains and the structure of $\Omega_3(G;R)$}
\author{Matthew Burfitt and Tyrone Cutler}
\date{}
\maketitle

\begin{abstract}
    Path homology plays a central role in digraph topology and GLMY theory more generally. Unfortunately, the computation of the path homology of a digraph $G$ is a two-step process, and until now no complete description of even the underlying chain complex has appeared in the literature.
   
    In this paper we introduce an inductive method of constructing elements of the path homology chain modules $\Omega_n(G;R)$ from elements in the preceding two dimensions. This proceeds via the formation of what we call upper and lower \emph{extensions}, that are parametrised by certain labelled multigraphs which we introduce and call \emph{face multigraphs}.

    The inductive elements we construct generate $\Omega_*(G;R)$ when $R$ has characteristic $2$. With characteristic $0$ coefficients, the inductive elements at least generate $\Omega_i(G;R)$ for $i=0,1,2,3$. In low dimensions, the inductive elements coincide with the natural generators, and when the digraph contains no multisquares, the inductive elements coincide with the basis elements produced by Fu and Ivanov.

    Inductive elements provide a new concrete structure on the path chain complex that can be directly applied to understand path homology, under no restriction on the digraph $G$. 
    We employ inductive elements to construct explicit generators of $\Omega_3(G;R)$ for a ring $R$ of characteristic $0$ or $2$, answering an open question posed by Grigor'yan. Several universal coefficient statements for path homology are obtained as a byproduct.
\end{abstract}

\section{Introduction}

Path homology, one of the several proposed homology theories for directed graphs, is based around the idea of detecting cycles in equal length paths within a graphical structure. Its construction was laid out by Grigor'yan, Lin, Muranov, and Yau in the foundational paper \cite{Grigoryan2013}, which built on earlier work of Dimakis and M\"{u}ller-Hoissen \cite{Dimakis1994}, where elements of the dual cohomology theory were considered in the study of field theory on discrete space time. Path homology now sits in the broader study of the geometry, topology, and homotopy theory of digraphs and quivers which is often referred to as GLMY-theory.

Due to its good properties, path homology has emerged as one of the most important homology theories of digraphs. For instance, path homology satisfies certain Eilenberg-Steenrod axioms~\cite{Grigoryan2018} and shares variations of many of the important properties enjoyed by singular and simplicial homology, including functorality, homotopy invariance~\cite{Grigor'yan2014b}, and a K\"{u}nneth Theorem for the box product of digraphs~\cite{Grigoryan2017}. These fundamental properties are not necessarily shared by other homology theories of digraphs.

Another interesting aspect of path homology is its relationship with magnitude homology~\cite{Hepworth2017}, where as first noted by Asao~\cite{Asao2023} it appears as a graded submodule of the second page of the magnitude-path spectral sequence (introduced in~\cite[Remark 8.7]{Hepworth2017}), whose first page is the naturally bigraded magnitude homology. As it turns out, the first page of this spectral sequence contains the module of path chains as a submodule, so magnitude homology contains also more delicate information than just the path homology modules.

Now, a central problem with path homology lies in the fact that, unlike simplicial homology, there is no simple general description of a path chain basis. More precisely, path chains are provided indirectly through compatibility with the differential, and are not generated by singular mappings. This feature adds extra complexity to the computation of path homology. In contrast, the singular digraph homologies are valued for their comparative simplicity and ease of computation.

Denoting by $\Omega_n(G;R)$ the path chains in dimension $n$ with coefficients in a commutative ring $R$, canonical bases of $\Omega_0(G;R)$ and $\Omega_1(G;R)$ are generated by all vertices and all edges, respectively. The situation in dimension $2$ is made more complicated by the existence of multisquare digraphs for which $\Omega_2(G;R)$ has no canonical basis. However, for any digraph $G$, a basis of $\Omega_2(G;R)$ can always be chosen from a subset of generators corresponding to certain subdigraphs of the form of squares, triangles, and double edges.

In dimension $3$, Grigor'yan~\cite{Grigoryan2022} described a basis of $\Omega_3(G;R)$ in terms of the images of the top dimensional generator of a trapezohedron under the constraints that $G$ contains no multisquares or double edges and the coefficients lie in a field. 
More generally, relaxing the condition on double edges, but assuming still field coefficients, Fu and Ivanov~\cite{Fu2024} provided a basis of $\Omega_*(G;R)$ by making use of certain graphical constructions. 

For digraphs containing multisquares, no general description of the path chains in dimensions greater than $2$ has yet appeared in the literature. Moreover, it remains unsatisfactory that one must restrict to field coefficients to apply the results that do exist. The goal of the present work is to describe a method of constructing systems of generators for the higher dimensional path chain modules which are available for more general coefficient rings and do not place restrictions on the digraph's structure.

One import motivation for such a construction comes from computational applications, where digraphs containing multisquares commonly occur. Here there has been much interest in efficiently computing path homology, with persistent path homology being known to be stable with respect to perturbations in edge weighted filtrations of the digraph \cite{Chowdhury2018}. 
Yet, effective computation of path homology remains limited to low dimensions, with one of the most effective algorithms devised \cite{Dey2022} being only for the computation of path homology in dimension $1$ and depending primarily on the knowledge of low dimensional bases of the path chain complex.

More generally, a simple procedure for the computation of path homology is provided by Grigor'yan~\cite[\S 1.7]{Grigoryan2022} as a special case of persistent path homology~\cite[\S 5]{Chowdhury2018}. However, the practicality of these algorithms is limited by the computation of a chain level base, as it involves the computation of the null space of certain large matrices whose sizes grow rapidly with the number of digraph edges.

Path homology algorithms would be greatly simplified without the need to first indirectly compute a basis of the paths chains, as pointed out in~\cite[Problem 1.7]{Grigoryan2022}.
In particular, improved computational speed of path homology in higher dimensions would greatly enhance the practicality of a wide range of applications.

The description of path homology chains introduced in the present work makes no assumptions on the structure of the digraph $G$. Our primary construction is of structures we call face multigraphs (Definition~\ref{def:FaceMultiGraph}), which are principally considered together with their complete extensions by a vertex (Definition~\ref{def:extoffmhg}). We use face multigraphs to construct $(n+1)$-dimensional path chains from given path chains in dimensions $n$ and $n-1$.

Restricting further to strongly connected extensions, we establish the notion of inductive elements of $\Omega_*(G;R)$ (Definition~\ref{def:InductiveElement}), built inductively from the vertex basis of $\Omega_0(G;R)$ as strongly connected extensions over face multigraphs. Inductive elements provide the structure for the following central result concerning generating sets of $\Omega_*(G;R)$.
\begin{theorem}[Corollaries~\ref{cor:FieldGeneratorBasis}~and~\ref{cor:IductiveBasisIntegral}]\label{thm:gen}
The $n$-dimensional inductive elements generate $\Omega_n(G;R)$ when $R$ has characteristic $2$ and $n\geq 0$, or when $R$ has characteristic $0$ and $n\leq 3$.
\end{theorem}
In particular, Theorem~\ref{thm:gen} offers a complete description of $\Omega_*(G;\mathbb{Z}_2)$, and a description of $\Omega_n(G;\mathbb{Z})$ for $n\leq3$. In dimensions $0,1,2$ the inductive elements we construct coincide with the natural generators up to sign. However, the generators in dimension $3$ are new, and we have more to say about them below.

Of course, when working over a field, any generating set can be reduced to a basis. Furthermore, we show (Lemma~\ref{lem:CoeffChangeGraded})
that a basis with $\mathbb{Z}_2$ coefficients can be be transferred to the path chains with coefficients in any ring of characteristic $2$. Thus we have the following statement.
\begin{corollary}
For any $n\geq 0$ and ring $R$ of characteristic $2$, the module $\Omega_n(G;R)$ admits a basis of inductive elements. For any field $K$ of characteristic $0$, the module $\Omega_3(G;K)$ admits a basis of inductive elements.
\end{corollary}

Section~\ref{sec:InductiveBasis} completes the central theoretical constructions presented in this work. The remainder of the paper, Sections~\ref{sec:relwithsn(G)classes},~\ref{sec:Dim3},~and~\ref{sec:ImportantExamples}, are focused on application and can be read largely independently from this point.

In Section~\ref{sec:relwithsn(G)classes}, we explore the relationship between inductive elements and thick $S_n(G)$-classes, the objects introduced by Fu and Ivanov~\cite{Fu2024} to describe bases of $\Omega_*(G;K)$ when $G$ contains no multisquares and $K$ is a field. Up to this point, the theory of thick $S_n(G)$-classes has been the most general characterisation of path homology chains to be found in the literature.

We show that if $R$ is a field and the digraph $G$ contains no multisquares, then the inductive elements defined in this paper reduce to the basis elements provided by thick $S_n(G)$-classes. Combined with our results concerning changes of coefficients in Section~\ref{sec:Coefficients}, our work extends Fu and Ivanov's result to a wider class of coefficients.
 \begin{corollary}[Theorem~\ref{thm:InductiveShortCorrespondence}, Corollary~\ref{cor:InductiveShortCorrespondenceForRings}, and \cite{Fu2024} Theorem 4.7]\label{cor:introoddprime}
If $G$ contains no multisquares, then for any commutative ring $R$ of zero or prime characteristic and any integer $n\geq0$ the module $\Omega_n(G;R)$ admits a basis of inductive elements. 
\end{corollary}
We stress that the bases appearing Corollary~\ref{cor:introoddprime} are essentially those constructed in \cite[Theorem 4.7]{Fu2024}. Our only claim to originality here is the increased generality which is afforded by the more powerful description of these bases in terms of inductive elements.

Turning to our results in Section~\ref{sec:Dim3}, the more precise basis descriptions of $\Omega_3(G;K)$ given by Grigor'yan are expressed in terms of induced images of trapezohedron elements (equation~\eqref{eq:TrapezohedronElement}), the distinguished generator of $\Omega_3(\mathbb{T}_m;R)$ for some member $\mathbb{T}_m$ of the sequence of trapezohedron digraphs (Definition~\ref{def:Trapezohedron}).
We use inductive elements to study the structure of $\Omega_3(G;R)$ in Section~\ref{sec:Dim3}, yielding the following description.
\begin{theorem}[Corollary~\ref{cor:Omega3Explicitly}]
    Let $R$ be a ring of characteristic $0$ or $2$. Then there is a generating set of $\Omega_3(G;R)$ consisting of elements obtained as the induced image of a trapezohedron element under a digraph map.
\end{theorem}
This generalises the earlier result of Grigor'yan by removing the restrictions on $G$, addressing~\cite[Problem 2.11]{Grigoryan2022} with respect to the coefficients above. In fact, we give a much more powerful description of $\Omega_3(G;R)$ in Theorem~\ref{thm:Omega3Explicitly}, of which Corollary~\ref{cor:Omega3Explicitly} is a consequence. To obtain this we are led to consider several infinite families of digraphs, whose importance seems to have not been previously noticed in the literature. 

Finally, in Section~\ref{sec:ImportantExamples}, we use the theory of extensions to construct certain digraphs
demonstrating the necessity of the conditions required in the definition of complete and strongly connected extensions over face multigraphs. In particular, we provide a sequence of digraphs whose path homology differential with respect to an inductive basis contains entries of arbitrary multiplicity.

As an application for the theory of inductive elements we develop several universal coefficient theorems for path homology. In Section~\ref{sec:Coefficients} we produce change of coefficient maps $\bar{\mu}_*\colon H^P_*(G;\mathbb{Z})\otimes R\rightarrow H^P_*(G;R)$ for any digraph $G$ and ring $R$. These maps are quite badly behaved in general, but under certain conditions we are able to show that $\bar{\mu}_n$ is injective and compute its cokernel. 
\begin{theorem}[Propositions~\ref{Pr:coefisopid} and~\ref{prop:univcoeffs} and Theorem~\ref{th:Dim3Char0Char2invareince}]
Let $G$ be a digraph and $R$ a commutative ring and consider a three-term sequence of the form
\begin{equation}\label{eq:univcoefexactseq}
0\rightarrow H_n^P(G;\mathbb{Z})\otimes R\xrightarrow{\bar{\mu}_n} H^P_n(G;R)\rightarrow \mathrm{Tor}(H_{n-1}^P(G;\mathbb{Z}),R)\rightarrow0.
\end{equation}
The following statements hold.
\begin{enumerate}
    \item If $G$ contains no multisquares and $R$ has odd prime characteristic, then there is an exact sequence of the form~\eqref{eq:univcoefexactseq} for each $n\geq0$.
    \item  If $R$ has characteristic $2$, then there is an exact sequence of the form~\eqref{eq:univcoefexactseq} for each $n=0,1,2$. For $n\geq3$, no sequence of the form~\ref{eq:univcoefexactseq} can be exact in general.
    \item If $R$ has characteristic $0$, then $\bar{\mu}_*$ induces isomorphisms $H^P_n(G;R)\cong H_n^P(G;\mathbb{Z})\otimes R$ for all $n\geq0$.
\end{enumerate}
\end{theorem}
Examples appearing in the literature~\cite{Fu2024} show that part $(1)$ cannot be generalised to include rings of characteristic $2$.
In part $(2)$, the boundary case $n=3$ is particularly interesting, and in Theorem~\ref{th:Dim3Char0Char2invareince} we are able to exactly explain the failure of~\eqref{eq:univcoefexactseq} to be exact in terms of a certain obstruction group $\mathcal{B}_3(G;R)$ which is associated to $G$.
The material developed in Section~\ref{sec:Coefficients} shows that part $(2)$ has an analogue when $R$ is a ring of odd characteristic and $n=0,1$, but we leave open the question as to whether the statement extends also to $n=2$ in this case.

The basic constructions in this paper can be made to work for more general rings and fields of odd primary characteristic. The price to pay for the increased generality, however, is a steep increase in complexity. Because of this, these constructions will be dealt with in a separate paper.

Even without the extra complexity, the theory developed in the current paper is already rich enough to serve as foundation for many applications. In a paper under preparation we use the description of $\Omega_3(G;\mathbb{Z}_2)$ given in Section~\ref{sec:Dim3} to obtain a new, efficient algorithm for persistent path homology up to dimension $2$, extending the existing state of the art algorithm for dimension one.

\begin{remark}[Open questions]
Following on from the results presented in this paper, some important questions concerning the structure of path chains remain unresolved.
\begin{enumerate}
    \item 
    Do inductive elements generate $\Omega_n(G;\mathbb{Z})$ for all $n \geq 4$?
    \item 
    Do inductive elements generate $\Omega_3(G;R)$ for any commutative ring $R$?
    \item For any digraph $G$, is there an exact sequence of the form $0\rightarrow H_3^P(G;\mathbb{Z})\otimes R \rightarrow H_3^P(G;R)\rightarrow \mathrm{Tor}(H_{2}^P(G;\mathbb{Z}), R)\rightarrow 0$ when $R$ is a ring of odd prime characteristic?
\end{enumerate}
\end{remark}

\begin{remark}[Acknowledgements]
    The authors would like to thank Alexander Grigor'yan, Sergei Ivanov, and Mengmeng Zhang for comments on an earlier draft of this paper.
\end{remark}

\section{Background}

Throughout this work, $\mathbb{Z}$ is the ring of integers, $\mathbb{Z}_p$ the finite field with $p$ elements for some prime $p$, $\mathbb{R}$ the field of real numbers, and $R$ denotes an arbitrary commutative ring with a unit, unless otherwise stated.

In the sequel we will frequently make use of certain modules being free. Recall that a \emph{basis} of a free $R$-module $F$ is a family $\{f_i\mid i\in I\}$ of distinct elements $f_i \in F$, where $i$ runs over the members of some indexing set $I$, such that every element $f \in F$ can be uniquely expressed as a \emph{linear combination}
\[
    f = \sum_{i\in I} \alpha_i f_i
\]
where each \emph{coefficient} $\alpha_i \in R$ and at most finitely many $\alpha_i \neq 0$. Every free $R$-module has a basis.

\subsection{Directed graphs and labelled multigraphs}

A \emph{digraph} $G$ is a pair $(V_G,E_G)$, whose \emph{vertices} are elements of the set $V_G$ and whose \emph{edges} are element of the set
\[
    E_G \subseteq \{ (u,v) \in V_G \times V_G \: | \: u \neq v \}.
\]
We often denote edges $(u,v) \in E_G$ by $u \to v$.
A digraph $G$ is called \emph{finite} if the set $V_G$ is finite. However, unless otherwise stated, we place no finiteness assumptions on the digraphs in this work. 

A \emph{subdigraph} of a digraph $G$ is a digraph $H$ such that $V_H\subseteq V_G$ and $E_H\subseteq E_G$.
A \emph{map} of digraphs $f \colon G \to G'$ is a function $f\colon V_G \to V_{G'}$ such that for all $(u,v) \in E_G$, either $f(u) = f(v)$ or $(f(u),f(v)) \in E_{G'}$. If $(f(u),f(v)) \in E_{G'}$ for all $(u,v)\in E_G$, then $f$ is said to be a \emph{strong map}.

A \emph{multigraph} $M$ consists of a pair $(V_M,E_M)$, whose \emph{vertices} are elements of the set $V_M$ and whose \emph{edges} are elements of the multiset $E_M$, the members of which are sets of pairs of distinct elements of $V_M$. A \emph{submultigraph} of a multigraph $M$ is a multigraph $N$ such that
\[
    V_N \subseteq V_M 
    \;\;\; \text{and} \;\;\;
    E_N \subseteq E_M.
\]
A multigraph $M$ is called a \emph{graph} if each member of $E_M$ has multiplicity $1$.

A \emph{vertex labelled multigraph} with label set $L_V$ is a multigraph $M$ together with a function $l_V \colon V_M \to L_V$. Similarly, an \emph{edge labelled multigraph} with label set $L_E$ is a multigraph $M$ together with a function $l_E\colon E_M \to L_E$. A multigraph that is both vertex labelled and edge labelled is called a \emph{labelled multigraph}.

A multigraph is said to be \emph{connected} if for any distinct vertices $u,w \in V_M$, there exists a positive integer $t$ and a sequence of vertices $u=v_1,\dots,v_t=v$ such that $(v_i,v_{i+1})\in E_M$ for each $i=1,\dots,t-1$. The same terminology applies to labelled multigraphs, so that one such is connected in case the underlying multigraph is.

\subsection{Path homology}

We detail here the construction of path homology as laid out in the foundational paper \cite{Grigoryan2013}. In the next subsection we provide an alternative characterisation of the path chain modules in terms of the diagonal magnitude homology groups. Throughout this section, let $G$ be a digraph and $R$ a commutative ring with a unit.

An \emph{elementary $n$-path} ($n\geq0$) in a set $V$ is a sequence $v_0,\dots,v_n \in V$, which we denote $e_{v_0,\dots,v_n}$. Define $\Lambda_n(V;R)$ to be the free $R$-module generated by all elementary $n$-paths in $V$.
In addition, set $\Lambda_{n}(V;R) = 0$ for $n<0$.
Then every $x\in \Lambda_n(V;R)$ has a unique expression
\begin{equation}\label{eq:UniqueChainFormInitial}
    x = \sum_{v_0,\dots,v_n \in V} \alpha_{v_0,\dots,v_n} e_{v_0,\dots,v_n}
\end{equation}
where each $\alpha_{v_0,\dots,v_n} \in R$ and at most finitely many $\alpha_{v_0,\dots,v_n}\neq 0$.

Define maps $\partial_{n,i}^P\colon \Lambda_n(V;R) \to \Lambda_{n-1}(V;R)$ for each $i=0,\dots,n$ by linearly extending
\begin{equation*}
    \partial_{n,i}^P(e_{v_0,\dots,v_n}) = e_{v_0,\dots,\hat{v}_i,\dots,v_n}
\end{equation*}
where $v_0,\dots,\hat{v}_i,\dots,v_n$ denotes the sequence $v_0,\dots,v_n$ with the element $v_i$ removed.
Then the graded module $\Lambda_*(V;R)$ becomes a chain complex $(\Lambda_*(V;R),\partial^P_*)$ with differential
\[
    \partial^P_n = \sum_{i=0}^n (-1)^i\partial^P_{n,i}.
\]
We call $\partial^P_n$ the \emph{path differential}.
Clearly, the chain complex $(\Lambda_*(V;R),\partial^P_*)$ has trivial homology in all but degree $0$.

An elementary $n$-path $e_{v_0,\dots,v_n}\in \Lambda_n(V;R)$ is called \emph{regular} if $v_{i-1} \neq v_i$ for every $i=1,\dots,n$, and \emph{irregular} otherwise. Denote by $\mathcal{I}_n(V;R)$ the free $R$-module generated by the set of elementary irregular $n$-paths and define a graded module $\mathcal{R}_*(V;R)$ by
\[
    \mathcal{R}_n(V;R) = \Lambda_n(V;R) / \mathcal{I}_n(V;R).
\]
The path differential $\partial^P_*$ descends to the quotient and $(\mathcal{R}_*(V;R), \partial^P_*)$ becomes a chain complex. Still, $(\mathcal{R}_*(V;R), \partial^P_*)$ like $(\Lambda_*(V;R), \partial^P_*)$ has trivial homology in all but degree $0$.

An \emph{(allowed) $n$-path} in a digraph $G=(V_G,E_G)$ is an elementary $n$-path $e_{v_0,\dots,v_n}$ in $V_G$ such that
\[
    (v_{i-1},v_i) \in E_G
\]
for each $i = 1, \dots, n$. 
The allowed paths span a submodule of $\Lambda_*(V_G;R)$ which is mapped injectively into $\mathcal{R}_*(V_G;R)$ by the quotient projection. Let $\mathcal{A}_*(G;R)\subseteq \mathcal{R}_*(V_G;R)$ be the image of this submodule. Following a standard abuse, we call also the members of $\mathcal{A}_*(G;R)$ \emph{allowed paths}, and will denote the cosets in $\mathcal{A}_*(G;R)$ by their unique allowed representatives. A fact that we use repeatedly in the sequel is that $\mathcal{A}_*(G;R)$ is a free $R$-module.

On the other hand, the graded module $\mathcal{A}_*(G;R)$ is not generally a subcomplex of $(\mathcal{R}_*(V_G;R),\partial^P_*)$, as it need not be the case that $\partial^P_n(A_n(G;R))\subseteq A_{n-1}(G;R)$. 
Therefore, we pass to the submodule $\Omega_*(G;R)$ of $\mathcal{A}_*(G;R)$ defined to be
\[
    \Omega_n(G;R) = \{ x \in \mathcal{A}_n(G;R) \: | \: \partial^P_{n}(x) \in \mathcal{A}_{n-1}(G;R) \}.
\]
Note that, similar to the expression in equation~\eqref{eq:UniqueChainFormInitial}, it remains the case that each $x\in \Omega_n(G;R)$ can be written uniquely in the form
\begin{equation}\label{eq:UniqueChainForm}
    x = \sum_{e_{v_0,\dots,v_n} \in P^G_n} \alpha_{v_0,\dots,v_n} e_{v_0,\dots,v_n}
\end{equation}
where $P^G_n$ is the set of all $n$-paths of $G$ and $\alpha_{v_0,\dots,v_n} \in R$ with at most finitely may $\alpha_{v_0,\dots,v_n}\neq 0$.

By construction, $\Omega_*(G;R)$ is the smallest submodule of $\mathcal{A}_{n}(G;R)$ on which $\partial^P_n$ is a differential. Following the usual convention, we assume that $\Omega_n(G;R)=0$ for $n < 0$.
\begin{definition}
    We call the chain complex $(\Omega_*(G;R), \partial^P_*)$ the \emph{path chain complex} of $G$ with coefficients in $R$.
    An element of $\Omega_n(G;R)$ is called a \emph{path chain} of dimension $n$.
    The homology $H^P_*(G;R)$ of $(\Omega_*(G;R), \partial^P_*)$ is called the \emph{path homology} of $G$ with coefficients in $R$.
\end{definition}

Path homology is functorial with respect to digraph maps. Any function $f\colon V\rightarrow V'$ induces a linear map $f_{n}\colon \Lambda_n(V;R) \to \Lambda_n(V';R)$ which acts on basis elements as
\[
    f_n(e_{v_0,\dots,v_n}) = e_{f(v_0),\dots,f(v_n).}
\]
This preserves irregular paths and hence descends to $f_n\colon\mathcal{R}_n(V;R)\rightarrow\mathcal{R}_n(V';R)$. In case $f \colon G \to G'$ is a map of digraphs, $f_n$ preserves allowed paths and restricts to $f_n\colon\mathcal{A}_n(G;R)\rightarrow\mathcal{A}_n(G';R)$. It's clear that all these maps commute with the path differentials and hence give rise to a chain map $f_{\#}\colon \Omega_*(G;R) \to \Omega_*(G';R)$. In turn, there is an induced graded linear map $f_*\colon H_*^P(G;R) \to H_*^P(G';R)$ (for a proof see \cite[Proposition 1.6]{Grigoryan2022}).

\subsection{Magnitude homology and path homology}\label{sec:MagnitudeHomology}

Magnitude homology was introduced by Hepworth and Willerton~\cite{Hepworth2017} as a homology theory for graphs and later generalised to quasi-metric spaces, as presented here. The relationship between the magnitude and path homologies of a digraph was first discussed by Asao~\cite{Asao2023}, making use of a spectral sequence identified by Hepworth and Willerton. In this spectral sequence, the first page coincides with the magnitude homology and the second page contains the path homology along a row. In this section we make explicit the relationship between the path chains and magnitude chains. 

A  \emph {(extended) quasi-metric space} $(X,d)$ is a set $X$ together with a function
\[
    d \colon X \times X \to [0,\infty]
\]
such that
\begin{enumerate}[(1)]
\item $d(x_1,x_1)=0$,
\item $d(x_1,x_3)\leq d(x_1,x_2)+d(x_2,x_3)$ and
\item $d(x_1,x_2)=d(x_2,x_1)=0$ implies that $x_1=x_2$
\end{enumerate}
for all $x_1,x_2,x_3 \in X$.

Thus let $(X,d)$ be a quasi-metric space. We denote by $\langle x_0,\dots,x_n\rangle$ an $(n+1)$-tuple $(x_0,\dots,x_n) \in X^{n+1}$ for which $x_{i-1}\neq x_i$ for all $i=1,\dots,n$.
We write
\begin{equation}\label{eq:length}
    \ell\langle x_0,\dots,x_n\rangle=\sum^{n}_{i=1}d(x_{i-1},x_i),
\end{equation}
and call this quantity the \emph{length} of $\langle x_0,\dots,x_n\rangle$.
For fixed $l\in\mathbb{R}$ define free $R$-modules $\text{C}^M_{n,l}(X;R)$ by
\begin{equation*}
    C^M_{n,l}(X;R) =
    R
    \left[
    \left\{
    \langle x_0, \dots, x_n \rangle
    \mid
    \ell\langle x_0,\dots,x_n\rangle = l
    \right\}
    \right]
\end{equation*}
for $n\geq0$, and $C^M_{n,l}(X;R) = 0$ for $n<0$.
For $i=1,\dots,n-1$ define $\partial_{i,n}^M\colon C^M_{n,l}(X;R) \to C^M_{n-1,l}(X;R)$ by linearly extending
\begin{equation*}
    \partial_{n,l,i}^M\langle x_0,\dots,x_n\rangle =
    \begin{cases} 
        \langle x_0,\dots,x_{i-1},x_{i+1},\dots, x_n \rangle
        &
        \text{if}\;\; d(x_{i-1},x_i) + d (x_i,x_{i+1}) = d(x_{i-1}, x_{i+1}), \\
        0
        &
        \text{otherwise}
    \end{cases}
\end{equation*}
and setting
\begin{equation*}
    \partial_{n,l}^M = \sum_{i=1}^{n-1} (-1)^i \partial_{n,l,i}^M.
\end{equation*}
Then $(C^M_{*,l}(X;R),\partial^M_{*,l})$ is a chain complex for each $l \in \mathbb{R}$.

\begin{definition}
    For $l\in \mathbb{R}$ define the \emph{magnitude homology} $H^M_{*,l}(X;R)$ to be the homology of the chain complex $(C^M_{*,l}(X;R), \partial_{*,l}^M)$.
\end{definition}

A digraph $G$ comes furnished with a natural quasi-metric $d_G$ given by
\begin{equation}\label{eq:GraphMetric}
d_G(u,v)=\min\{n\geq0\mid \exists \: e_{u=v_0,\dots,v_n=v}\in P^G_n\},
\end{equation}
where we understand that $d_G(u,v)=\infty$ if there is no allowed path in $G$ from $u$ to $v$. We will generally suppress this quasi-metric from notation, writing $C^M_{n,l}(G;R)$ and $H^M_{n,l}(G;R)$ for the magnitude chains and magnitude homology of the quasi-metric space $(G,d_G)$. 
Note that in this case we need only consider integer values of $l$ to obtain all information about the magnitude homology.

The magnitude homology $H^M_{n,n}(G;R)$ for $n\geq 0$ is called the \emph{diagonal magnitude homology} of the digraph $G$. As $C^M_{n,l}(G;R)=0$ for $l<n$, it holds that
\begin{equation}\label{eq:DiagonalKernal}
    H^M_{n,n}(G;R) =\ker(\partial^M_{n,n}).
\end{equation}
A digraph map $f\colon G\rightarrow G'$ induces for each $n>0$ a linear map $f_n\colon C^M_{n,n}(G;R)\rightarrow C^M_{n,n}(G';R)$, which sends a basis element $\langle v_0,\dots,v_n\rangle$ to $\langle f(v_0),\dots,f(v_n)\rangle$ if $f(v_i)\neq f(v_{i+1})$ for all $i$, and to $0$ otherwise. Since these maps commute with the magnitude differential, there are induced maps $f_*\colon H^M_{n,n}(G;R)\rightarrow H^M_{n,n}(G';R)$.
\begin{lemma}[\cite{Asao2023} Lemma 6.8]\label{lem:PathChianKernel}
    The map $\phi_n \colon \mathcal{A}_n(G;R)\to C^M_{n,n}(G;R)$ given by linearly extending
    \[
        \phi_n(e_{v_0,\dots,v_n}) = \langle v_0,\dots,v_n \rangle
    \]
    is a natural isomorphism of $R$-modules and induces a natural isomorphism
    \[
        \phi_n \colon \Omega_n(G;R)\to H^M_{n,n}(G;R).
    \]
\end{lemma}

\begin{proof}
    When $n=l$, condition~\eqref{eq:length} together with the structure of $G$ as a quasi-metric space implies that for any $\langle v_0,\dots,v_n \rangle \in C^M_{n,n}(G;R)$ we have $d(v_{i-1},v_i)=1$, and hence $(v_{i-1},v_{i}) \in E_G$, for each $i=1,\dots,n$.
    As $\mathcal{A}_n(G;R)$ and $ C^M_{n,n}(G;R)$ are free $R$-modules on generators indexed by precisely the same sequences of vertices, $\phi_n \colon \mathcal{A}_n(G;R)\to C^M_{n,n}(G;R)$
    is a well defined isomorphism. Its naturality is apparent.
    
    Now, for any allowed path $e_{v_0,\dots,v_n}$ we have that
    \begin{align*}
        & \partial^P_{n}(e_{v_0,\dots,v_n})
        =
        e_{v_1,\dots,v_n} + (-1)^n e_{v_0,\dots,v_{n-1}} +
        \sum_{i=1}^{n-1} (-1)^ie_{v_0,\dots,\hat{v}_i,\dots,v_n}
        \\ &=
        e_{v_1,\dots,v_n} + (-1)^n e_{v_0,\dots,v_{n-1}} +
        \sum_{\substack{i=1,\dots,n-1 \\ (v_{i-1},v_{i+1}) \in E_G}}
        (-1)^i e_{v_0,\dots,\hat{v}_i,\dots,v_n}
        +
        \sum_{\substack{i=1,\dots,n-1 \\ (v_{i-1},v_{i+1}) \notin E_G}}
        (-1)^i e_{v_0,\dots,\hat{v}_i,\dots,v_n}
    \end{align*}
    where the last summation contains all non-allowable elementary paths which appear in the expression.
    Moreover, these are the only elementary $(n-1)$-paths whose images under $\phi_{n-1}$ can correspond to non-trivial summands $\langle v_0,\dots,\hat{v}_i,\dots,v_n \rangle$ in the image of $\partial_{n,n,j}^M(\phi_n(e_{v_1,\dots,v_n}))$ for some $j=1,\dots,n-1$. 
    By linearly extending the above calculation, we obtain that any $x\in \mathcal{A}_n(G;R)$ satisfies $\partial_n^P(x) \in \Omega_{n-1}(G;R)$ only when $\partial_{n,n}^M( \phi_n (x)) = 0$.
    \qed
\end{proof}

In the remainder of this work, making use of the isomorphism~\eqref{eq:DiagonalKernal}, we treat the module of path chains $\Omega_n(G;R)$ as the kernel of $\partial^M_{n,n}$ under the identification provided by the map $\phi_n$ from Lemma~\ref{lem:PathChianKernel}.

\subsection{Basis constructions for \texorpdfstring{$\Omega_n(G;R)$}{path chains}}\label{sec:LowDimBasis}

Most of the material covered in this section is contained in \cite{Grigoryan2022} and \cite{Fu2024}.
However, some of the content is from, or was originally presented in, other works that we cite at the corresponding parts of the section.
Throughout this section, let $G$ be a digraph and $R$ a commutative ring with a unit.

For any vertex $v\in V_G$, $e_v$ is an allowed path and $\partial^P_0(e_v) = 0$.
Hence,
\begin{equation}\label{eq:Dim0Basis}
    \{ e_v \: | \: v \in V_G \} \; \text{is a basis of} \:\: \Omega_0(G;R).
\end{equation}
Similarly, for any $(u,v)\in E_G$, the element $e_{u,v}$ is an allowed path and $\partial^P_1(e_{u,v}) = e_u - e_v \in \mathcal{A}_0(G;R)$.
Therefore, 
\begin{equation}\label{eq:Dim1Basis}
    \{ e_{u,v} \: | \: (u,v)\in E_G  \} \; \text{is a basis of} \:\: \Omega_1(G;R).
\end{equation}

The first non-straightforward case occurs when $n=2$.
Let $v_0,v_1,v'_1,v_2 \in V_G$.
When $(v_0,v_1) \in E_G$ and $(v_1,v_0) \in E_G$ we call $e_{v_0,v_1,v_0}$ a \emph{double edge}.
We call $e_{v_0,v_1,v_2}$ a \emph{directed triangle} if $(v_0,v_1),(v_1,v_2),(v_0,v_2) \in E_G$.
Finally, we call
$e_{v_0,v_1,v_2}-e_{v_0,v'_1,v_2}$ a \emph{directed square} if $(v_0,v_1),(v_0,v'_1),(v_1,v_2),(v'_1,v_2) \in E_G$, $v_0 \neq v_2$, $v_1 \neq v'_1$, and $(v_0,v_2) \notin E_G$.
It is straightforward to check that double edges, directed squares, and directed triangles are elements of $\Omega_2(G;R)$.
\begin{center}
    \tikz {
        \node (b) at (0,0) {$\:$};
        \node (u) at (0,1) {$v_0$};
        \node (v) at (1.5,1) {$v_1$};
        \draw[->] (u) to [out=40,in=140] (v);
        \draw[->] (v) to [out=220,in=320] (u);
        }
        \;\;\;\;\;\;\;\;\;\;\;\;\;\;\;
        \tikz {
        \node (b) at (0,0) {$\:$};
        \node (v0) at (0,0.25) {$v_0$};
        \node (v1) at (1.5,0.25) {$v_1$};
        \node (v2) at (1.5,1.75) {$v_2$};
        \draw[->] (v0) -- (v1);
        \draw[->] (v1) -- (v2);
        \draw[->] (v0) -- (v2);
        }
        \;\;\;\;\;\;\;\;\;\;\;\;\;\;\;
        \tikz {
        \node (v0) at (0,0) {$v_0$};
        \node (v1) at (-1,1) {$v_1$};
        \node (v-1) at (1,1) {$v'_1$};
        \node (v2) at (0,2) {$v_2$};
        \draw[->] (v0) -- (v1);
        \draw[->] (v0) -- (v-1);
        \draw[->] (v1) -- (v2);
        \draw[->] (v-1) -- (v2);
        }
\end{center}

Analogously, a digraph $G$ is said to contain a \emph{double edge}, a \emph{directed triangle}, or a \emph{directed square} if it contains one of the above respective digraphs as a subdigraph. These are not required to be full subdigraphs, but in the case of a directed square it is important that $(v_0,v_2)\not\in E_G$. For instance, of the two digraphs below, 
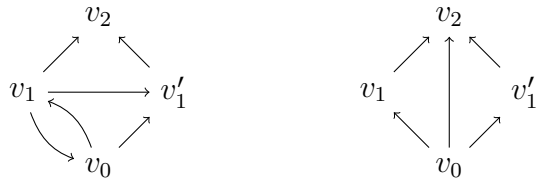
\begin{center}
        \tikz {
        \node (v0) at (0,0) {$v_0$};
        \node (v1) at (-1,1) {$v_1$};
        \node (v-1) at (1,1) {$v'_1$};
        \node (v2) at (0,2) {$v_2$};
        \draw[->] (v0) to [out=110,in=-20] (v1);
        \draw[->] (v1) to [out=-70,in=160] (v0);
        \draw[->] (v1) -- (v-1);
        \draw[->] (v0) -- (v-1);
        \draw[->] (v1) -- (v2);
        \draw[->] (v-1) -- (v2);
        }
        \;\;\;\;\;\;\;\;\;\;\;\;\;\;\;
        \tikz {
        \node (v0) at (0,0) {$v_0$};
        \node (v1) at (-1,1) {$v_1$};
        \node (v-1) at (1,1) {$v'_1$};
        \node (v2) at (0,2) {$v_2$};
        \draw[->] (v0) -- (v1);
        \draw[->] (v0) -- (v-1);
        \draw[->] (v0) -- (v2);
        \draw[->] (v1) -- (v2);
        \draw[->] (v-1) -- (v2);
        }
\end{center}
the left-hand digraph contains one double edge, three directed triangles, and one directed square, while the right-hand digraph contains two directed triangles, no double edges, and no directed squares.

\begin{definition}
    Let $v_0,v_2\in V$ such that $v_0 \neq v_2$, $(v_0,v_2)\notin G$ and,
    \[
        S_{v_0,v_2} = \{ v_1 \in V \: | \: (v_0,v_1), (v_1,v_2) \in E_G \}.
    \]
    When $|S_{v_0,v_2}| \geq 3$, we say that $G$ contains a \emph{multisquare} between $v_0$ and $v_2$. If $G$ contains a multisquare between $v_0$ and $v_2$, then any directed square of the form $e_{v_0,v_1,v_2}-e_{v_0,v'_1,v_2}$ is said to lie \emph{within} the multisquare. If $G$ does not contain any multisquares between any pair of vertices, then we say $G$ \emph{contains no multisquares}.
\end{definition}

Variations of the following proposition have been proved for both field and integral coefficients in \cite[Proposition 4.2]{Grigoryan2013}, \cite[Proposition 2.9]{Grigor'yan2014b}, and \cite[Theorem 1.8]{Grigoryan2022}. We include a proof at the end of Section~\ref{sec:BigradingAndFaceMaps}.

\begin{proposition}[\cite{Grigor'yan2014b, Grigoryan2013, Grigoryan2022}]\label{prop:Dim2Base}
    The double edges, directed triangles, and directed squares generate $\Omega_2(G;R)$.
    Once a basis of directed squares within each multisquare is chosen, these directed squares, double edges, directed triangles, and a choice up to sign of directed squares not contained within a multisquare form a basis of $\Omega_2(G;R)$.
\end{proposition}
 
The existence of multisquares implies that no canonical basis for $\Omega_2(G;R)$ exists in general. The following example demonstrates that this is also the case for $\Omega_n(G;R)$ when $n\geq 2$.
\begin{exmp}\label{ex:ExtendedMultisquare}
    For $t\geq 3$ consider the digraph $G$ with $V_G= \{ v_0,v_1^1,v_1^2,v_1^3,v_2,v_3,\dots,v_{t-1},v_t \}$ and edges
    \begin{align*}
        E_G = & \{  (v_0,v^1_1), (v^1_1,v_2), (v^1_1,v_3), (v_0,v^2_1), (v^2_1,v_2), (v^2_1,v_3), (v_0,v^3_1), (v^3_1,v_2), (v^3_1,v_3) \}
        \\
        & \cup
        \{ (v_{i},v_{i+1}) \: | \: i=2,\dots,t-1 \}
        \cup
        \{ (v_{i},v_{i+2}) \: | \: i=2,\dots,t-2 \}.
    \end{align*}
    \begin{center}
        \tikz {
            \node (v0) at (0,0) {$v_0$};
            \node (v1) at (1,1) {$v^1_1$};
            \node (v12) at (1,-1) {$v^2_1$};
            \node (v13) at (1,2) {$v^3_1$};
            \node (v2) at (2,0) {$v_2$};
            \node (v3) at (3.414,0) {$v_3$};
            \node (v4) at (4.414,1) {$v_4$};
            \node (v5) at (5.414,0) {$v_5$};
            \node (v6) at (6.414,1) {$v_6$};
            \node (vt-2) at (9.414,0) {$v_{t-2}$};
            \node (vt-1) at (10.414,1) {$v_{t-1}$};
            \node (vt) at (11.414,0) {$v_t$};
            \node (d) at (8.414,0.9) {$\cdots$};
            \node (d) at (7.414,0.1) {$\cdots$};
            \node (B6) at (7.414,1) {};
            \node (B5) at (6.414,0) {};
            \node (B-6) at (7.014,0.4) {};
            \node (Bt-1) at (9.414,1) {};
            \node (Bt-2) at (8.414,0) {};
            \node (B-t-2) at (8.814,0.6) {};
            \draw[->] (v0) -- (v1);
            \draw[->] (v0) -- (v12);
            \draw[->] (v0) -- (v13);
            \draw[->] (v1) -- (v2);
            \draw[->] (v12) -- (v2);
            \draw[->] (v13) -- (v2);
            \draw[->] (v2) -- (v3);
            \draw[->] (v3) -- (v4);
            \draw[->] (v4) -- (v5);
            \draw[->] (v5) -- (v6);
            \draw[->] (vt-2) -- (vt-1);
            \draw[->] (vt-1) -- (vt);
            \draw[->,gray] (v1) -- (v3);
            \draw[->,gray] (v12) -- (v3);
            \draw[->,gray] (v13) -- (v3);
            \draw[->,gray] (v2) -- (v4);
            \draw[->,gray] (v3) -- (v5);
            \draw[->,gray] (v4) -- (v6);
            \draw[->,gray] (vt-2) -- (vt);
            \draw[-,gray] (v6) -- (B6);
            \draw[-] (v6) -- (B-6);
            \draw[-,gray] (v5) -- (B5);
            \draw[->,gray] (Bt-1) -- (vt-1);
            \draw[->] (B-t-2) -- (vt-2);
            \draw[->,gray] (Bt-2) -- (vt-2);
        }
    \end{center}
    The digraph $G$ contains a multisquare between $v_0$ and $v_2$. The $t$-chains
    \[
        e_{v_0,v_1^j,v_2,v_3\dots,v_{t-1},v_t}
        -
        e_{v_0,v_1^{j'},v_2,v_3\dots,v_{t-1},v_t}
    \]
    for $j,j'=1,2,3$, and $j<j'$, generate $\Omega_t(G;R)$, and any two of them form a basis of $\Omega_t(G;R)$.
\end{exmp}

To end the section, we summarise the existing partial results for bases of $\Omega_n(G;R)$ when $n\geq 3$.

\begin{definition}\label{def:Trapezohedron}
    For $t\geq2$ let $\mathbb{T}_{t}$ be the digraph with vertices
    \[
        V_{\mathbb{T}_{t}} =
        \{ T,u_1,\dots,u_t,v_1,\dots,v_t,H \}
    \]
    and edges
    \[
        T \to u_i,\;\;\;
        u_i \to v_i, \;\;\;
        u_i \to v_{i+1},
        \;\;\; \text{and} \;\;\;
        v_i \to H
    \]
    for $i=1,\dots,t$, where all index values are assumed to be modulo $t$.
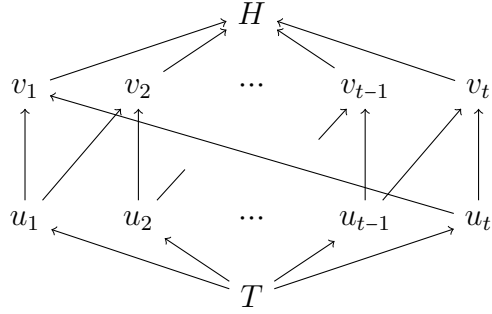
    \begin{center}
        \tikz {
            \node (T) at (2.25,0) {$T$};
            \node (u1) at (-0.75,1) {$u_1$};
            \node (u2) at (0.75,1) {$u_2$};
            \node (v1) at (-0.75,2.75) {$v_1$};
            \node (v2) at (0.75,2.75) {$v_2$};
            \node (du) at (2.25,1) {$\cdots$};
            \node (dv) at (2.25,2.75) {$\cdots$};
            \node (ut-1) at (3.75,1) {$u_{t-1}$};
            \node (vt-1) at (3.75,2.75) {$v_{t-1}$};
            \node (ut) at (5.25,1) {$u_t$};
            \node (vt) at (5.25,2.75) {$v_t$};
            \node (H) at (2.25,3.75) {$H$};
            \node (u2-) at (1.5,1.825) {};
            \node (vt-1-) at (3,1.925) {};
            \draw[->] (T) -- (u1);
            \draw[->] (T) -- (u2);
            \draw[->] (T) -- (ut-1);
            \draw[->] (T) -- (ut);
            \draw[->] (u1) -- (v1);
            \draw[->] (u1) -- (v2);
            \draw[->] (u2) -- (v2);
            \draw[->] (ut-1) -- (vt-1);
            \draw[->] (ut-1) -- (vt);
            \draw[->] (ut) -- (vt);
            \draw[->] (ut) -- (v1);
            \draw[->] (v1) -- (H);
            \draw[->] (v2) -- (H);
            \draw[->] (vt-1) -- (H);
            \draw[->] (vt) -- (H);
            \draw[-] (u2) -- (u2-);
            \draw[->] (vt-1-) -- (vt-1);
        }
    \end{center}
    The digraph $\mathbb{T}_{t}$ is called a \emph{trapezohedron} of order $t$.
\end{definition}

\begin{proposition}[\cite{Grigoryan2022} Proposition 2.1]\label{prop:Trapezohedron}
    The $R$-module $\Omega_3(\mathbb{T}_t;R)$ is free with rank $1$ and $H_n^P(\mathbb{T}_t;R)=0$ for $n\geq 1$.
\end{proposition}
The unique up to sign generator of $\Omega_3(\mathbb{T}_t;R)$ is called a \emph{trapezohedron element} and can be explicitly realised as
\begin{equation}\label{eq:TrapezohedronElement}
    T_t = \sum_{i=1}^t e_{T,u_i,v_i,H} - e_{T,u_i,v_{i+1},H}
\end{equation}
where $v_{t+1}=v_1$.

\begin{theorem}[\cite{Grigoryan2022} Theorem 2.10]\label{thm:Dim3BasisNoDoubleNoMulti}
    Let $K$ be a field and $G$ a digraph containing no double edges and no multisquares. Then there is a basis of $\Omega_3(G;K)$, each of whose members is the image of a trapezohedron element under the homomorphism induced by a digraph map $\mathbb{T}_t \to G$ for some integer $t \geq 2$.
\end{theorem}

In an alternative, more general approach, under only the assumption that $G$ contains no multisquares, Fu and Ivanov~\cite[Theorem 4.7]{Fu2024} give an explicit description of a basis of $\Omega_n(G;K)$ when $K$ is a field. To explain this more precisely, we introduce the following notation from \cite{Fu2024}.

The basic definitions make sense for an arbitrary coefficient ring $R$ and we state them at this level of generality. A path $e_{v_0,\dots,v_n}\in \mathcal{A}(G;R)$ is called \emph{thin} if there is $i\in\{1,\dots,n-1\}$ such that $(v_{i-1},v_{i+1}) \notin E_G$ and $e_{v_{i-1},v_i,v_{i+1}}$ is the only allowed $2$-path beginning at $v_{i-1}$ and ending at $v_{i+1}$. A path which is not thin is called \emph{thick}. Two paths $e_{u_0,\dots,u_n},e_{v_0,\dots,v_n}\in \mathcal{A}(G;R)$ are said to differ by a \emph{short move} at $i\in\{1,\dots,n-1\}$ if
\begin{equation}\label{eq:ShortMove}
    u_i \neq v_i, \;\;\; u_j = v_j \;\;\; \text{for} \;\;\; i \neq j = 0,\dots,n, \;\;\; \text{and} \;\;\; (u_{i-1},u_{i+1}) \notin E_G.
\end{equation}
The \emph{graph of short moves} $\mathcal{S}_n(G)$ is the graph whose vertices are the allowed $n$-paths in $G$ and in which there is an edge between two vertices exactly when there is a short move between the corresponding paths. The edges of $\mathcal{S}_n(G)$ will be labelled by the integer $i\in\{1,\dots,n-1\}$ which is determined by~\eqref{eq:ShortMove} and the vertices by their corresponding allowed paths. The connected components of $\mathcal{S}_n(G)$ are called \emph{$\mathcal{S}_n(G)$-classes}. An $\mathcal{S}_n(G)$-class is \emph{thick} if all its vertices are thick paths.

We say that an $\mathcal{S}_n(G)$-class has a \emph{partition} if its vertices can be separated into two non-empty disjoint subsets such that no pair of elements in either subset is joined by an edge in $\mathcal{S}_n(G)$. By construction, the vertices in an $\mathcal{S}_n(G)$-class are obtained from a connected subgraph of $\mathcal{S}_n(G)$. Thus if an $\mathcal{S}_n(G)$-class has a partition, it must be unique up to the order of the chosen subsets.

\begin{theorem}[\cite{Fu2024} Theorem 4.7]\label{thm:NoMultisquareBasis}
    Let $G$ be a digraph without multisquares and let $K$ be a field. Then we have the following basis descriptions of $\Omega_n(G;K)$.
    \begin{enumerate}
        \item 
            If $K$ has characteristic $2$, then there is a basis of $\Omega_n(G;K)$ where each element is 
            the sum of the paths belonging to a thick $\mathcal{S}_n(G)$-class.
        \item 
            If the characteristic of $K$ is not $2$, then there is a basis of $\Omega_n(G;K)$ 
            where each element uniquely corresponds to a
            thick $\mathcal{S}_n(G)$-class that has a partition. More precisely, up to sign, each element of the basis is the sum of the paths in the $\mathcal{S}_n(G)$-class from the first subset of a partition minus the sum of paths in the $\mathcal{S}_n(G)$-class from the second subset of the partition.
    \end{enumerate}
\end{theorem}
Other than the choice of sign, the bases above are canonically defined. However, the assumption in Theorems~\ref{thm:Dim3BasisNoDoubleNoMulti} and~\ref{thm:NoMultisquareBasis} that $G$ contain no multisquares is very restrictive. Part of the motivation for the present work was to develop techniques for studying digraphs containing multisquares. We do this in the next few sections, revisiting Theorem~\ref{thm:NoMultisquareBasis} in Section~\ref{sec:relwithsn(G)classes}
and offering a new perspective on its content.

\section{Structure morphisms on the path chain complex}\label{sec:StructureMaps}

Throughout this section, let $G$ be a digraph equipped with the quasi-metric $d_G$ given in equation~\eqref{eq:GraphMetric} and $R$ a commutative ring with a unit, unless stated otherwise. In the section we define several morphisms on $\Omega_n(G;R)$ that are fundamental to the constructions presented in the rest of the paper.

\subsection{A further characterisation of \texorpdfstring{$\Omega_n(G;R)$}{path chains}}

We first provide the following lemma, which gives a characterisation of the elements of $\Omega_n(G;R)$, further refining that given in Lemma~\ref{lem:PathChianKernel}. We are not aware of the lemma having been proved previously in precisely the form presented. However, an essentially equivalent result is provided in \cite[Proposition 1.12]{Di2024} and a related statement is made in \cite[Lemma 4.1]{Grigoryan2013}. For completeness, we give a full proof directly, relying only on Lemma~\ref{lem:PathChianKernel}.

\begin{lemma}\label{lem:NoPositionSwap}
    Let $x \in \mathcal{A}_n(G;R)$.
    Then $x\in \Omega_n(G;R)$ if and only if
    \[
        \partial^M_{n,n,i}(x) = 0 
    \]
    for each $i=1,\dots,n-1$, where $\partial^M_{n,n,i}$ is the $i^\text{th}$ component of the magnitude differential.
\end{lemma}

\begin{proof}
    Sufficiency is clear from Lemma~\ref{lem:PathChianKernel}, as by definition $\partial_{n,n}^M = \sum_{i=1}^{n-1}(-1)^{i}\partial_{n,n,i}^M$.
    To show necessity, we adopt the notation for magnitude homology from Section~\ref{sec:MagnitudeHomology}, identifying $\mathcal{A}_n(G;R)$ with $C^M_{n,n}(G;R)$ and $\Omega_n(G;R)$ with $H^M_{n,n}(G;R)$ using the map $\phi_n$ from Lemma~\ref{lem:PathChianKernel}. Thus if $x\in \Omega_n(G;R)$, then, similarly to equation~\eqref{eq:UniqueChainForm}, $x$ can be written uniquely as
    \begin{equation*}
        x = \sum_{j \in J} \alpha_j \langle v_0^j,\dots,v_n^j \rangle
    \end{equation*}
    where $J$ is a finite set, $e_{v_0^j,\dots,v_n^j}$ is an allowed $n$-path, and $\alpha_j \in R$ is nonzero for each $j \in J$. 
    
    We claim that it is enough to show that for any $i=1,\dots,n-1$ and $j \in J$ satisfying
    \begin{equation}\label{eq:MagnitudeCondition}
        (v_{i-1}^j,v_{i+1}^j)\notin E_G
        \;\;\; \text{and} \;\;\;
        v_{i-1}^j \neq v_{i+1}^j
    \end{equation}
    there is a $K^j_i \subseteq J \setminus \{j\}$ with $(v^k_{i-1},v^k_{i+1}) \notin E_G$ and $v_{i-1}^k \neq v_{i+1}^k$ for each $k\in K^j_i$
    such that
    \[
        \sum_{k\in K^j_i \cup \{ j\}} \alpha_k
        \langle x_0^k,\dots,\hat{x}_{i}^k,\dots,x_n^k \rangle
        = 0.
    \]
    Indeed, if this holds, then it can be applied a second time to the set $J\setminus (K^i_j \cup \{j\})$. Successive applications will eventually yield
    $\bar{J}\subseteq J$ such that $\partial_{n,n,i}^M(\langle x_0^j,\dots,x_n^j \rangle) = 0$ for all $j\in \bar{J}$.
    
    To show the claim, note that $x \in \ker \partial^M_{n,n}$ by Lemma~\ref{lem:PathChianKernel}, so for any $i=1,\dots,n-1$ and $j \in J$, there is a minimally sized $K \subseteq J \setminus \{j\}$
    and an $S_k\subseteq\{1,\dots,n-1\}$ for each $k\in K$
    such that
    \begin{equation}\label{eq:GeneralBoundarySum}
        (-1)^i\alpha_j
        \langle v_0^j,\dots,\hat{v}_i^j,\dots,v_n^j \rangle
        + \sum_{k\in K} 
        \sum_{s_k\in S_{k}}
        (-1)^{s_k} \alpha_k
        \langle v_0^k,\dots,\hat{v}_{s_k}^k,\dots,v_n^k \rangle
        = 0.
    \end{equation}
    Suppose that for some $k \in K$ and $s_{k}\in S_{k}$ it is the case that $s_{k} \neq i$.
    In particular, by the minimal size of $K$ we can assume
    \[
        \langle v_0^j,\dots,\hat{v}_i^j,\dots,v_n^j \rangle
        =
        \langle v_0^k,\dots,\hat{v}_{s_k}^k,\dots,v_n^k \rangle.
    \]
    Without loss of generality assume $i<s_{k}$. Then
    \begin{align}\label{ChainEqualityConditions}
         v_0^j,\dots,v_{i-1}^j = & \: v_0^{k},\dots,v_{i-1}^{k}
         \nonumber
        \\ \text{while} \: 
         v_{i+1}^j,\dots,x_{s_{k}}^j = & \: v_i^{k},\dots,v_{s_{k}-1}^{k}
        \\ \text{and} \:
         v_{s_{k}+1}^j,\dots,v_{n}^j = & \: v_{s_{k}+1}^{k},\dots,v_{n}^{k}.
         \nonumber
    \end{align}
    As $\ell(\langle v_0^j,\dots,v_n^j \rangle) = n$ and
    the distance on the vertices of $G$ comes from the digraph quasi-metric, we must have that $d_G(v_t^j,v_{t+1}^j) = 1$ and $d_G(v_t^{k},v_{t+1}^{k}) = 1$ for $t=0,\dots,n-1$.
    However, the construction of the digraph quasi-metric also implies that
    \begin{equation}\label{eq:xMetricCondition}
        d_G(v_{i-1}^j,v_{i+1}^j) = d_G(v_{i-1}^{k},v_i^{k}) = 1
    \end{equation}
    by using the last position of the first line of equation~\eqref{ChainEqualityConditions} and the first position of the second line.
    Meanwhile, by the assumption in equation~\eqref{eq:MagnitudeCondition}, we also have that
    \[
        d_G(v_{i-1}^j,v_{i+1}^j) = d_G(v_{i-1}^j,v_i^j)+d_G(v_i^j,v_{i+1}^j) = 2
    \]
    which contradicts equation~\eqref{eq:xMetricCondition}.
    Therefore, $s_k = i$ for all $s_k \in S_K$ and all $k\in K$. Together with equation \eqref{eq:GeneralBoundarySum}, the proof is completed by setting $K_i^j = K$.
     \qed
\end{proof}

\subsection{Change of coefficients}\label{sec:Coefficients}

We now begin to examine when the structure of $\Omega_n(G;R)$ is preserved under changes of coefficients.
Throughout this section, let $\phi\colon R\rightarrow R'$ be a homomorphism of commutative, unital rings. This equips $R'$ with the structure of an $R$-module and usually we shall make reference only to this structure. Fix also a digraph $G$.

For $n, l\geq0$ define maps
\begin{equation}\label{eq:AllowedPathCoefficientChange}
\mu_n\colon\mathcal{A}_n(G;R)\otimes_RR'\rightarrow \mathcal{A}_n(G;R'),\qquad \mu_n\colon C^M_{n,l}(G;R)\otimes_RR'\rightarrow C^M_{n,l}(G;R')
\end{equation}
by linearly extending the assignments $\mu_n(e_{u_0,\dots,u_n}\otimes r)=r\cdot e_{u_0,\dots,u_n}$ and $\mu_n(\langle u_0,\dots,u_n \rangle \otimes r)=r \cdot \langle u_0,\dots,u_n \rangle$ for an allowed path $e_{u_0,\dots,u_n}$ or $(n+1)$-tuple $\langle u_0,\dots,u_n \rangle$, respectively.

\begin{lemma}\label{muisoonac}
    For each $n, l\geq0$ the two maps $\mu_n$ from equation~\eqref{eq:AllowedPathCoefficientChange}
    are isomorphisms of $R'$-modules.
\end{lemma}

\begin{proof}
    In each case, both the target and domain of $\mu_n$ are free $R'$-modules, and $\mu_n$ establishes a bijection between their canonical bases. \qed
\end{proof}

Making use of the natural isomorphism from Lemma~\ref{lem:PathChianKernel}, for each $n\geq0$ the following diagram commutes
\[
\xymatrix{\mathcal{A}_n(G;R)\otimes_RR'\ar[r]^-{\mu_n}\ar[d]^-{\partial^M_{n,n}\otimes1}&\mathcal{A}_n(G;R')\ar[d]^-{\partial^M_{n,n}}\\
C^M_{n-1,n}(X;R)\otimes_RR'\ar[r]^-{\mu_{n-1}}&C^M_{n-1,n}(G;R').}
\]
By Lemma~\ref{lem:PathChianKernel}, the kernel of $\partial^M_{n,n}$ is the module of path chains, so from the diagram we obtain a well defined map
\begin{equation}\label{eq:FullCoefChangeMu}
\mu_n\colon\Omega_n(G;R)\otimes_RR'\rightarrow\Omega_n(G;R').
\end{equation}
Recall that an $R$-module $M$ is \emph{flat} if the functor $(-)\otimes_RM$ preserves kernels, and that a submodule $N\subseteq M$ of an $R$-module $M$ is \emph{pure} if the map $N\otimes_R P\rightarrow M\otimes_R P$ is monic for each $R$-module $P$.

\begin{proposition}\label{prop:coefprop1}
    The following statements hold.
    \begin{enumerate}
        \item Suppose that $R'$ is a flat $R$-module. Then $\mu_*\colon\Omega_*(G;R)\otimes_RR'\rightarrow\Omega_*(G;R')$ is an isomorphism.
        \item Suppose that $\Omega_n(G;R)$ is a pure submodule of $\mathcal{A}_n(G;R)$. Then $\mu_*\colon\Omega_n(G;R)\otimes_RR'\rightarrow\Omega_n(G;R')$ is injective.
    \end{enumerate}
\end{proposition}

\begin{proof}
Consider the commutative diagram
\[
\xymatrix{0\ar[r]&
\ker(\partial^M_{n,n}\otimes1)
\ar[r]\ar[d]^-{(1)}&\mathcal{A}_n(G;R)\otimes_RR'\ar[rr]^-{\partial^M_{n,n}\otimes1}\ar[d]^-{(2)}&&{C^M_{n-1,n}(G;R)\otimes_RR'}\ar[d]^-{(3)}\\
0\ar[r]&\Omega_n(G;R')\ar[r]&\mathcal{A}_n(G;R')\ar[rr]^-{\partial^M_{n,n}}&&C^M_{n-1,n}(G;R')}
\]
in which the vertical maps are induced by $\mu_n$. The bottom row of the diagram is exact by Lemma~\ref{lem:PathChianKernel} and equation~\eqref{eq:DiagonalKernal}, and the top row is exact by construction. Since $(2),(3)$ are isomorphisms by Lemma~\ref{muisoonac}, the Five Lemma shows that $(1)$ is an isomorphism.

Under the assumption that $R'$ is $R$-flat, the natural map $\alpha \colon \Omega_n(G;R)\otimes_RR'\rightarrow \ker(\partial^M_{n,n}\otimes1)$ is an isomorphism and we obtain the first statement. On the other hand, under the assumption that $\Omega_n(G;R)\subseteq\mathcal{A}_n(G;R)$ is pure, $\alpha$ is injective and we obtain the second statement.
\qed
\end{proof}

A ring is von Neumann regular if and only if all its modules are flat. If $R$ is such a ring, then by Proposition~\ref{prop:coefprop1} $\mu_*\colon\Omega_*(G;R)\otimes_RR'\rightarrow\Omega_*(G;R')$ will be an isomorphism with no assumptions on $R'$ save that it be an $R$-module. We will use only the following special case of this observation.

\begin{corollary}\label{cor:CoeffChange1}
Suppose that $R$ is a field. Then $\mu_*\colon\Omega_*(G;R)\otimes_RR'\rightarrow\Omega_*(G;R')$ is an isomorphism. \qed
\end{corollary}
We also extract a statement covering change from integral to rational or real coefficients.
\begin{corollary}\label{cor:CoeffChange2}
Suppose that $R=\mathbb{Z}$ and $R'$ has characteristic $0$. Then $\mu_*:\Omega_*(G;R)\otimes_RR'\rightarrow\Omega_*(G;R')$ is an isomorphism.
\end{corollary}

\begin{proof}
Since $\mathbb{Z}$ is a PID, any torsion-free $\mathbb{Z}$-module is flat. \qed
\end{proof}
Finally, we give an application for the second part of the proposition.
\begin{corollary}\label{cor:CoeffChange3}
Suppose that $R$ is a PID. Then $\mu_*\colon\Omega_*(G;R)\otimes_RR'\rightarrow\Omega_*(G;R')$ is injective.
\end{corollary}
\begin{proof}
The map $\partial^M_{n,n}\colon\mathcal{A}_n(G;R)\rightarrow C^M_{n-1,n}(G;R)$ induces an isomorphism of the quotient module $\mathcal{A}_n(G;R)/\Omega_n(G;R)$ onto a submodule of $C^M_{n-1,n}(G;R)$. Since $C^M_{n-1,n}(G;R)$ is free and $R$ is a PID, it follows that $\mathcal{A}_n(G;R)/\Omega_n(G;R)$ is free and in particular flat. This is sufficient criteria for $\Omega_n(G;R)$ to be a pure submodule of $\mathcal{A}_n(G;R)$. \qed
\end{proof}
One can check that the isomorphisms in Lemma~\ref{muisoonac} and Proposition~\ref{prop:coefprop1}, as well as in Corollaries~\ref{cor:CoeffChange1} and~\ref{cor:CoeffChange2}, are natural in $G$. Similarly, Corollary~\ref{cor:CoeffChange3} supplies a natural injection.

Having produced several change of coefficient statements for path chains, it is natural to look for applications in homology. As $\mu_n$ commutes with the path differentials, there is a natural induced map
\[
    \bar{\mu}_n\colon H^P_n(G;R)\otimes_RR'\rightarrow H^P_n(G;R').
\]
\begin{proposition}\label{Pr:coefisopid}
Suppose that $R$ is a PID and $R'$ is a flat $R$-module. Then there is a natural isomorphism
\begin{equation*}
    H_*^P(G;R')\cong H^P_*(G;R)\otimes_RR'.
\end{equation*}
\end{proposition}
\begin{proof}
The assumption implies that $(\Omega_*^P(G;R),\partial^P)$ is a complex of free $R$-modules. Thus using Proposition~\ref{prop:coefprop1}, the statement of the proposition is a consequence of the universal coefficient theorem given in \cite[Theorem 7.55, pg. 448]{Rotman2009}. \qed
\end{proof}
Note that if $R$ is a field, then Proposition~\ref{Pr:coefisopid} holds with no conditions on $R'$. In any case, we obtain isomorphisms 
\[
H^P_*(G;\mathbb{Q})\cong H^P_*(G;\mathbb{Z})\otimes\mathbb{Q},\qquad H^P_*(G;\mathbb{R})\cong H^P_*(G;\mathbb{Z})\otimes\mathbb{R}
\]
for any digraph $G$.

We end the section with a few further remarks. There is a natural map
\begin{equation}\label{eq:TensorInclusion}
    i_{R,R'} \colon \Omega_n(G;R) \to \Omega_n(G;R) \otimes_{R} R'
\end{equation}
given by $i_{R,R'}(x) = x \otimes 1$. When $\Omega_*(G;R)$ is free (such as when $R$ is a PID) and  $R$, $R'$ have the same characteristic,
$i_{R,R'}$ restricts to a bijection on the canonical bases. In particular, when $R=\mathbb{Z}$ or $R=\mathbb{Z}_p$ for some prime $p$, Corollary~\ref{cor:CoeffChange2}~or~\ref{cor:CoeffChange1} demonstrates that $\Omega_n(G;R')$ is a free module for any $R'$ of characteristic $0$ or $p$, respectively.
Furthermore, by Corollary~\ref{cor:CoeffChange3}, $\mu_*\colon\Omega_*(G;\mathbb{Z})\otimes_\mathbb{Z}\mathbb{Z}_p\rightarrow\Omega_*(G;\mathbb{Z}_p)$ is injective for any prime $p$.

\subsection{Bigrading of path chains and face maps at a vertex}\label{sec:BigradingAndFaceMaps}

The material presented in this section, which is related to connectedness and bigradings of $\Omega_n(G;R)$, constitutes a generalisation of parts of~\cite[Section 2.2]{Grigoryan2022} to the case of non-field coefficients.
In the reminder of the section, given $x \in \Omega_n(G;R)$ or $x \in \mathcal{A}_n(G;R)$, we assume that $x$ has the form
\begin{equation}\label{genformx}
    x = \sum_{e_{v_0,\dots,v_n}\in P^G_n} \alpha_{v_0,\dots,v_n} e_{v_0,\dots,v_n}
\end{equation}
where each $\alpha_{v_0,\dots,v_n} \in R$ and only finitely many $\alpha_{v_0,\dots,v_n} \neq 0$. 

\begin{definition}\label{def:UpperLowerConnected}
    Let $x \in \Omega_n(G;R)$. Define the \emph{head set} $\bar{h}(x)$ and the \emph{tail set} $\bar{t}(x)$ to be
    \begin{align*}
        \bar{h}(x) &=
        \{ v_n \: | \: e_{v_0,\dots,v_n} \in P^G_n \; \text{and} \; \alpha_{v_0,\dots,v_n} \neq 0 \}
        \\ \;\;\; \text{and} \;\;\;
        \bar{t}(x) &=
        \{ v_0 \: | \: e_{v_0,\dots,v_n} \in P^G_n \; \text{and} \; \alpha_{v_0,\dots,v_n} \neq 0 \}
        .
    \end{align*}
    If $|\bar{h}(x)| = 1$, then $x$
    is called \emph{upper connected} and the unique element $h(x) \in \bar{h}(x)$ is called the \emph{head} of $x$.
    Similarly, if $|\bar{t}(x)| = 1$, then $x$
    is called \emph{lower connected} and the unique element $t(x) \in \bar{t}(x)$ is called the \emph{tail} of $x$.
    The element $x$ is called \emph{connected} if it is both upper connected and lower connected.
\end{definition}

Recall from Lemma~\ref{lem:PathChianKernel} that there are isomorphisms $\phi_n\colon \mathcal{A}(G;R)\to C_{n,n}^M(G;R)$ and $\phi_n\colon \Omega_n(G;R)\to H_{n,n}^M(G;R)$. From equation~\eqref{eq:DiagonalKernal} it follows that 
\[    
\ker(\partial^M_{n,n})= H^M_{n,n}(G;R) \cong \Omega_n(G;R)
\]
where $\partial^M_{n,n}$ is the magnitude differential.
Since $\partial^M_{n,n}$ preserves the endpoints of the $(n+1)$-tuples in the magnitude chain complex, $\ker(\partial^M_{n,n})$ is generated by elements which are sums of tuples with the same endpoints. These elements correspond under the isomorphism above to elements of $\Omega(G;R)$ which are sums of paths having the same endpoints. Therefore, $\Omega_n(G;R)$ is generated by connected elements (compare \cite[Lemma 2.2]{Grigoryan2022}, where a similar statement was made for field coefficients).

It now follows that for any vertices $v_h,v_t\in V_G$, the connected elements with head $v_h$ and tail $v_t$ span a submodule $\Omega_n^{v_t,v_h}(G;R)$ of $\Omega_n(G;R)$. Since the connected elements with different endpoints are linearly independent, there is a decomposition
\begin{equation}\label{eq:Bigrading}
    \Omega_n(G;R) = \bigoplus_{v_t,v_h \in V_G} \Omega_n^{v_t,v_h}(G;R)
\end{equation}
which we use to define a bigrading on $\Omega_n(G;R)$.

We say that a subset of $\Omega_n(G;R)$ \emph{respects the bigrading} $\Omega_*^{*,*}(G;R)$ if each of its elements is contained in some $\Omega_n^{v_t,v_h}(G;R)$. From the discussion above, it follows that $\Omega_n(G;R)$ always has a generating set which respects the bigrading, and in particular has a basis which respects the bigrading whenever it is a free $R$-module.

\begin{definition}
    For $v\in V_G$ define homomorphisms 
        $\delta^h_{n,v},\: \delta^t_{n,v} \colon 
        \mathcal{A}_n(G;R) \to \mathcal{A}_{n-1}(G;R)$
    by
    \begin{equation*}
        \delta_{n,v}^h(x)  =
        \sum_{\substack{e_{v_0,\dots,v_n}\in P^G_n \\ v_{n-1}=v}}
        \alpha_{v_0,\dots,v_n} e_{v_0,\dots,v_{n-1}} 
    \;\;\; \text{and} \;\;\;
        \delta_{n,v}^t(x)  =
        \sum_{\substack{e_{v_0,\dots,v_n}\in P^G_n \\ v_1=v}}
        \alpha_{v_0,\dots,v_n} e_{v_1,\dots,v_n}
    \end{equation*}
    where we assume that $x \in \mathcal{A}_n(G;R)$ has the form \eqref{genformx} and we follow the usual convention that an empty sum is zero.
\end{definition}

\begin{lemma}\label{lem:SumFace}
    For each $v\in V_G$ the functions $\delta_{n,v}^h,\:\delta_{n,v}^t\colon \mathcal{A}_n(G;R) \to \mathcal{A}_{n-1}(G;R)$ restrict to well defined $R$-module homomorphisms
    \begin{equation}\label{eq:OmegaBimodule}
        \delta_{n,v}^h,\:\delta_{n,v}^t \colon
        \Omega_n(G;R)\to \Omega_{n-1}(G;R).
    \end{equation}
\end{lemma}

\begin{proof}
    We prove only the part of the statement regarding $\delta_{n,v}^h$.
    Thus suppose that $x\in \Omega_n(G;R) \subseteq \mathcal{A}_n(G;R)$.
    Applying Lemma~\ref{lem:NoPositionSwap} gives $\partial^M_{n,n,i}(x)=0$ for each $i=1,\dots,n-1$, which, by the definition of $\delta_{n,v}^h$, implies that $\partial^M_{n,n,i}(\delta_{n,v}^h(x))=0$ for each $i=1,\dots,n-2$.
    This shows that $\delta_{n,v}^h(x) \in \Omega_{n-1}(G;R)$, again by applying Lemma~\ref{lem:NoPositionSwap}.
    \qed
\end{proof}

Observe that the homomorphisms $\delta_{n,v}^h$ and $\delta_{n,v}^t$ respect the $\Omega_*^{*,*}(G;R)$ bigrading.
That is, $\delta_{n,v}^h,\delta_{n,v}^t$ restrict to $R$-module homomorphisms
\begin{equation*}
    \delta_{n,v}^h \colon
    \Omega_n^{u,w}(G;R)\to \Omega_{n-1}^{u,v}(G;R)
    \;\;\; \text{and} \;\;\;
    \delta_{n,v}^t \colon
    \Omega_n^{u,w}(G;R)\to \Omega_{n-1}^{v,w}(G;R)
\end{equation*}
for any $u,v,w\in V_G$. Moreover, for any $x\in \Omega_n(G;R)$ and $v \in V_G$, the element $\delta_{n,v}^h(x)$ is upper connected if nonzero, and $\delta_{n,v}^t(x)$ is lower connected if nonzero. We obtain homomorphisms
\[
    \delta^h_n,\: \delta^t_n\; \colon 
    \Omega_n(G;R) \to \Omega_{n-1}(G;R)
\;\;\; \text{given by} \;\;\;
    \delta^h_n(x) = \sum_{v\in V_G} \delta_{n,v}^h(x)
    \;\;\; \text{and} \;\;\;
    \delta^t_n(x) = \sum_{v\in V_G} \delta_{n,v}^t(x)
\]
that coincide with the maps $(-1)^n\partial_{n,n}^M$ and $\partial_{n,0}^M$, respectively.

The next lemma, concerning the interactions between the maps $\delta_*^h$, $\delta_*^t$ and $i_{R,R'}$, $\mu_*$ from Section~\ref{sec:Coefficients}, follows directly from the definitions of $\mu_*$, $i_{R,R'}$, $\delta_*^h$ and $\delta_*^t$.

\begin{lemma}\label{lem:CoeffChange}
   The following diagram commutes for each $v \in V_G$.
    \begin{align*}
        &\xymatrix{
            \Omega_n(G;R)
            \ar[r]^-{i_{R,R'}}
            \ar[d]_-{\delta^h_{n,v}}
            &
            \Omega_n(G;R) \otimes_{R} R'
            \ar[r]^-{\mu_n}
            \ar[d]_-{\delta^h_{n,v}\otimes 1_{R'}}
            &
            \Omega_n(G;{R'})
            \ar[d]^-{\delta^h_{n,v}}
            \\
            \Omega_{n-1}(G;R)
            \ar[r]^-{i_{R,R'}}
            &
            \Omega_{n-1}(G;R) \otimes_{R} R'
            \ar[r]^-{\mu_{n-1}}
            &
            \Omega_{n-1}(G;R')
        }
    \end{align*}
    Similarly, the diagram obtained by replacing $\delta^h_{n,v}$ everywhere with $\delta^t_{n,v}$ commutes.
\end{lemma}

To complete this section we show how the path bigrading can be used to prove Proposition~\ref{prop:Dim2Base} from the previous section.

\begin{proof}[Proof of Proposition~\ref{prop:Dim2Base}]
    Using the direct sum decomposition in equation~\eqref{eq:Bigrading}, the result follows by obtaining generating sets and bases for each $\Omega_2^{v_t,v_h}(G;R)$.
    If $d_G(v_t,v_h) = 0$, then $v_t = v_h$, either $\Omega_2^{v_t,v_h}(G;R) = 0$ or it is generated by double edges and these also form a basis.
    If $d_G(v_t,v_h) = 1$, then either $\Omega_2^{v_t,v_h}(G;R) = 0$ or it is generated by a set of directed triangles which also form a basis.
    If $d_G(v_t,v_h)>2$, then $\Omega_2^{v_t,v_h}(G;R) = 0$.
    Hence, assume that $d_G(v_t,v_h) = 2$ and $\Omega_2^{v_t,v_h}(G;R) \neq 0$.
    Then either there is a multisquare between $v_t$ and $v_h$, or $\Omega_2^{v_t,v_h}(G;R)$ is generated by a directed square and the result follows similarly to the previous cases.
    \qed
\end{proof}

\section{Upper and lower extensions over face multigraphs}\label{sec:Extensions}

In this section we introduce the fundamental constructions on the path chain complex which are needed in the rest of the paper.
Throughout the section, let $G$ be a digraph, $n$ a non-negative integer, and $R$ a commutative ring with unit, unless stated otherwise.

\subsection{Extensions and complete face multigraphs}

We introduce here the processes of forming upper and lower extensions, which are methods of constructing upper and lower connected elements of $\mathcal{A}_{n+1}(G;R)$ from elements of $\Omega_n(G;R)$.
After this, we define face multigraphs. These objects contain the data needed to ensure that the extensions formed belong to $\Omega_{n+1}(G;R)$. 

The theory developed in this section will eventually be applied to construct generating sets for $\Omega_n(G;R)$ when $R$ has characteristic $0$ or $2$. To cope with arbitrary coefficient rings, face multigraphs must be replaced by more complex objects, and these constructions will be dealt with in a future paper. Still, the definitions can be formulated independently of $R$, and applications in this context can be found in the results of Sections~\ref{sec:relwithsn(G)classes}~and~\ref{sec:Dim3}, and in the examples of Section~\ref{sec:ImportantExamples}.

To end the section, we provide a notion of completeness of a face multigraph with respect to a vertex.
The main result of Section~\ref{sec:compltext} below states that the existence of a complete face multigraph is sufficient to obtain an extension within $\Omega_{n+1}(G;R)$.

\begin{definition}
    Let $x \in \Omega_{n}(G;R)$ have the form
\begin{equation*}
    x = \sum_{e_{v_0,\dots,v_n}\in P^G_n} \alpha_{v_0,\dots,v_n} e_{v_0,\dots,v_n}
\end{equation*}
where each $\alpha_{v_0,\dots,v_n} \in R$.
    Define the \emph{upper extension} $[x]^v$ of $x$ by $v\in V_G$ to be
    \[
        [x]^v =
        \sum_{\substack{e_{v_0,\dots,v_n,v}\in P^G_{n+1} 
        }}
        \alpha_{v_0,\dots,v_n}
        e_{v_{0},\dots,v_{n}, v}
        \in \mathcal{A}_{n+1}(G;R).
    \]
    An \emph{upper extension} of $x$ is any upper extension by some $v \in V_G$.
    Similarly, 
    define the \emph{lower extension} $[x]_u$ of $x$ by $u \in V_G$ to be
    \[
        [x]_u =
        \sum_{\substack{e_{u,v_0,\dots,v_n}\in P^G_{n+1} 
        }}
        \alpha_{v_0,\dots,v_n}
        e_{u, v_{0},\dots,v_{n}}
        \in \mathcal{A}_{n+1}(G;R).
    \]
    A \emph{lower extension} of $x$ is any lower extension by some $u \in V_G$.
    Any upper extension or lower extension of $x$ will be referred to as an \emph{extension} of $x$.
\end{definition}
Note that the only nonzero summands of the upper extension $[x]^v$ defined above are formed from paths $e_{v_0,\dots,v_n}$ when $\alpha_{v_0,\dots,v_n} \neq 0$ and $(v_n,v)\in E_G$, where in particular $v_n\in \bar{h}(x)$. A similar statement can also be made for lower extensions and tail sets. 

In the sequel, we develop in parallel two separate notions of upper and lower face multigraphs corresponding to the use of upper and lower extensions as defined above, and making use of the maps $\delta^h_{n,v}$ and $\delta^t_{n,v}$ for $v\in V_G$, respectively.
We focus primarily on the case of upper extensions, explaining any differences for the lower case in brackets immediately after.

\begin{definition}\label{def:FaceMultiGraph}
    Let $n\geq 0$ and $x_1,\dots,x_m\in\Omega_n(G;R)$ be such that $x_i \neq - x_j$ for any $i,j = 1,\dots,m$ and $i\neq j$.
    We define an \emph{upper (lower) face multigraph} on $x_1,\dots,x_m$ to be a labelled multigraph $F_n(x_1,\dots,x_m)$ together with a choice, for each $u\in V_G$ and $i = 1,\dots,m$ such that $\delta_{n,u}^h(x_i)\neq 0$ ($\delta_{n,u}^t(x_i)\neq 0$), of $x_i^{u,1},\dots,x_i^{u,m_i^u} \in \Omega_{n-1}(G;R)$ such that
    \begin{equation}\label{eq:n-1Faces}
        \delta^h_{n,u}(x_i) =
        x_i^{u,1} + \cdots + x_i^{u,m_i^u}
        \;\;\;
        (\delta^t_{n,u}(x_i) =
        x_i^{u,1} + \cdots + x_i^{u,m_i^u})
    \end{equation}
    where no sub-sequence of $x_i^{u,1},\dots,x_i^{u,m_i^u}$ sums to zero. This data is to be subject to the following conditions.
    \begin{enumerate}
        \item 
        The multigraph $F_n(x_1,\dots,x_m)$ has $m$ vertices which are labelled with a bijection to the multiset
        \[
            \{x_1,\dots,x_m\}.
        \]
        \item
        There can be an edge between distinct vertices $x_i,x_j$ only if $\delta_{n,u}^h(x_i)\neq 0$ and $\delta_{n,u}^h(x_j)\neq 0$ ($\delta_{n,u}^t(x_i)\neq 0$ and $\delta_{n,u}^t(x_j)\neq 0$) for some $u \in V_G$ and $x_i^{u,k} = -x_j^{u,l}$ for some $k=1,\dots,m_i^u$ and $l=1,\dots,m_j^u$. If an edge between these vertices exists, then it is labelled with a pair
        \[
            (\{x_i^{u,k},x_j^{u,l}\},u)
            \;\;\; \text{such that} \;\;\;
            x_i^{u,k} = -x_j^{u,l}
        \]
        where $k\in \{1,\dots,m_i^u$\} and $l \in \{1,\dots,m_j^u\}$.
        \item 
        For $u \in V_G$ and $i=1,\dots,m$ such that $\delta_{n,u}^h(x_i)\neq 0$ and $k=1,\dots,m_i^u$, each $x_i^{u,k}$ appears in no more than one edge label.
    \end{enumerate}
    We refer to the decomposition in equation~\eqref{eq:n-1Faces} as the \emph{ridge decomposition} of $x_i$ at $u$ in $F_n(x_1,\dots,x_m)$.
\end{definition}
Note that there will usually be a large number of distinct face multigraphs on any given set of vertices $x_1,\dots,x_m \in \Omega_n(G;R)$. Note also that in general the vertices $x_1,\dots,x_m$ need not be distinct. However, in later sections of this paper we will work with rings $R$ of characteristic $2$, and in this case the multiset in part (1) above reduces to a set, as the condition $x_i \neq - x_j$ for any $i,j = 1,\dots,m$ ensures that there are no elements of multiplicity greater than $1$. 

When we draw face multigraphs, for simplicity, an edge labelled $(\{x_i^{u,k},x_j^{u,l}\},u)$ is often written only with the label $x_i^{u,k}$ or $x_i^{u,l}$ if no ambiguity can arise. Similarly, we also often refer to the elements $x_i^{u,k}$ for $i = 1,\dots, m$, and $k=1,\dots,m_i^u$ which appear in the ridge decompositions 
as the \emph{edge labels} of face multigraph $F_n(x_1,\dots,x_m)$. Examples of face multigraphs are constructed and drawn in the next subsection (see also Example~\ref{exam:2Trapezohedron}).
\begin{rmk}\label{rmk:NecessityOfConditions}
    It is not immediately clear why it is necessary to allow for ridge decompositions with multiple nonzero or non-unital terms. However, later we will want to consider the case when the $x_i^{u,k}$ are elements of a basis, and when $n=3$ and there is a multisquare between vertices $w$ and $u$, then a nonzero $\delta^h_{3,u}(x_i)$ lies in a $\Omega_n^{w,u}(G;R)$ for which a basis consists of at least two elements.

     More concretely, in Example~\ref{exam:SingleFaceDecomposition} we construct a digraph $G$ for which there is an element $I_4 \in \Omega_4(G;\mathbb{Z})$ and a vertex $u \in V_G$ such that $\delta^h_{n,u}(I_4)$ must decompose as a sum of two distinct elements of the given $\Omega_3(G;\mathbb{Z})$ basis. Similarly, in Example~\ref{exam:ArbitraryMultiplicities}, for $t=2,3,\dots$ we construct digraphs $\mathbb{M}_t$ and elements $I_4^t \in \Omega_4(\mathbb{M}_t;\mathbb{Z})$ such that $\delta^h_{n,u}(I_4^t)=te+\dots$ for some element $e$ of the given $\Omega_3(\mathbb{M}_t;\mathbb{Z})$ basis for some $u \in V_G$.
\end{rmk}

Face multigraphs will appear in this work only when considering an extension of some form over the sum of their vertex labels, as is made precise in Definition~\ref{def:extoffmhg} below. From now on we usually drop the words upper and lower when it is clear which type of face multigraph is under consideration.

In the following definition we set out notions of what we call properness and completeness with respect to an extension for a face multigraph. The complete extensions will be the central notion used in the reminder of the paper and it is these extensions which we later use to construct higher-dimensional path chains. 
Although it would be possible to package face multigraphs, proper extensions, and complete extension together as one concept, we choose to introduce these objects separately, as we believe these notions might be used independently in the construction of path homology algorithms.
\begin{definition}\label{def:extoffmhg}
    Using the notation of Definition~\ref{def:FaceMultiGraph}, an upper (lower) face multigraph $F_n(x_1,\dots,x_m)$ is said to be $v$\emph{-proper} over $v\in V_G$ if all edges $(\{x_{i}^{u,k},x_{j}^{u,l}\},u)$ corresponding to $u\in V_G$ are such that 
    \[
        (u,v) \notin E_G \;\;\; ((v,u) \notin E_G) \;\;\; \text{and} \;\;\; u \neq v.
    \]
    The upper (lower) face multigraph $F_n(x_1,\dots,x_m)$ is further called $v$-\emph{complete} if the converse also holds. That is, if for any
    $i=1,\dots,m$ and $u \in \bar{h}(\delta^h_n(x_i)) \setminus\{v\}$ ($u \in \bar{t}(\delta^t_n(x_i)) \setminus\{v\}$)
    with $(u,v)\not\in E_G$ ($(v,u)\not\in E_G)$), $F_n(x_1,\dots,x_n)$ has
    for each $k=1,\dots,m_i^u$ an edge labelled $(\{x_{i}^{u,k},x_{j}^{u,l}\},u)$ for some $l=1,\dots,m^u_{j}$ and $j=1,\dots,t$.
\end{definition}
A $v$-proper face multigraph $F_n(x_1,\dots,x_m)$ contains no edges of the form $(\{x_i^{u,k},x_j^{u,l}\},u)$ when  either $(u,v) \in E_G$, or $u=v$. In this case, the ridge decompositions for $\delta^h_{n,u}(x_i)$ and $\delta^h_{n,u}(x_j)$ carry essentially no information, as all possible choices yield otherwise identical face multigraphs.
\begin{definition}
Given an upper (lower) face multigraph $F_n(x_1,\dots,x_m)$ with vertices $x_1,\dots,x_m \in \Omega_n(G;R)$ and given $v\in V_G$, the upper extension $[x_1+\cdots+x_m]^v$ (lower extension $[x_1+\cdots+x_m]_v$) is called an \emph{upper (lower) extension} by $v$ over $F_n(x_1,\dots,x_m)$ when $(v_i,v)\in E_G$ ($(v,v_i)\in E_G$) for each $v_i \in \bar{h}(x_i)$ ($v_i \in \bar{t}(x_i)$) and $i=1,\dots,m$. The extension is called $v$-\emph{proper} or $v$-\emph{complete} if $F_n(x_1,\dots,x_m)$ is $v$-proper or $v$-complete, respectively.
\end{definition}

\subsection{Initial examples}

In this section we give a number of simple, explicit, and more general examples of extensions over face multigraphs. 
\begin{exmp}\label{exam:VerticesAndLines}\normalfont
    For any $v\in V_G$, an element $x \in \Omega_1(G;R)$ which is obtained as an upper extension over a connected $v$-complete face multigraph must be extended over a multigraph which consists of a single vertex,
    since
    \[
        \delta_{0,u}^h\delta^h_{1,v}(x) = 0
    \]
    for any $u\in V_G$.
    As discussed in Section~\ref{sec:LowDimBasis}, the set $\{e_u\mid u\in V_G\}$ is a basis of $\Omega_0(G;R)$.
    The extension described above over the basis elements $e_u \in \Omega_0(G;R)$ is a face multigraph consisting of a single vertex labelled $e_u$ such that $(u,v)\in E_G$.
    That is, such extensions correspond to edges in $G$ that begin at $u$ and end at some $v\in V_G$.
    In particular, recall again from Section~\ref{sec:LowDimBasis}  that the set of $e_{u,v}$ such that $(u,v)\in E_G$ forms a basis of $\Omega_1(G;R)$.
\end{exmp}

\begin{exmp}\label{exam:DirectedSquare}\normalfont
    Suppose that $G$ is a digraph containing a subdigraph
    
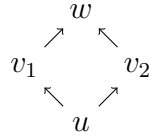
\begin{center}
        \tikz {
            \node (a) at (0,0) {$u$};
            \node (b1) at (-0.75,0.75) {$v_1$};
            \node (c) at (0,1.5) {$w$};
            \node (b2) at (0.75,0.75) {$v_2$};
            \draw[->] (a) -- (b1);
            \draw[->] (b1) -- (c);
            \draw[->] (b2) -- (c);
            \draw[->] (a) -- (b2);
        }
    \end{center}
    such that $(u,w) \notin G$.
    We have $e_{u,v_1}, e_{u, v_2} \in \Omega_1(G;R)$ and $e_u \in \Omega_0(G;R)$.
    The directed square
    \[
        e_{u,v_1,w} - e_{u,v_2,w} \in \Omega_2(G;R)
    \]
    can be obtained as the upper extension $[e_{u,v_1}-e_{u, v_2}]^w$, which is an upper extension over the face multigraph
    \begin{center}
        \tikz {
            \node (b1) at (0,0) {$e_{u,v_1}$};
            \node (b2) at (3,0) {$- e_{u, v_2}$};
            \draw[-] (b1) -- (b2) node[midway,above] {$e_u$};
        }
    \end{center}
    where we apply the conventions discussed above Remark~\ref{rmk:NecessityOfConditions}
    to the edge label, whose full label would be $(\{e_u, -e_u \},u)$.
    In fact, this extension is $w$-complete over the face multigraph, as $(u,w) \notin G$, $u \neq w$,
    \[
        \delta_{1,u}^h(e_{u,v_1}) =
        \delta_{1,u}^h(e_{u,v_2}) =
        e_u,
        \;\;\; \text{and} \;\;\;
        \delta_{1,v}^h(e_{u,v_1}) =
        \delta_{1,v}^h(e_{u,v_2}) =
        0
    \]
    for any $v\in V_G$ such that $v \neq u$.
    Therefore, the face multigraph above is the unique connected, $w$-complete face multigraph containing vertices $e_{u,v_1}$ and $-e_{u,v_2}$.
    
    The example explains how every directed square arises as an extension over a complete face multigraph. Similarly, it can be seen that any directed triangle or double edge is an extension over a face multigraph consisting of a single vertex.
\end{exmp}

An \emph{inclusion} of face multigraphs is an inclusion of multigraphs that preserves vertex and edge labels.
A face multigraph $F_n$ is a \emph{sub-face multigraph} of a face multigraph $F'_n$ if there is an inclusion of face multigraphs from $F_n$ to $F'_n$.

\begin{exmp}\normalfont
    For each $v \in V_G$, by definition, any connected $v$-complete face multigraph is maximal among all connected $v$-proper face multigraph when these objects are partially ordered by inclusion of sub-face multigraphs.
\end{exmp}

A pair $(x,v)$ for $x\in \Omega_n(G;R)$ and $v\in V_G$ is called \emph{upper isolated} if for any $u\in \bar{h}(\delta_{n-1}^h\delta_{n,v}^h(x))$ and $w_h\in \bar{h}(x)$, either $u=w_h$ or there is an edge $(u,w_h)\in E_G$.
Similarly, a pair $(x,v)$ for $x\in \Omega_n(G;R)$ and $v\in V_G$ is called \emph{lower isolated} if for any $u\in \bar{t}(\delta^t_{n-1}\delta^t_{n,v}(x))$ and $w_t\in \bar{t}(x)$, either $u=w_t$ or there is an edge $(w_t,u)\in E_G$.
A digraph $G$ is called \emph{transitive} when for any $u,v,w \in V$, if $(u,v),(v,w) \in E_G$ then $(u,w)\in E_G$. 

\begin{exmp}\normalfont
    A connected $v$-proper upper (lower) face multigraph containing a vertex $x$ such that $(x,v)$ is upper (lower) isolated consists of only that vertex. 
    In particular, if $G$ is a transitive digraph, then all connected $v$-proper upper (lower) face multigraphs are of this form.
\end{exmp}

\begin{exmp}\label{exam:Dimension3}\normalfont
    Recall from Proposition~\ref{prop:Dim2Base} that $\Omega_2(G;R)$ is generated by double edges, directed triangles, and directed squares. A typical directed triangle or double edge takes the form
    \[
        T_{u,v,w}= e_{u,v,w} \in \Omega_2(G;R)
        \;\;\;
        \text{and}
        \;\;\;
        E_{u,v} = e_{u,v,u} \in \Omega_2(G;R)
    \]
    for some
    $u,v,w \in V_G$
    such that $(u,w) \in G$. For any $v'\in V_G$, the only nonzero $\delta_{n,v'}^h(T_{u,v,w})$ and $\delta_{n,v'}^h(E_{u,v})$ occur when $v'=v$, in which case
    \[
        \delta_{n,v}^h(T_{u,v,w}) = e_{u,v}
        \;\;\; \text{and} \;\;\;
        \delta_{n,v}^h(E_{u,v}) = e_{u,v}.
    \]
    Hence, if $T_{u,v,w}$ or $E_{u,v}$ appears as a vertex in an upper face multigraph, then it has valence at most $1$. Similarly, if a directed square appears as a vertex in an upper face multigraph, then it has valence at most $2$. Therefore, any upper face multigraph whose vertices are double edges, directed triangles, or directed squares from $\Omega_2(G;R)$ consists of a disjoint union of face multigraphs that are either lines or cycles as unlabelled multigraphs.
\end{exmp}

For future reference, we summarise the content of Examples~\ref{exam:VerticesAndLines},~\ref{exam:DirectedSquare},~and~\ref{exam:Dimension3} in the following proposition.
\begin{proposition}\label{prop:exsumm}
The following statements hold.
\begin{enumerate}
    \item
    An edge $e_{u,v} \in \Omega_1(G;R)$ is the unique upper (lower) extension over a $v$-complete ($u$-complete) connected face multigraph with a vertices labelled $e_u \in \Omega_0(G;R)$ ($e_v \in \Omega_0(G;R)$).
    \item
    An element of $\Omega_2(G;R)$ is an extension over a complete connected face multigraph with vertex labels of the form $e_{u,v} \in \Omega_1(G;R)$ if and only if it is a directed square, directed triangle, or double edge up to sign.
    \item 
    Any connected face multigraph with vertices, double edges, directed triangles, or directed squares from $\Omega_2(G;R)$ is as an unlabelled multigraph a line or cycle. \qed
\end{enumerate}
\end{proposition}

\subsection{Complete extensions over face multigraphs are path chains}\label{sec:compltext} 

The next proposition demonstrates the reason for considering vertex complete extensions over face multigraphs.

\begin{proposition}\label{prop:GraphExtentsion}
    Let $F_n(x_1,\dots,x_m)$ be a face multigraph, where $x_1,\dots,x_m\in \Omega_n(G;R)$, and let $v \in V_G$. If the upper (lower) extension of $x_1+\cdots+x_m$ by $v$ is $v$-complete over $F_n(x_1,\dots,x_m)$, then
    \[
        [x_1 + \cdots + x_m]^v \in \Omega_{n+1}(G;R)
        \;\;\;
        ([x_1 + \cdots + x_m]_v \in \Omega_{n+1}(G;R)).
    \]
\end{proposition}

\begin{proof}
    We prove the case of upper extensions, the proof for lower extensions being similar.
    As $x_1+\cdots+x_m \in \Omega_n(G)$, by Lemma~\ref{lem:NoPositionSwap} we have that $\partial^M_{n,n,i}(x_1+\cdots+x_m) = 0$ and so $\partial^M_{n+1,n+1,i}([x_1+\cdots+x_m]^v) = 0$ for each $i=1,\dots,n-1$.
    Therefore, it remains to check that 
    \[
        \partial^M_{n+1,n+1,n}([x_1+\cdots+x_m]^v) = 0.
    \]

    To begin, for each $i=1,\dots,m$ let $\{x_i^{u,j}\}_{j=1,\dots,m_i^u}$ be the data of the ridge decomposition of $F_n(x_1,\dots,x_m)$ given in equation~\eqref{eq:n-1Faces} of Definition~\ref{def:FaceMultiGraph}. Now consider any 
    path $e_{u_0,\dots,u_{n-1}}$ occurring as a nonzero summand $\alpha e_{u_0,\dots,u_{n-1}}$ in some $x_i^{u_{n-1},k}$ for some $i = 1 ,\dots, m$, $k=1,\dots,m_i^{u_{n-1}}$.
    If 
    $(u_{n-1},v) \in E_G$ or $u_{n-1} = v$,
    then by definition
    \begin{equation}\label{eq:DegenerateCase}
         \partial^M_{n+1,n+1,n}([[e_{u_0,\dots,u_{n-1}}]^{u_n}]^v )
         = 0
    \end{equation}
    for any $u_{n}\in V_G$. Thus we reduce to the case that  $(u_{n-1},v) \notin E_G$ and $u_{n-1} \neq v$.
    
    When $(u_{n-1},v) \notin E_G$, $u_{n-1} \neq v$ and $x_i^{u_{n-1},k}$ appears uniquely in an edge of $F_n(x_1,\dots,x_m)$,
    by the $V$-completeness of $F_n(x_1,\dots,x_m)$, the unique edge associated to $x_i^{u_{n-1},k}$ provides a unique $j\in \{1,\dots,m\}$ and $l\in \{1,\dots,m_j^{u_{n-1}}\}$ such that $j\neq i$ and
    \[
        x_i^{u_{n-1},k}=-x_j^{u_{n-1},l}.
    \]
    In particular, $-\alpha e_{u_0,\dots,u_{n-1}}$ occurs as a summand in $x_j^{u_{n-1},l}$ uniquely corresponding to a summand $\alpha e_{u_0,\dots,u_{n-1}}$ in $x_{i}^{u_{n-1},k}$.
    Moreover,
    \begin{equation}\label{eq:Non-DegenerateCase}
        \partial^M_{n+1,n+1,n}(
        [[e_{u_0,\dots,u_{n-1}}]^{u_n}]^v
        +
        [[-e_{u_0,\dots,u_{n-1}}]^{u'_n}]^v
        )
        = 0
    \end{equation}
    for any $u_{n}, u'_n\in V_G$ such that $(u_{n-1},u_n),(u_{n-1},u'_n),(u_n,v),(u'_n,v)\in E_G$.
    
    It is possible that the path $e_{u_0,\dots,u_{n-1}}$ does not upper extend by some $u_n \in V_G$ and correspond to a path in a nonzero summand of $x_i$.
    However, all such paths cancel within the decomposition $x_i^{u_{n-1},1}+\cdots+x_i^{u_{n-1},m_i} = \delta_{n,u_{n-1}}^h(x_i)$ prior to upper extension.
    Otherwise, equations~\eqref{eq:Non-DegenerateCase} together with the unique determination of $j$ and
    the case corresponding to equation~\eqref{eq:DegenerateCase}, imply that $\partial^M_{n+1,n+1,n}([x^{n-1}_1+\cdots+x^{n-1}_m]^v) = 0$ as required.
    \qed
\end{proof}

\subsection{Strong connectedness over face multigraphs}\label{sec:Connectedness}

We now give a notion of connectedness for path chains $x\in \Omega_n(G;R)$ arising from their expressions as extensions over a designated class of face  multigraphs. This notion will be stronger than the notion of connectedness from Definition~\ref{def:UpperLowerConnected} and stricter than $x$ being obtained as an extension over a face multigraph with a connected underlying multigraph. The incentive for introducing this new concept will become apparent in the next section while trying to reduce the size of the generating sets for $\Omega_n(G;R)$.

Recall that a \emph{sub-face multigraph} of a face multigraph $F_n$ is a submultigraph of $F_n$ retaining the edge and vertex labels. A face multigraph is \emph{connected} if its underlying multigraph is connected. A \emph{connected component} of a face multigraph is a maximally connected sub-face multigraph.

As motivation, suppose that $x_1,\dots,x_n \in \Omega_n(G;R)$, $x=[x_1+\cdots +x_n]^v$ is an upper extension over a face multigraph $F=F_n(x_1,\dots,x_m)$, and that $F$ is the disjoint union of two sub-face multigraphs $F'=F_n(x_1,\dots,x_k)$ and $F''=F_n(x_{k+1},\dots,x_m)$. Then $x'=[x_1,\cdots,x_k]^v$ is an upper extension over $F'$, $x''=[x_{k+1},\cdots,x_n]^v$ is an upper extension over $F''$, and $x=x'+x''$. This suggests that while looking to construct a manageable generating set for $\Omega_n(G;R)$ consisting of extensions one should restrict to considering only connected face multigraphs. However, the next example shows that in general this is not enough, and that it will be profitable to formulate a stronger notion of connectedness.

\begin{exmp}\label{exam:2Trapezohedron}
    Consider the trapezohedron $\mathbb{T}_2$ of order $2$ from Definition~\ref{def:Trapezohedron}.
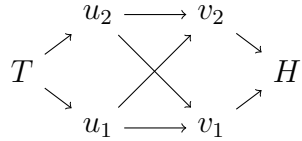
    \begin{center}
        \tikz {
            \node (T) at (0,0) {$T$};
            \node (a1) at (1,-0.75) {$u_1$};
            \node (a2) at (1,0.75) {$u_2$};
            \node (b1) at (2.5,-0.75) {$v_1$};
            \node (b2) at (2.5,0.75) {$v_2$};
            \node (H) at (3.5,0) {$H$};
            \draw[->] (T) -- (a1);
            \draw[->] (T) -- (a2);
            \draw[->] (a1) -- (b1);
            \draw[->] (a1) -- (b2);
            \draw[->] (a2) -- (b1);
            \draw[->] (a2) -- (b2);
            \draw[->] (b1) -- (H);
            \draw[->] (b2) -- (H);
        }
    \end{center}
    Up to sign there are two directed squares 
    \begin{align*}
        S_1
        =
        e_{T,u_1,v_1}
        -
        e_{T,u_2,v_1}
        \in \Omega_2(\mathbb{T}_2;R)
        ,\;\;\;
        S_2
        =
        e_{T,u_1,v_2}
        -
        e_{T,u_2,v_2}
        \in \Omega_2(\mathbb{T}_2;R)
    \end{align*}
    whose tail vertices are $T$.
    The upper extension $[S_1-S_2]^H$ is an upper extension over the $H$-complete face multigraph
    
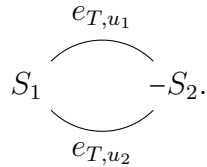
\begin{center}
        \tikz {
            \node (S1) at (0,0) {$S_1$};
            \node (S2) at (2,0) {$-S_2.$};
            \draw[-] (S1) to [out=45,in=135] node[pos=0.5,above] {$e_{T,u_1}$} (S2);
            \draw[-] (S1) to [out=315,in=225] node[pos=0.5,below] {$e_{T,u_2}$} (S2);
        }
    \end{center}
    Hence, $[S_1-S_2]^H \in \Omega_3(G;R)$ by Proposition~\ref{prop:GraphExtentsion} and $[S_1-S_2]^H$ coincides with the trapezohedron element generating $\Omega_3(\mathbb{T}_2;R)$.
    
    On the other hand, consider the following two $H$-complete face multigraphs.
    
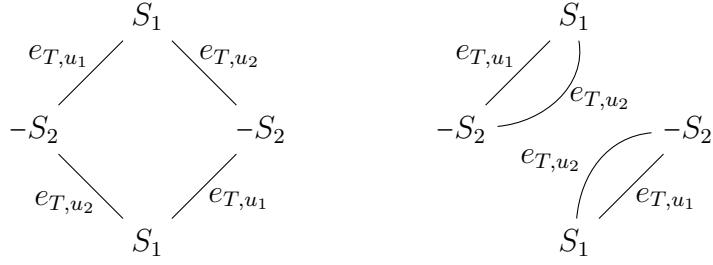
\begin{center}
        \tikz {
            \node (S1) at (0,0) {$S_1$};
            \node (S2) at (1.5,1.5) {$-S_2$};
            \node (S-1) at (0,3) {$S_1$};
            \node (S-2) at (-1.5,1.5) {$-S_2$};
            \draw[-] (S1) -- (S2)  node[pos=0.3,right] {$\:e_{T,u_1}$};
            \draw[-] (S1) -- (S-2)  node[pos=0.25,left] {$e_{T,u_2}$};
            \draw[-] (S2) -- (S-1)  node[pos=0.75,right] {$e_{T,u_2}$};
            \draw[-] (S-2) -- (S-1) node[pos=0.75,left] {$e_{T,u_1}\:$};
        }
        \;\;\;\;\;\;\;\;\;\;\;\;
        \tikz {
            \node (S1) at (0,0) {$S_1$};
            \node (S2) at (1.5,1.5) {$-S_2$};
            \node (S-1) at (0,3) {$S_1$};
            \node (S-2) at (-1.5,1.5) {$-S_2$};
            \draw[-] (S1) -- (S2)  node[pos=0.3,right] {$\:e_{T,u_1}$};
            \draw[-] (S2) to [out=185,in=85] node[pos=0.5,left] {$e_{T,u_2}\:$} (S1);
            \draw[-] (S-1) to [out=280,in=5] node[pos=0.5,right] {$e_{T,u_2}$} (S-2);
            \draw[-] (S-2) -- (S-1)  node[pos=0.75,left] {$e_{T,u_1}\:$};
        }
    \end{center}
    The element $[S_1+S_1-S_2-S_2]^H$ upper extends over both face multigraphs, and disconnectedness of the right-hand face multigraph gives
    \[
        [S_1+S_1-S_2-S_2]^H = 2 [S_1-S_2]^H.
    \]
    
    Hence, for example, $[S_1+S_1-S_2-S_2]^H$ cannot be an element of any basis of $\Omega_3(\mathbb{T}_2;\mathbb{Z})$.
    More generally, a generator of $\Omega_3(\mathbb{T}_t;\mathbb{Z})$ for $t\geq 3$ (c.f. Definition \ref{def:Trapezohedron}) can be similarly obtained as an extension over a face multigraph whose underlying multigraph is a cycle on $t$ vertices.
\end{exmp}

Motivated by the previous example and discussion above, we introduced the following terminology and notion of connectedness on elements of $\Omega_n(G;R)$.

\begin{definition}\label{def:connectedeness}
    Let $E_{n-1} \subseteq \Omega_{n-1}(G;R)$, $E_{n-2} \subseteq \Omega_{n-2}(G;R)$ be subsets and $x\in \Omega_n(G;R)$ an upper (lower) connected element.
    Then an $h(x)$-complete upper ($t(x)$-complete lower) face multigraph $F_n(x_1,\dots,x_m)$ is called  an \emph{upper (lower) $(E_{n-1},E_{n-2})$-inductive structure} on $x$ if the following three conditions hold.
    \begin{enumerate}
        \item 
        Its vertex labels $x_1,\dots,x_2$ are elements of $E_{n-1}$ up to sign,
        \item
        Elements that appear in the ridge decompositions of $F_n(x_1,\dots,x_m)$ are elements of $E_{n-2}$ up to sign,
        \item
        $x = [x_1,\dots,x_2]^{h(x)}$ $(x = [x_1,\dots,x_2]_{t(x)})$.
    \end{enumerate}
    An upper (lower) connected element $x\in \Omega_n(G;R)$ is called \emph{$(E_{n-1},E_{n-2})$-strongly upper (lower) connected} if an upper (lower) $(E_{n-1},E_{n-2})$-inductive structure on $x$ exists and any such $(E_{n-1},E_{n-2})$-inductive structure is connected. We say that a subset $E\subseteq\Omega_n(G;R)$ is \emph{$(E_{n-1},E_{n-2})$-strongly upper (lower) connected} if each of its elements is $(E_{n-1},E_{n-2})$-strongly upper (lower) connected.
\end{definition}
Eventually the sets $E_{n-1},E_{n-2}$ will be taken to be certain bases of $\Omega_{n-1}(G;R),\Omega_{n-1}(G;R)$ and we will look to construct an $(E_{n-1},E_{n-2})$-strongly connected basis for $\Omega_n(G;R)$. For now, the extra generality allowed by the definition will be useful.

We end the section by demonstrating that the definitions of strong connectedness and inductive structures have the following convenient property made use of later in this work.
\begin{lemma}\label{lem:ConnectedSplitting}
    Let $E_{n-1} \subseteq \Omega_{n-1}(G;R)$, $E_{n-2} \subseteq \Omega_{n-2}(G;R)$, and  suppose that $x\in \Omega_n(G;R)$ has an upper (lower) $(E_{n-1},E_{n-2})$-inductive structure.
    Then there exist $(E_{n-1},E_{n-2})$-strongly upper (lower) connected elements $x'_1,\dots,x'_k \in \Omega_n(G;R)$ such that $x=x'_1+\cdots +x'_k$.
\end{lemma}
\begin{proof}
    We prove the case of upper extensions, the proof for lower extensions being similar. 
    Among all
    $(E_{n-1},E_{n-2})$-inductive structures,
    there is an
    $(E_{n-1},E_{n-2})$-inductive structure 
    $F_n$ with all connected components $(E_{n-1},E_{n-2})$-strongly upper connected. Therefore, we may take $x'_1,\dots,x'_k$ to be the upper extensions of the connected components of $F_n$. \qed
\end{proof}

\section{Inductive elements, \texorpdfstring{$\Omega_n(G;R)$}{path chains} generating sets, and bases}\label{sec:InductiveBasis}

We are now ready to define our main objects of study and prove the central results demonstrating their use.
Throughout this section, let $G$ be a digraph
and $R$ a commutative ring with unit, unless stated otherwise. When considering a non-finite $G$ in this section, we note that the axiom of choice is required in the construction of bases as a subset of infinite spanning sets of a vector space.

\subsection{Inductively extending bases}

In this subsection we consider the properties of strongly connected elements of $\Omega_n(G;R)$ which are obtained inductively from elements of arbitrary bases in the previous two dimensions. This preparatory material is used in the next subsection, where inductive elements are defined and the proof of Theorem~\ref{thm:gen} is completed.

We begin by explaining that, in many cases, the structure of strongly connected elements is preserved by a change of coefficients within the same characteristic. Let $B_{n-1}$ be an $R$-basis of $\Omega_{n-1}(G;R)$ and $B_{n-2}$ an $R$-basis of $\Omega_{n-2}(G;R)$, assumed to respect the bigradings within $\Omega_{n-1}^{*,*}(G;R)$ and $\Omega_{n-2}^{*,*}(G;R)$, respectively. Given rings $R,R'$, where $R'$ is an $R$-module, recalling the notation of Section~\ref{sec:Coefficients}, for an integer $k$ let $\widetilde\mu_k$ be the composite
\[
    \widetilde\mu_k:\Omega_k(G;R)\xrightarrow{i_{R,R'}}\Omega_k(G;R)\otimes_RR'\xrightarrow{\mu_k}\Omega_k(G;R').
\]

\begin{lemma}\label{lem:CoeffChangeGraded}
    Let $R=\mathbb{Z}$ or $R=\mathbb{Z}_p$ for some prime $p$ and let $R'$ be a ring of characteristic $0$ or $p$, respectively. If $B$ is a $(B_{n-1},B_{n-2})$-strongly connected basis of $\Omega_n(G;R)$ satisfying the conditions above, then $\widetilde\mu_n(B)$ is a $(\widetilde\mu_{n-1}(B_{n-1}),\widetilde\mu_{n-2}(B_{n-2}))$-strongly connected basis of $\Omega_n(G;R')$.
\end{lemma}
\begin{proof}
    As explained below equation~\eqref{eq:TensorInclusion}, the assumptions imply that $i_{R,R'}$ will map any $R$-basis of $\Omega_k(G;R)$ to an $R'$-basis of $\Omega_k(G;R)\otimes_RR'$. Moreover, Corollaries~\ref{cor:CoeffChange1}~and~\ref{cor:CoeffChange2} imply that $\mu_{n-2}$, $\mu_{n-1}$, and $\mu_n$ are isomorphisms of $R'$-modules. Therefore, $\tilde{\mu}_{n-2}$, $\tilde{\mu}_{n-1}$, and $\tilde{\mu}_{n}$ map $R$-bases in $\Omega_*(G;R)$ to $R'$-bases in $\Omega_*(G;R')$.

    Now, from the commutativity of the diagrams in Lemma~\ref{lem:CoeffChange} it follows that the images of the homomorphisms $\delta^h_{n,v}$, $\delta^t_{n-1,v}$ and $\delta^t_{n,v}$, $\delta^h_{n-1,v}$ are also preserved under $\tilde{\mu}_*$. Therefore, the image of the basis $B$ under $\tilde{\mu}_n$ is a $(\widetilde\mu_{n-1}(B_{n-1}),\widetilde\mu_{n-2}(B_{n-2}))$-strongly connected basis. \qed
\end{proof}

The following result, which will be used extensively in the sequel, states that there are always enough strongly connected elements for certain coefficient rings.

\begin{theorem}\label{thm:BasisExtension}
    Let $n\geq 1$ and let $R=\mathbb{Z}$ or $R=\mathbb{Z}_2$.
    In addition, let $B_{n-1}$ be an $R$-basis of $\Omega_{n-1}(G;R)$ and $B_{n-2}$ an $R$-basis of $\Omega_{n-2}(G;R)$, both of which respect bigradings. 
    Then the upper (lower) $(B_{n-1},B_{n-2})$-strongly connected elements generate $\Omega_n(G;R)$. Moreover, when $R = \mathbb{Z}_2$, there is a basis of $\Omega_{n}(G;\mathbb{Z}_2)$ consisting of $(B_{n-1},B_{n-2})$-strongly connected elements.
\end{theorem}

\begin{proof}
    We prove the statement for upper strongly connected elements, leaving the statement for lower strongly connected elements to the reader. 
    
    As $R$ is a principle ideal domain we may choose a basis $B_n$ of $\Omega_n(G;R)$ that respects the bigrading $\Omega_n^{*,*}(G;R)$. As $B_n$ respects the bigrading $\Omega_n^{*,*}(G;R)$, each $x\in B_n$ is upper connected, with a well defined head vertex $h(x)\in V_G$. Our aim is to construct for any $x\in B_n$ a $(B_{n-1},B_{n-2})$-inductive structure $F_{n-1}(x_1,\dots,x_m)$ in accordance with Definition~\ref{def:connectedeness}. Then either $x$ is $(B_{n-1},B_{n-2})$-strongly upper connected, or if not, we may apply Lemma~\ref{lem:ConnectedSplitting} to obtain $x$ as a linear combination of $(B_{n-1},B_{n-2})$-strongly upper connected elements. Thus it will follow that $(B_{n-1},B_{n-2})$-inductive elements generate $\Omega_n(G;R)$, as required.
    
    We first explain how to choose the $x_i$'s. To begin, notice that for any $v\in V$, we have $\delta_{n,v}^h(x) \in \Omega_{n-1}(G;R)$ by Lemma~\ref{lem:SumFace}. Assuming $\delta_{n,v}^h(x)\neq 0$, then as $B_{n-1}$ is an $R$-basis, there is a smallest integer $t_v >0$ with a unique up to reordering sum decomposition
    \begin{equation}\label{eq:n-1MinimalDecompostion}
        \delta_{n,v}^h(x) = x_1^v+\cdots+x_{t_v}^v
    \end{equation}
    where either $x_k^v \in B_{n-1}$ or $-x_k^v \in B_{n-1}$ for $k=1,\dots,t_v$. Here it is important to realise that in the $R=\mathbb{Z}$ case we allow for repetition of the $x_i^v$'s and expand
    \[
    k\cdot y=\begin{cases}y+y+\cdots+y\; (\text{$k$ times})&k\geq0\\
    -y-y-\cdots-y\; (\text{$-k$ times})&k<0\end{cases}
    \]
    for an integer $k$.
    
    In any case, we may take the multiset $x_1,\dots,x_m$ to be an enumeration of the elements in the finite multiset
    \begin{align*}
        \{ x_k^v \: | \: &
        v\in V_G, \:
        \delta_{n,v}^h(x)\neq 0, \:
        \text{and} \:
        \text{$x_k^v$ appears in decomposition~\eqref{eq:n-1MinimalDecompostion} for} \: k=1,\dots,t_v  \}.
    \end{align*}
    Note that we cannot have $x_i = -x_j$ for any $i,j=1,\dots,m$, as this would imply that paths in $[x_i]^{h(x)}$ and $[x_j]^{h(x)}$ cancel, which is impossible as $x \in B_n \subseteq \Omega_n(G;R)$.
    
    Now, for any $v\in V$ and $i=1,\dots,m$, we have that $\delta_{n,v}^h(x_i) \in \Omega_{n-2}(G;R)$ by Lemma~\ref{lem:SumFace}. In accordance with the conditions for $h(x)$-completeness from Definition~\ref{def:extoffmhg}, suppose that $i= 1,\dots,m$, $\delta_{n,v}^h(x_i) \neq 0$, $(v,h(x_i)) \notin E_G$ and $v \neq h(x)$.
    Then, similarly to equation~\eqref{eq:n-1MinimalDecompostion},
    there is a smallest integer $t_{v}^i >0$ with a unique up to reordering sum decomposition
    \begin{equation}\label{eq:n-2MinimalDecompostion}
        \delta^h_{n,v}(x_i) = x_1^{v,i}+\cdots+x_{t_v^i}^{v,i}
    \end{equation}
    where either $x_k^{v,i} \in B_{n-2}$ or $-x_k^{v,i} \in B_{n-2}$ for $k=1,\dots,t_v^i$.
    Again, as previously
    \begin{equation}\label{eq:n-2Idependence}
        x_k^{v,i} \neq -x_{k'}^{v,i}
    \end{equation}
    for any $k,k'=1,\dots,t_v^i$.
    Moreover, no sub-sequence of $x_1^{v,i},\dots,x_{t_v^i}^{v,i}$ can sum to zero as
    equation~\eqref{eq:n-2Idependence} is obtained as a refinement of a unique linear combination in the members of the $R$-basis $B_{n-2}$.
    
    Let
    \[
        V_x =
        \{ v \in V_G \: | \:
        \delta_{n,v}^h(x_i) \neq 0, \:
        (v,h(x))\notin E_G \:
        \text{and} \:
        v \neq h(x) \:
        \text{for some} \: i = 1,\dots,m \}.
    \]
    As $x \in B_n \subseteq \Omega_n(G;R)$, by Lemma~\ref{lem:NoPositionSwap}, $\partial^M_{n,n,j}(x)=0$ for each $j=1,\dots,n-1$ and hence $\partial^M_{n-1,n-1,j'}\delta^h_{n,v}(x) = 0$ for each $v\in V_G$, and $j'=1,\dots,n-2$.
    Furthermore,
    \[
        0
        =
        \partial_{n-1,n-1}^M\delta^h_{n,v}(x) 
        =
        \sum_{v \in V_x} \sum_{k=1}^{t_v^i} x_k^{v,i}.
    \]
    Hence, using the fact that each $x_k^{v,i}$ is up to sign an element of $B_{n-2}$ and that any additive torsion in $R$ is of order $2$, any $x_k^{v,i}$ can be paired with another $x_{k'}^{v,i'}$ such that
    \begin{equation}\label{eq:OpositSignPairing}
        x_k^{v,i} = -x_{k'}^{v,i'}
    \end{equation}
    for some $i'=1,\dots,m$ and $k'=1,\dots,t_v^{i'}$.
    In addition, as a consequence of equation~\eqref{eq:n-2Idependence}, it must be that $i\neq i'$.
    
    We may now construct a face multigraph $F_{n-1}(x_1,\dots,x_m)$ with vertices $x_1,\dots,x_m$ and all edges described by the labels
    \[
        (\{ x_k^{v,i}, x_{k'}^{v,i'} \}, v)
    \]
    using the pairings and matchings obtained above.
    The face multigraph $F_{n-1}(x_1,\dots,x_m)$ is $h(x)$-complete as we have constructed all required edges for vertices in $V_x$. Therefore, $F_{n-1}(x_1,\dots,x_m)$ is an inductive structure on $x$. \qed
\end{proof}

To complete the proof of Theorem~\ref{thm:gen}, in Corollaries~\ref{cor:FieldGeneratorBasis}~and~\ref{cor:IductiveBasisIntegral} below we take advantage of the fact that by applying Lemma~\ref{lem:CoeffChangeGraded} the $(B_{n-1},B_{n-2})$-strongly connected generating sets constructed in Theorem~\ref{thm:BasisExtension} can be used to obtain explicit strongly connected generating sets and bases in $\Omega_n(G;R)$ for any $R$ of characteristic $0$ or $2$ by an inductive construction beginning with the natural low dimensional bases described in equations~\eqref{eq:Dim0Basis}~and~\eqref{eq:Dim1Basis}.

\subsection{Inductive elements}\label{sec:InductiveElments}

In this subsection we describe an inductive method for constructing generating sets of the path chain modules $\Omega_n(G;R)$. The heavy lifting here is done by Theorem~\ref{thm:BasisExtension}. The main theorems included in the introduction are stated with greater generality in this subsection and proved in full.

\begin{definition}\label{def:InductiveElement}
    Define sets $\bar{E}_{n}^h,\bar{E}_{n}^t \subseteq \Omega_n(G;R)$ for $n\geq-1$ by means of the following inductive construction.
    \begin{enumerate}
        \item 
        Let $\bar{E}^h_{-1} = \bar{E}^t_{-1} = \emptyset$
        and $\bar{E}^h_0 = \bar{E}^t_0 = \{\pm e_v \: | \: v \in V_G \}$.
        \item 
        For $n\geq 1$, the set $\bar{E}^h_n$ ($\bar{E}^t_n$) consists of all $(\bar{E}^h_{n-1},\bar{E}^h_{n-2})$-strongly upper connected elements ($(\bar{E}^t_{n-1},\bar{E}^t_{n-2})$-strongly lower connected elements).
    \end{enumerate}
    An element belonging to $\bar{E}_n^h$ ($\bar{E}^t_n$) for some $n\geq 0$ is called an
    \emph{upper (lower) inductive element}.
    The members of $\bar{E}^h_n$ ($\bar{E}^t_n$) for a fixed $n$ are called \emph{$n$-dimensional upper (lower)
    inductive elements}.
\end{definition}
By construction, all upper and lower inductive elements are connected.
In particular, $h(x)$ and $t(x)$ are well defined for any upper or lower inductive element $x$. Applying parts (1) and (2) of Proposition~\ref{prop:exsumm}, we immediately obtain the following descriptions of inductive elements in low dimensions.

\begin{proposition}\label{prop:LowDimInductiveElements}
    The following properties of upper or lower inductive elements hold.
    \begin{enumerate}
        \item 
        In dimension $1$, inductive elements are given by $\bar{E}^h_1 = \bar{E}^t_1 = \{ \pm e_{u,v} \in \Omega_1(G;R) \}$.
        \item
        In dimension $2$, inductive elements are directed squares, directed triangles or double edges up to sign. \qed
    \end{enumerate}
\end{proposition}

An $(\bar{E}^h_{n-1},\bar{E}^h_{n-2})$-inductive structure ($(\bar{E}^t_{n-1},\bar{E}^t_{n-2})$-inductive structure) on an upper (lower) connected $x\in \Omega_n(G;R)$ is called an \emph{upper (lower) inductive structure on $x$}. 
For simplicity, below we will use the word \emph{inductive} to describe both upper and lower inductive elements and structures.

Theorem~\ref{thm:BasisExtension} now has the following immediate consequence when the coefficients are of characteristic $2$.

\begin{corollary}\label{cor:FieldGeneratorBasis}
    Let $G$ be any digraph, $R$ a ring of characteristic $2$, and $n$ a non-negative integer. Then the $n$-dimensional inductive elements generate $\Omega_n(G;R)$.  Moreover, the $n$-dimensional inductive elements contain a subset that is a basis of $\Omega_n(G;R)$.
\end{corollary}

\begin{proof}
    Using the determination of low-dimensional bases in Section~\ref{sec:LowDimBasis}, the statement of the corollary holds when $n=0$ by construction of inductive elements and when $n=1$ as a consequence of part (1) of Proposition~\ref{prop:LowDimInductiveElements}. The full statement follows for $R=\mathbb{Z}_2$ by induction on $n$ using Theorem~\ref{thm:BasisExtension}, and the general case follows from this using Lemma~\ref{lem:CoeffChangeGraded}.
    \qed
\end{proof}

Theorem~\ref{thm:BasisExtension} also provides the next corollary, which gives the statement corresponding to Corollary~\ref{cor:FieldGeneratorBasis} for the integral and characteristic $0$ cases.

\begin{corollary}\label{cor:IductiveBasisIntegral}
    Let $G$ be any digraph and $R$ a ring of characteristic $0$.
    Then the following statements hold.
    \begin{enumerate}
        \item 
        Inductive elements contain a basis of $\Omega_i(G;R)$ for $i=0,1,2$.
        \item
        The $3$-dimensional inductive elements generate $\Omega_3(G;R)$.
        \item
        When $R$ is a field, a subset of the inductive elements forms a basis of $\Omega_3(G;R)$.
    \end{enumerate}
\end{corollary}

\begin{proof}
     The proof is again by induction. The description of the bases in Section~\ref{sec:LowDimBasis} shows that the statement holds when $n=0$. It holds when $n=1,2$ as a consequence of parts (1) and (2) of Proposition~\ref{prop:LowDimInductiveElements} and Proposition~\ref{prop:Dim2Base}.
    The case when $n=3$ now follows from Theorem~\ref{thm:BasisExtension} when $R=\mathbb{Z}$ and then from Lemma~\ref{lem:CoeffChangeGraded} in general.
    \qed
\end{proof}

In the case of integral coefficients, it is not generally true that a generating set of a finite-dimensional free $\mathbb{Z}$-module may be reduced to a basis.
This fact is the only obstruction to further extending Corollary~\ref{cor:IductiveBasisIntegral} to higher dimensions.
However, we are not currently aware of any examples where a subset of inductive elements do not form a basis.
In general, a basis can still be determined from a generating set of $\Omega_n(G;\mathbb{Z})$ by the computation of a Hermite normal form with respect to the intersection with the usual basis of $\mathcal{A}_n(G;\mathbb{Z})$.
In particular, when the number of paths in the generating set $\Omega_n(G;\mathbb{Z})$ is significantly smaller than the rank of $\mathcal{A}_n(G;\mathbb{Z})$, this approach would be more computationally efficient than computing $\Omega_n(G;\mathbb{Z})$ directly as a submodule of $\mathcal{A}_n(G;\mathbb{Z})$.

\section{Thick \texorpdfstring{$\mathcal{S}_n(G)$}{Sn(G)}-classes as inductive elements}\label{sec:relwithsn(G)classes}

In this section we explain the relationship between inductive elements and the bases of $\Omega_n(G;K)$ constructed by Fu and Ivanov~\cite[Theorem 4.7]{Fu2024} under the restrictions that $G$ has no multisquares and $K$ is a field. These bases were discussed at the end of Section~\ref{sec:LowDimBasis}, and Theorem~\ref{thm:NoMultisquareBasis} summarises them in a manner where basis elements correspond to certain thick $\mathcal{S}_n(G)$-classes. To begin the discussion we first require some further background on the structure of short move graphs. 

Throughout this section, assume that $K$ is a field, unless stated otherwise. All results related to inductive elements will be stated in terms of upper inductive elements, structures, and extensions. The dual results for lower inductive elements also hold and have similar proofs.

\begin{lemma}\label{lem:ShortGraphEdgeRestriction}
    Let $G$ be a digraph without multisquares. Then any vertex of $\mathcal{S}_n(G)$ is incident to at most one edge with any given label.
\end{lemma}

\begin{proof}
    Suppose there is a vertex $u=e_{u_0,\dots,u_n}$ of $\mathcal{S}_n(G)$ that has two incident edges whose labels are both $i$ from some $i=1,\dots,n-1$. As $\mathcal{S}_n(G)$ is a graph, the endpoint vertices $v$, $w$ of the two edges labelled $i$ that are not $u$, are distinct.
    In particular,
    $v$ and $w$ are related to $e_{u_0,\dots,u_n}$ by short moves in position $i$. Therefore,
    $v=e_{u_0,\dots,u_{i-1},v_i,u_{i+1},\dots,u_n}$ and $w=e_{u_0,\dots,u_{i-1},w_i,u_{i+1},\dots,u_n}$, where $u_i$, $v_i$ and $w_i$ are distinct vertices of $\mathcal{S}_n(G)$. Moreover, using the definition of allowed paths we have distinct edges
    \[
        (u_{i-1},u_i),(u_{i-1},v_i),(u_{i-1},w_i),(u_i,u_{i+1}),(v_i,u_{i+1}),(w_i,u_{i+1})\in E_G
    \]
    and by construction of the short moves $(u_{i-1},u_{i+1})\notin E_G$. This implies that $G$ contains a multisquare, which is a contradiction. \qed
\end{proof}

Let $S_n$ be a thick $\mathcal{S}_n(G)$-class with a vertex partition when considering a field $K$ of characteristic other than $2$. For each $i=1,\dots,n-1$ define an equivalence relation $\sim_i$ on the vertices of $S_n$ as that generated by the relation
\begin{equation}\label{eq:ClassVertexRelations}
    u \sim_i v \iff \{ u, v \} \in E_{S_n} \; \text{with label less than} \; i.
\end{equation}
We denote by $B_i(S_n)$ the set of $\sim_i$ equivalence classes. Each $c \in B_i(S_n)$ uniquely determines a thick $\mathcal{S}_{i}(G)$-class $\bar{c}$ by applying the following steps:
\begin{enumerate}[(i)]
    \item
    Take $c$ as a vertex set and add an edge between two of its elements if and only if there is an edge between these elements in $S_n$ that has a label less than $i$.
    \item 
    Replace each elementary $n$-path of $c$ with the elementary $i$-path obtained by removing its last $n-i$ elements.
    \item
    When $K$ has characteristic other than $2$, $\bar{c}$ inherits a partition from $S_n$.
\end{enumerate}
The $\mathcal{S}_{i}(G)$-class $\bar{c}$ is well defined as its graph is connected by the definition of $\sim_i$, and maximally connected in $\mathcal{S}_{i}(G)$ as any edge incident to a vertex of $\bar{c}$ with a label less than $i$ is an edge in $S_n$. Moreover, the thickness of $\bar{c}$ follows from the thickness of $S_n$.
By Theorem~\ref{thm:NoMultisquareBasis}, when $G$ has no multisquares, each $\mathcal{S}_{i}(G)$-class constructed above is contained in a basis of $\Omega_{i}(G;K)$.

\begin{lemma}\label{lem:ShortGraphEquivlenceRestriction}
    Let $G$ be a digraph without multisquares, $S_n$ a thick $\mathcal{S}_n(G)$-class, and $c_1, c_2 \in B_{n-2}(S_n)$ contained in distinct elements of $B_{n-1}(S_n)$. Then, if any element of $c_1$ is joined to an element of $c_2$ in $S_n$ by an edge labelled $n-1$, any other edges of $S_n$ labelled $n-1$ incident to an element of $c_1$ must also be incident to an element of $c_2$.
\end{lemma}

\begin{proof}
    Suppose that $v_1 = e_{u_0^1,\dots,u_n^1}\in c_1$ and $v_2 = e_{u_0^2,\dots,u_n^2}\in c_2$, with $v_1$ and $v_2$ joined in $S_n$ by an edge labelled $n-1$. In addition, suppose also that there is a $c_3\in B_{n-2}(S_n)$, which is distinct from $c_1$ and $c_2$, and contains an element $v_3 = e_{u_0^3,\dots,u_n^3}$ which is joined in $S_n$ to $v'_1\in c_1$ by an edge labelled $n-1$.

    As $v_1$ and $v_2$ are joined
    by an edge labelled $n-1$, we have $u_{n-2}^1 = u_{n-2}^2$, $u_{n}^1 = u_{n}^2$, and $u_{n-1}^1 \neq u_{n-1}^2$. Similarly, as $v'_1$ and $v_3$ are joined
    by an edge labelled $n-1$ and $v_1 \sim_{n-2} v'_1$, we have $u_{n-2}^1 = u_{n-2}^3$, $u_{n}^1 = u_{n}^3$, and $u_{n-1}^1 \neq u_{n-1}^3$. Moreover, $v_1$ and $v'_1$ are joined by a sequence of edges
    with labels less than $n-2$, so if $u_{n-1}^2 = u_{n-1}^3$, then the same sequence of short moves connects $v_2$ and $v_3$ through edges in $S_n$.
    This implies that $v_2 \sim_{n-1} v_3$, which contradicts the assumption that $c_2$ and $c_3$ are distinct equivalence classes of $B_i(S_n)$. Therefore, we must have $u_{n-1}^2 \neq u_{n-1}^3$.
    
    Finally, using the definition of allowed paths we have distinct edges
    \[
        (u^1_{n-2},u^1_{n-1}),(u^1_{n-2},u^2_{n-1}),(u^1_{n-2},u^3_{n-1}),(u^1_{n-1},u^1_{n}),(u^2_{n-1},u^1_{n}),(u^3_{n-1},u^1_{n})\in E_G,
    \]
    and by existence of the paths labelled $n-1$, also $(u^1_{n-2},u^1_{n})\notin E_G$. However, this implies that $G$ contains a multisquare, contradicting the standing assumption on $G$. \qed
\end{proof}

The following theorem relates the inductive elements presented in Section~\ref{sec:InductiveElments} with the bases of Fu and Ivanov from Theorem~\ref{thm:NoMultisquareBasis}.

\begin{theorem}\label{thm:InductiveShortCorrespondence}
    Let $G$ be a digraph without multisquares and $K$ a field. Then the following statements hold.
    \begin{enumerate}
        \item
            When $K$ has characteristic $2$, inductive elements coincide with the basis described in part (1) of Theorem~\ref{thm:NoMultisquareBasis}.
        \item
            When $K$ has characteristic other than $2$, then inductive elements in $\Omega_n(G;K)$ coincide with the basis described in part (2) of Theorem~\ref{thm:NoMultisquareBasis} up to sign.
    \end{enumerate}
\end{theorem}

\begin{proof}
    By Theorem~\ref{thm:NoMultisquareBasis}, for any field $K$ there is a basis of $\Omega_n(G;K)$ with each element determined by a unique thick $\mathcal{S}_n(G)$-class $\{ S_n^j \}_{j\in J}$. We proceed by induction on $n$. 
    
    When $n=0$, the short move graph $\mathcal{S}_0(G)$ has vertices $\{e_v\mid v\in V_G\}$ and no edges. Hence by construction, both inductive elements and thick $\mathcal{S}_0(G)$-classes correspond to the vertices of $G$ up to sign. Similarly, when $n=1$, the short move graph $\mathcal{S}_1(G)$ has vertices $\{e_{u,v}\mid (u,v)\in E_G\}$ and no edges. By part (1) of Proposition~\ref{prop:LowDimInductiveElements}, inductive elements in $\Omega_1(G;K)$ are also identified with edges of $G$ up to sign. 
    
    Now suppose that $n \geq 2$. We will show that when $K$ has characteristic $2$, there is a bijective correspondence between thick $\mathcal{S}_n(G)$-classes and inductive elements. Concurrently, when $K$ does not have characteristic $2$, we will show that there is a bijection between thick $\mathcal{S}_n(G)$-classes with a choice of partition and inductive elements up to sign.

    For any $j\in J$ and $i=1,\dots,n-1$, each $c\in B_i(S_n^j)$ determines a thick $\mathcal{S}_i(G)$-class $\bar{c}$, as explained in the bullet points after equation~\eqref{eq:ClassVertexRelations}. In turn, by induction, $\bar{c}$ determines an inductive element up to sign which we denote by $x_{c} \in \Omega_{i}(G;K)$. More precisely, 
    \begin{equation}\label{eq:InductiveElementFormClass}
        x_c = \sum_{x\in V_{\bar{c}}}
        (-1)^{p_x}x
    \end{equation}
    where $p_x=0$ in the characteristic $2$ cases, and $p_x = 0$ or $p_x = 1$ depending on the partition of $S^j_n$ in the non-characteristic $2$ cases. We intend to take the inductive elements $x_c \in \Omega_{n-1}(G;K)$ provided by each $c\in B_{n-1}(S_n^j)$ as the vertex labels of a face multigraph $F_{n-1}(B_{n-1}(S_n^j))$, with the edges and their labels obtained under an identification with elements of $B_{n-2}(S_n^j)$ as follows.

    Given $c_{n-1} \in B_{n-1}(S_n^j)$, each $e_{u_0,\dots,u_n}\in c_{n-1}$ has the same last two vertices $u_{n-1}, u_{n} \in V_G$. Consider the set $\{c_{n-2} \in B_{n-2}(S_n^j) \: | \: c_{n-2} \subseteq c_{n-1}\}$ and index it $\{c_{n-2}^i\}_{i\in I}$.
    Then for any $i\in I$, each $e_{u_0,\dots,u_n} \in V_{c_{n-2}^i}$ has the same last three vertices $u_{n-2}, u_{n-1}, u_{n} \in V_G$. For $u\in V_G$, let $I_u \subseteq I$ be the largest subset such that for each $i\in I_u$, every $e_{v_0,\dots,v_n}\in  c_{n-2}^i$ satisfies $v_{n-2}=u$.
    Then
    \begin{equation*}
        \sum_{i\in I_u}\;\sum_{x\in V_{\bar{c}^i_{n-2}}}(-1)^{p_x}x=\delta^h_u(x_{c_{n-1}}).
    \end{equation*}
    Therefore, using equation~\eqref{eq:InductiveElementFormClass},
    for each $u\in V_G$ the sum
    \[
        \delta^h_u(x_{c_{n-1}})=\sum_{i\in I_u}x_{c_{n-2}^i}
    \]
    provides a decomposition into inductive elements of the form required in the ridge decomposition of the face multigraph.
    
    In accordance with the conditions for $u_n$-completeness from Definition~\ref{def:extoffmhg}, suppose that there is $i\in I$ for which $u_{n-2} \neq u_n$ and $(u_{n-2},u_n)\notin E_G$ for each $e_{u_0,\dots,u_n} \in V_{c_{n-2}^{i}}$. As $S_n^j$ is thick, every element of $c_{n-2}^{i}$ is incident to an edge labelled $n-1$. Therefore, by Lemma~\ref{lem:ShortGraphEquivlenceRestriction},
    there is a unique $c'_{n-2}\in B_{n-2}(S_n^j)$ such that for each $e_{u_0,\dots,u_n} \in V_{c_{n-2}^{i}}$, the edge labelled $n-1$ incident to $e_{u_0,\dots,u_n}$ is incident also to some $e_{u_0,\dots,u_{n-2},u'_{n-1},u_n} \in V_{c'_{n-2}}$.
    As elements in the two classes are joined by edges labelled $n-1$, it must be the case that $c^{i}_{n-2} \neq c'_{n-2}$ and $c'_{n-2} \nsubseteq c_{n-1}$.
    Hence, $c'_{n-2} \subseteq c'_{n-1} \in B_{n-1}(S^j_n)$ where $c'_{n-1} \neq c_{n-1}$.
    
    Using the existence of the $S_n^j$ edges labelled $n-1$ above, for each $e_{u_0,\dots,u_n}\in V_{c^{i}_{n-2}}$ and $e_{v_0,\dots,v_n}\in V_{c'_{n-2}}$, we have $v_{n-2}=u_{n-2}$, and $v_{n} = u_n$, and by Lemma~\ref{lem:ShortGraphEdgeRestriction} $c'_{n-2}$ is uniquely determined by $c_{n-2}^{i}$.
    Moreover, for each $x=e_{u_0,\dots,u_{n-2}}\in V_{\bar{c}^{i}_{n-2}}$ there is a unique $e_{u_0,\dots,u_{n}} \in V_{c^{i}_{n-2}}$ attached by a unique $S_n^j$ edge labelled $n-1$ to some $e_{u_0,\dots,u_{n-2},u'_{n-1},u_n}\in V_{c^{i}_{n-2}}$.
    This implies that $y = e_{u_0,\dots,u_{n-2}} \in  V_{\bar{c}'_{n-2}}$ with $(-1)^{p_y} y= (-1)^{p_x} x$ in the characteristic $2$ case and $(-1)^{p_y} y= -(-1)^{p_x} x$ otherwise.
    Therefore, letting $\bar{V}'_i = V_{\bar{c}_{n-2}^{i}} \cup V_{\bar{c}'_{n-2}}$ we obtain that
    \[
    \sum_{x\in \bar{V}'_i} (-1)^{p_x}x = 0
    \]
    and the edge in the face multigraph $F_{n-1}(S_n^j)$ labelled $(\{ x_{c_{n-1}}, x_{c'_{n-1}} \}, u_{n-2})$ is well defined. Finally, varying the choices of $c_{n-1} \in B_{n-1}(S^j_n)$ and the value of $i\in I$, we obtain all edges and labels of $F_{n-1}(S_n^j)$ required for a $u_n$-complete face multigraph in the same fashion.

    Conversely, take an inductive element $\bar{x} \in \Omega_n(G;K)$ and fix an inductive structure $\bar{F}_{n-1}$ on $\bar{x}$. Recall that we denote the set of thick $\mathcal{S}_n(G)$-classes by $\{ S_n^j \}_{j\in J}$, and for any $j\in J$ and $i=1,\dots,n-1$, each $c\in B_i(S_n^j)$ determines a thick $\mathcal{S}_i(G)$-class $\bar{c}$, which by induction determines an inductive element $x_{c} \in \Omega_{i}(G;K)$ up to sign.
    
    Let $\bar{x}_1,\dots,\bar{x}_m \in \Omega_{n-1}(G;k)$ be the inductive elements appearing as vertex labels in $\bar{F}_{n-1}$.
    By definition, $\bar{x}_i\neq -\bar{x}_{i'}$ for any $i,i'=1,\dots,m$ and
    \begin{equation}\label{eq:Extesion1}
        \bar{x} = [\bar{x}_1+\cdots+\bar{x}_m]^{h(\bar{x})}.
    \end{equation}
    Meanwhile, by Theorem~\ref{thm:NoMultisquareBasis}, the thick $\mathcal{S}_n(G)$-classes $\{ S_n^j \}_{j\in J}$ provide a basis $\{x_j\}_{j\in J}$ of $\Omega_n(G;K)$. Hence $\bar{x}=\sum_{j\in J}\alpha_j x_j$ for some uniquely determined $\alpha_j \in K$, at most finitely many of which are $0$.
    Since thick $\mathcal{S}_n(G)$-classes are connected components of the short move graph, for any $j\in J$, each $c \in B_{n-1}(S^j_n)$ contains paths of the form $e_{u_0,\dots,u_n} \in \mathcal{A}_{n-1}(G;R)$
    with $u_n = h(\bar{x})$.
    This implies that
    \begin{equation}\label{eq:Extesion2}
        \bar{x}=\left[\sum_{j \in J} \sum_{c\in B_{n-1}(S_n^j)} \alpha_j x_c\right]^{h(\bar{x})}.
    \end{equation}
    Consequently, combining equations~\eqref{eq:Extesion1}~and~\eqref{eq:Extesion2} we obtain
    \begin{equation}\label{eq:EquivlenceOfSums}
        \bar{x}_1 + \cdots + \bar{x}_m
        =
        \sum_{j \in J} \sum_{c\in B_{n-1}(S_n^j)} \alpha_j x_c.
    \end{equation}
    
    The $x_c$ above are linearly independent over $K$. Therefore, since by induction an inductive basis is unique up to sign, the uniqueness of linear expressions in a basis and the integer multiplicity of the $\bar{x}_i$ on the right hand side, each $\alpha_j$ must be an integer.
    
    Recall that we denote the set of thick $\mathcal{S}_n(G)$-classes by $\{ S_n^j \}_{j\in J}$ and let $T= \{ j \in J \: | \: \alpha_j \neq 0 \}$. If $|T| = 1$ and $|\alpha_t| \leq 1$ for each $t\in T$, the proof is complete. Otherwise, $|T| \geq 2$ or $|\alpha_t| \geq 2$ for some $t\in T$. By assumption, each $x_t$ for $t \in T$ is an inductive element with inductive structure $F_{n-1}(S_n^t)$ constructed in the previous part of the proof. Hence, using equation~\ref{eq:Extesion2} and the discussion above it,
    the disjoint union of the face multigraphs $F_{n-1}(S_n^t)$ for $t\in T$ is a disconnected inductive structure for $\bar{x}$.
    However, this contradicts the strong connectedness of $\bar{F}_{n-1}$ that originates from the assumption that $\bar{x}$ is an inductive element.
    \qed
\end{proof}

By Theorem~\ref{cor:FieldGeneratorBasis}, inductive elements generate $\Omega_n(G;\mathbb{Z}_2)$. However, although inductive elements do not necessary generate $\Omega_n(G;K)$ for any field $K$ and $n\geq 0$, Theorem~\ref{thm:InductiveShortCorrespondence} implies that they do when $G$ contains no multisquares and moreover form a basis. Furthermore, the correspondence between inductive elements and thick $\mathcal{S}_n(G)$-classes yields the following generalisation of Theorem~\ref{thm:NoMultisquareBasis}.

\begin{corollary}\label{cor:InductiveShortCorrespondenceForRings}
    Theorem~\ref{thm:NoMultisquareBasis} holds as stated also when $K$ is any commutative ring of characteristic $0$ or $p$ for some prime $p$.
\end{corollary}

\begin{proof}
    By Theorem~\ref{thm:InductiveShortCorrespondence}, the thick $\mathcal{S}_n(G)$-classes determine a basis of $\Omega_n(G;K)$ and are identified uniquely with inductive elements. By Lemma~\ref{lem:CoeffChangeGraded}, inductive elements with coefficients in characteristic $0$ or $p$ can be determined by setting $K=\mathbb{Z}$ or $K=\mathbb{Z}_p$, respectively. \qed
\end{proof}

In Section~\ref{sec:Coefficients} we discussed change of coefficients maps $\mu_n\colon\Omega_n(G;R)\otimes_RR'\rightarrow\Omega_n(G;R')$ and introduced also maps $\bar{\mu}_n\colon H^P_n(G;R)\otimes_RR'\rightarrow H^P_n(G;R')$. Having good descriptions of the path chains in case $G$ contains no multisquares allows for more to be said about these maps than could be previously.

\begin{proposition}\label{prop:univcoeffs}
Suppose that $G$ contains no multisquares and $R$ is a ring of odd prime characteristic. Then $\mu_*\colon\Omega_*(G;\mathbb{Z})\otimes_{\mathbb{Z}} R \rightarrow\Omega_*(G;R)$ is a chain isomorphism and for each $n\geq0$ there is an exact sequence
\[
0\rightarrow H_n^P(G;\mathbb{Z})\otimes R \xrightarrow{\bar{\mu}_n} H_n^P(G;R)\rightarrow \mathrm{Tor}(H_{n-1}^P(G;\mathbb{Z}), R)\rightarrow 0.
\]
This sequence is natural in maps between digraphs without multisquares. The sequence splits, but this splitting is not natural.
\end{proposition}

\begin{proof}
By Theorem~\ref{thm:NoMultisquareBasis} and Corollary~\ref{cor:InductiveShortCorrespondenceForRings}, both $\Omega^P_n(G;\mathbb{Z})$ and $\Omega^P_n(G;R)$ are generated by the same set of inductive elements corresponding to the thick $\mathcal{S}_n(G)$-classes. It follows from this that the map $\mu_{*}\colon\Omega_*(G;\mathbb{Z})\otimes R \rightarrow\Omega_*(G;R)$ is bijective and hence an isomorphism of chain complexes of $R$-modules since it commutes with the differentials.

Now, using $\mu_{*}$ to identify the two chain complexes we have $H^P_*(G;R)\cong H_*(\Omega_*(G;\mathbb{Z})\otimes R;\partial^P\otimes1)$. Thus, up to an identification of the middle module, the exact sequence is just the universal coefficient theorem given in \cite[Corollary 7.56, pg. 450]{Rotman2009}. \qed
\end{proof}

Proposition~\ref{prop:univcoeffs} cannot be extended to produce a universal coefficient sequences for digraphs containing multisquares. Examples given by the authors~\cite[Example 6.2]{BurfittCutler2024} show explicitly that the universal coefficient theorem fails in dimension $3$, and by taking suspensions of these digraphs (see \cite[Definition 2.12 and Theorem 4.13]{Grigoryan2018})
we obtain digraphs for which the universal coefficient theorem fails in any given dimension $\geq3$. Neither can Proposition~\ref{prop:univcoeffs} be generalised to include rings of characteristic $2$. A counterexample to this is found in~\cite[\S 5.4]{Fu2024} and discussed in some detail in Example~\ref{ex:UCTfails}. Theorem~\ref{th:Dim3Char0Char2invareince}, which is stated and proved in the next section, gives the best result one can generally expect for rings of characteristic $2$.

\section{Explicit description of \texorpdfstring{$\Omega_3(G;R)$}{dimension 3 path chains}}\label{sec:Dim3}

We now use the theory of inductive elements to give a set of generators of $\Omega_3(G;R)$ for an arbitrary digraph $G$ and ring $R$ of characteristic $0$ or $2$ (e.g. $R=\mathbb{Z}$ or $R=\mathbb{Z}_2$). This explicit description of $\Omega_3(G;R)$ generalises Grigor'yan's work~\cite[Theorem 2.10]{Grigoryan2022}, where a basis for $\Omega_3(G;K)$ was given when $G$ contains no multisquares or double edges and $K$ is a field (see Theorem~\ref{thm:Dim3BasisNoDoubleNoMulti}). Thus the outcome of this section is a complete solution to~\cite[Problem 2.11]{Grigoryan2022} for the considered coefficients.

We start in Definitions~\ref{def:AsymmetricLinearSequences},~\ref{def:UpperLinearElements},~\ref{def:LowerLinearElements} by introducing three infinite families of digraphs. These digraphs play a similar role to that played by the trapezohedron digraphs discussed earlier, and come along with distinguished $3$-dimensional path chains which are singled out in Definition~\ref{def:LinearEelmetns}. Together with the trapezohedra, we call these new digraphs \emph{$3$-path sequences}.

The $3$-path sequences are the domains of the \emph{$3$-path structure maps}, which are certain structure-preserving digraph maps introduced in 
Definition~\ref{def:3PathStructureMap}.
These maps become the central objects of study, and we use them to express the main results of the section, Theorem~\ref{thm:Omega3Explicitly} and Corollaries~\ref{cor:3PathStructureGeneration}~and~\ref{cor:Omega3Explicitly}, which are arrived at using the prior technical results in Theorem~\ref{thm:SructureMapsAreInductive}, Theorem~\ref{thm:3PathStructureConstruction}, and Lemma~\ref{lem:3PathStructHitsDirTriAndDblEdg}.
We end the section with Theorem~\ref{th:Dim3Char0Char2invareince}, which is a low-dimensional universal coefficient theorem for arbitrary digraphs and characteristic $2$ rings.

Throughout the section, we state and prove all results related to inductive elements in terms of upper inductive elements, structures, and extensions. The dual results for lower inductive elements can be easily formulated and proved in the same way.

\begin{definition}\label{def:AsymmetricLinearSequences}
    For each $t=1,2,\dots$ let $\mathbb{L}_{t}$ be the digraph with vertices
    \[
        V_{\mathbb{L}_{t}} =
        \{ T,u_1,\dots,u_t,v_1,\dots,v_t,H \}
    \]
    and edges
    \[
        T \to u_i,\;\;\;
        T \to v_1,\;\;\;
        u_i \to v_i, \;\;\;
        u_i \to v_{i+1}
        \; \text{when} \;
        i+1 \leq t, \;\;\;
        v_i \to H,
        \;\;\; \text{and} \;\;\;
        u_t \to H
    \]
    for $i=1,\dots,t$. Define the \emph{core} subdigraph $\overline{\mathbb{L}}_{t}$ of $\mathbb{L}_{t}$ to be that given by
    \[
        V_{\overline{\mathbb{L}}_{t}} = V_{\mathbb{L}_{t}}
        \;\;\; \text{and} \;\;\;
        E_{\overline{\mathbb{L}}_{t}} = E_{\mathbb{L}_{t}} \setminus \{ (T,v_1), (u_{t}, H) \}.
    \]
\end{definition}
For example, the digraphs $\mathbb{L}_1$, $\mathbb{L}_2$, $\mathbb{L}_3$, and $\mathbb{L}_4$ are as follows.
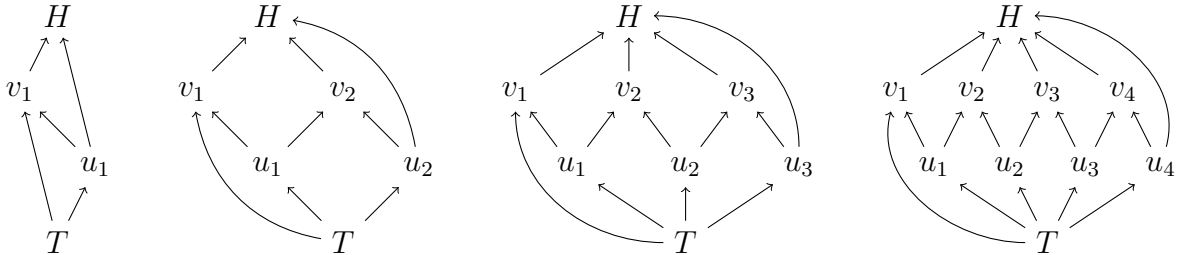
\begin{center}
    \tikz {
        \node (T) at (0,0) {$T$};
        \node (u1) at (0.5,1) {$u_1$};
        \node (v1) at (-0.5,2) {$v_1$};
        \node (H) at (0,3) {$H$};
        \draw[->] (T) -- (u1);
        \draw[->] (T) -- (v1);
        \draw[->] (u1) -- (v1);
        \draw[->] (v1) -- (H);
        \draw[->] (u1) -- (H);
    }
    \;\;\;
    \tikz {
        \node (T) at (2,0) {$T$};
        \node (u1) at (1,1) {$u_1$};
        \node (u2) at (3,1) {$u_2$};
        \node (v1) at (0,2) {$v_1$};
        \node (v2) at (2,2) {$v_2$};
        \node (H) at (1,3) {$H$};
        \draw[->] (T) -- (u1);
        \draw[->] (T) -- (u2);
        \draw[->] (T) to [out=170,in=280] (v1);
        \draw[->] (u1) -- (v1);
        \draw[->] (u1) -- (v2);
        \draw[->] (u2) -- (v2);
        \draw[->] (v1) -- (H);
        \draw[->] (v2) -- (H);
        \draw[->] (u2) to [out=100,in=350] (H);
    }
    \;\;\;
    \tikz {
        \node (T) at (2.25,0) {$T$};
        \node (u1) at (0.75,1) {$u_1$};
        \node (u2) at (2.25,1) {$u_2$};
        \node (u3) at (3.75,1) {$u_3$};
        \node (v1) at (0,2) {$v_1$};
        \node (v2) at (1.5,2) {$v_2$};
        \node (v3) at (3,2) {$v_3$};
        \node (H) at (1.5,3) {$H$};
        \draw[->] (T) -- (u1);
        \draw[->] (T) -- (u2);
        \draw[->] (T) -- (u3);
        \draw[->] (T) to [out=180,in=270] (v1);
        \draw[->] (u1) -- (v1);
        \draw[->] (u1) -- (v2);
        \draw[->] (u2) -- (v2);
        \draw[->] (u2) -- (v3);
        \draw[->] (u3) -- (v3);
        \draw[->] (v1) -- (H);
        \draw[->] (v2) -- (H);
        \draw[->] (v3) -- (H);
        \draw[->] (u3) to [out=90,in=0] (H);
    }
    \;\;\;
    \tikz {
        \node (T) at (2,0) {$T$};
        \node (u1) at (0.5,1) {$u_1$};
        \node (u2) at (1.5,1) {$u_2$};
        \node (u3) at (2.5,1) {$u_3$};
        \node (u4) at (3.5,1) {$u_4$};
        \node (v1) at (0,2) {$v_1$};
        \node (v2) at (1,2) {$v_2$};
        \node (v3) at (2,2) {$v_3$};
        \node (v4) at (3,2) {$v_4$};
        \node (H) at (1.5,3) {$H$};
        \draw[->] (T) -- (u1);
        \draw[->] (T) -- (u2);
        \draw[->] (T) -- (u3);
        \draw[->] (T) -- (u4);
        \draw[->] (T) to [out=180,in=255] (v1);
        \draw[->] (u1) -- (v1);
        \draw[->] (u1) -- (v2);
        \draw[->] (u2) -- (v2);
        \draw[->] (u2) -- (v3);
        \draw[->] (u3) -- (v3);
        \draw[->] (u3) -- (v4);
        \draw[->] (u4) -- (v4);
        \draw[->] (v1) -- (H);
        \draw[->] (v2) -- (H);
        \draw[->] (v3) -- (H);
        \draw[->] (v4) -- (H);
        \draw[->] (u4) to [out=75,in=0] (H);
    }
\end{center}
The core subdigraphs are obtained by deleting the two edges $T\rightarrow v_1$ and $u_t\rightarrow H$.

\begin{definition}\label{def:UpperLinearElements}
    For each $t=1,2,\dots$ let $\mathbb{L}^u_{t}$ be the digraph with vertices
    \[
        V_{\mathbb{L}_{t}^u} =
        \{ T,u_1,\dots,u_{t+1},v_1,\dots,v_t,H \}
    \]
    and edges
    \[
        T \to u_i,\;\;\;
        u_j \to v_j, \;\;\;
        u_{j+1} \to v_j, \;\;\;
        v_j \to H, \;\;\;
        u_1 \to H
        \;\;\; \text{and} \;\;\;
        u_{t+1} \to H
    \]
    for $i=1,\dots,t+1$ and $j=1,\dots,t$. Define the \emph{core} subdigraph $\overline{\mathbb{L}}{\vphantom{\mathbb{L}}}_{t}^u$ of $\mathbb{L}^u_{t}$ to be that given by
    \[
        V_{\overline{\mathbb{L}}{\vphantom{\mathbb{L}}}_{t}^u} = V_{\mathbb{L}^u_{t}}
        \;\;\; \text{and} \;\;\;
        E_{\overline{\mathbb{L}}{\vphantom{\mathbb{L}}}_{t}^u} = E_{\mathbb{L}^u_{t}} \setminus \{ (u_1, H), (u_{t+1}, H) \}.
    \]
    \end{definition}

The digraphs $\mathbb{L}^u_1$, $\mathbb{L}^u_2$ $\mathbb{L}^u_3$, and $\mathbb{L}^u_4$ are as follows.
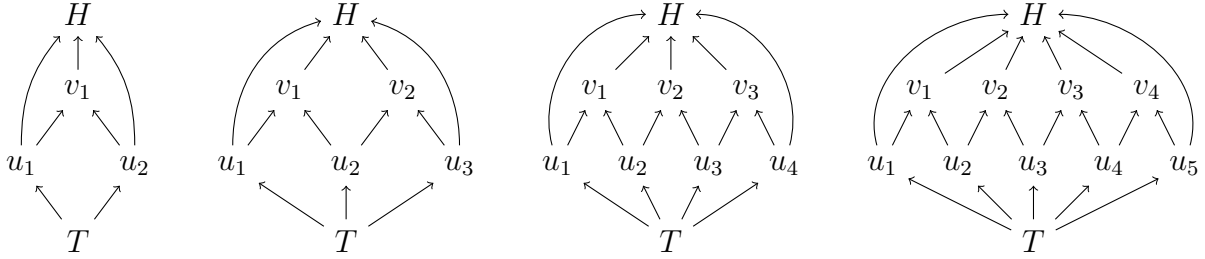
\begin{center}
    \tikz {
        \node (T) at (0,0) {$T$};
        \node (u1) at (-0.75,1) {$u_1$};
        \node (u2) at (0.75,1) {$u_2$};
        \node (v1) at (0,2) {$v_1$};
        \node (H) at (0,3) {$H$};
        \draw[->] (T) -- (u1);
        \draw[->] (T) -- (u2);
        \draw[->] (u1) -- (v1);
        \draw[->] (u2) -- (v1);
        \draw[->] (v1) -- (H);
        \draw[->] (u1) to [out=90,in=230] (H);
        \draw[->] (u2) to [out=90,in=310] (H);
    }
    \;\;\;
    \tikz {
        \node (T) at (2.25,0) {$T$};
        \node (u1) at (0.75,1) {$u_1$};
        \node (u2) at (2.25,1) {$u_2$};
        \node (u3) at (3.75,1) {$u_3$};
        \node (v1) at (1.5,2) {$v_1$};
        \node (v2) at (3,2) {$v_2$};
        \node (H) at (2.25,3) {$H$};
        \draw[->] (T) -- (u1);
        \draw[->] (T) -- (u2);
        \draw[->] (T) -- (u3);
        \draw[->] (u1) -- (v1);
        \draw[->] (u2) -- (v1);
        \draw[->] (u2) -- (v2);
        \draw[->] (u3) -- (v2);
        \draw[->] (v1) -- (H);
        \draw[->] (v2) -- (H);
        \draw[->] (u1) to [out=90,in=195] (H);
        \draw[->] (u3) to [out=90,in=345] (H);
    }
    \;\;\;
    \tikz {
        \node (T) at (2,0) {$T$};
        \node (u1) at (0.5,1) {$u_1$};
        \node (u2) at (1.5,1) {$u_2$};
        \node (u3) at (2.5,1) {$u_3$};
        \node (u4) at (3.5,1) {$u_4$};
        \node (v1) at (1,2) {$v_1$};
        \node (v2) at (2,2) {$v_2$};
        \node (v3) at (3,2) {$v_3$};
        \node (H) at (2,3) {$H$};
        \draw[->] (T) -- (u1);
        \draw[->] (T) -- (u2);
        \draw[->] (T) -- (u3);
        \draw[->] (T) -- (u4);
        \draw[->] (u1) -- (v1);
        \draw[->] (u2) -- (v1);
        \draw[->] (u2) -- (v2);
        \draw[->] (u3) -- (v2);
        \draw[->] (u3) -- (v3);
        \draw[->] (u4) -- (v3);
        \draw[->] (v1) -- (H);
        \draw[->] (v2) -- (H);
        \draw[->] (v3) -- (H);
        \draw[->] (u1) to [out=105,in=180] (H);
        \draw[->] (u4) to [out=75,in=0] (H);
    }
    \;\;\;
    \tikz {
        \node (T) at (1.5,0) {$T$};
        \node (u1) at (-0.5,1) {$u_1$};
        \node (u2) at (0.5,1) {$u_2$};
        \node (u3) at (1.5,1) {$u_3$};
        \node (u4) at (2.5,1) {$u_4$};
        \node (u5) at (3.5,1) {$u_5$};
        \node (v1) at (0,2) {$v_1$};
        \node (v2) at (1,2) {$v_2$};
        \node (v3) at (2,2) {$v_3$};
        \node (v4) at (3,2) {$v_4$};
        \node (H) at (1.5,3) {$H$};
        \draw[->] (T) -- (u1);
        \draw[->] (T) -- (u2);
        \draw[->] (T) -- (u3);
        \draw[->] (T) -- (u4);
        \draw[->] (T) -- (u5);
        \draw[->] (u1) -- (v1);
        \draw[->] (u2) -- (v1);
        \draw[->] (u2) -- (v2);
        \draw[->] (u3) -- (v2);
        \draw[->] (u3) -- (v3);
        \draw[->] (u4) -- (v3);
        \draw[->] (u4) -- (v4);
        \draw[->] (u5) -- (v4);
        \draw[->] (v1) -- (H);
        \draw[->] (v2) -- (H);
        \draw[->] (v3) -- (H);
        \draw[->] (v4) -- (H);
        \draw[->] (u5) to [out=75,in=0] (H);
        \draw[->] (u1) to [out=105,in=180] (H);
    }
\end{center}
The core subdigraphs are obtained by deleting the two edges $u_1 \rightarrow H$ and $u_{t+1}\rightarrow H$.

Finally, we introduce one last sequence of digraphs.
\begin{definition}\label{def:LowerLinearElements}
    For each $t=1,2,\dots$ let $\mathbb{L}^l_{t}$ be the digraph with vertices
    \[
        V_{\mathbb{L}_{t}^l} =
        \{ T,u_1,\dots,u_t,v_1,\dots,v_{t+1},H \}
    \]
    and edges
    \[
        T \to u_i,\;\;\;
        T \to v_1,\;\;\;
        T \to v_{t+1},\;\;\;
        u_i \to v_i, \;\;\;
        u_i \to v_{i+1}
        \;\;\; \text{and} \;\;\;
        v_j \to H
    \]
    for $i=1,\dots,t$ and $j=1,\dots,t+1$. Define the \emph{core} subdigraph $\overline{\mathbb{L}}{\vphantom{\mathbb{L}}}_{t}^l$ of $\mathbb{L}^l_{t}$ to be that given by
    \[
        V_{\overline{\mathbb{L}}{\vphantom{\mathbb{L}}}_{t}^l} = V_{\mathbb{L}^l_{t}}
        \;\;\; \text{and} \;\;\;
        E_{\overline{\mathbb{L}}{\vphantom{\mathbb{L}}}_{t}^l} = E_{\mathbb{L}^l_{t}} \setminus \{ (T,v_1), (T,v_{t+1}) \}.
    \]
\end{definition}
The digraphs $\mathbb{L}^l_1$, $\mathbb{L}^l_2$ $\mathbb{L}^l_3$, and $\mathbb{L}^l_4$ are as follows.
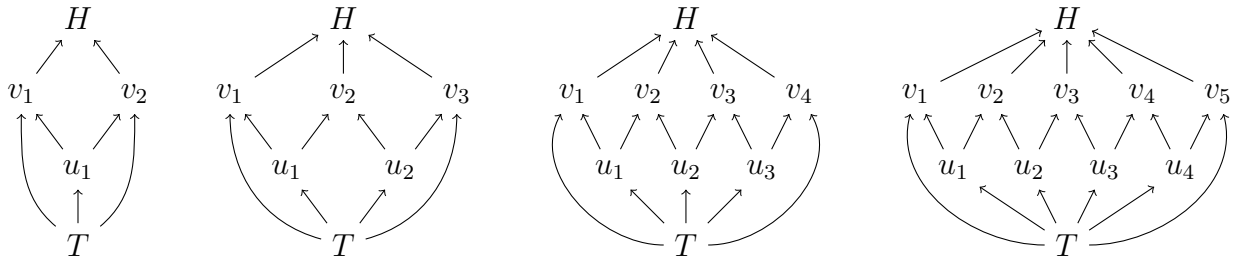
\begin{center}
\tikz {
        \node (T) at (0,3) {$H$};
        \node (u1) at (-0.75,2) {$v_1$};
        \node (u2) at (0.75,2) {$v_2$};
        \node (v1) at (0,1) {$u_1$};
        \node (H) at (0,0) {$T$};
        \draw[->] (u1) -- (T);
        \draw[->] (u2) -- (T);
        \draw[->] (v1) -- (u1);
        \draw[->] (v1) -- (u2);
        \draw[->] (H) -- (v1);
        \draw[->] (H) to [out=140,in=270] (u1);
        \draw[->] (H) to [out=40,in=270] (u2);
    }
    \;\;\;
\tikz {
        \node (T) at (2.25,3) {$H$};
        \node (u1) at (0.75,2) {$v_1$};
        \node (u2) at (2.25,2) {$v_2$};
        \node (u3) at (3.75,2) {$v_3$};
        \node (v1) at (1.5,1) {$u_1$};
        \node (v2) at (3,1) {$u_2$};
        \node (H) at (2.25,0) {$T$};
        \draw[->] (u1) -- (T);
        \draw[->] (u2) -- (T);
        \draw[->] (u3) -- (T);
        \draw[->] (v1) -- (u1);
        \draw[->] (v1) -- (u2);
        \draw[->] (v2) -- (u2);
        \draw[->] (v2) -- (u3);
        \draw[->] (H) -- (v1);
        \draw[->] (H) -- (v2);
        \draw[->] (H) to [out=165,in=270] (u1);
        \draw[->] (H) to [out=15,in=270] (u3);
    }
    \;\;\;
 \tikz {
        \node (T) at (2,0) {$T$};
        \node (u1) at (1,1) {$u_1$};
        \node (u2) at (2,1) {$u_2$};
        \node (u3) at (3,1) {$u_3$};
        \node (v1) at (0.5,2) {$v_1$};
        \node (v2) at (1.5,2) {$v_2$};
        \node (v3) at (2.5,2) {$v_3$};
        \node (v4) at (3.5,2) {$v_4$};
        \node (H) at (2,3) {$H$};
        \draw[->] (T) -- (u1);
        \draw[->] (T) -- (u2);
        \draw[->] (T) -- (u3);
        \draw[->] (u1) -- (v1);
        \draw[->] (u1) -- (v2);
        \draw[->] (u2) -- (v2);
        \draw[->] (u2) -- (v3);
        \draw[->] (u3) -- (v3);
        \draw[->] (u3) -- (v4);
        \draw[->] (v1) -- (H);
        \draw[->] (v2) -- (H);
        \draw[->] (v3) -- (H);
        \draw[->] (v4) -- (H);
        \draw[->] (T) to [out=180,in=240] (v1);
        \draw[->] (T) to [out=0,in=300] (v4);
    }
    \;\;\;
    \tikz {
        \node (T) at (1.5,3) {$H$};
        \node (u1) at (-0.5,2) {$v_1$};
        \node (u2) at (0.5,2) {$v_2$};
        \node (u3) at (1.5,2) {$v_3$};
        \node (u4) at (2.5,2) {$v_4$};
        \node (u5) at (3.5,2) {$v_5$};
        \node (v1) at (0,1) {$u_1$};
        \node (v2) at (1,1) {$u_2$};
        \node (v3) at (2,1) {$u_3$};
        \node (v4) at (3,1) {$u_4$};
        \node (H) at (1.5,0) {$T$};
        \draw[->] (u1) -- (T);
        \draw[->] (u2) -- (T);
        \draw[->] (u3) -- (T);
        \draw[->] (u4) -- (T);
        \draw[->] (u5) -- (T);
        \draw[->] (v1) -- (u1);
        \draw[->] (v1) -- (u2);
        \draw[->] (v2) -- (u2);
        \draw[->] (v2) -- (u3);
        \draw[->] (v3) -- (u3);
        \draw[->] (v3) -- (u4);
        \draw[->] (v4) -- (u4);
        \draw[->] (v4) -- (u5);
        \draw[->] (H) -- (v1);
        \draw[->] (H) -- (v2);
        \draw[->] (H) -- (v3);
        \draw[->] (H) -- (v4);
        \draw[->] (H) to [out=0,in=285] (u5);
        \draw[->] (H) to [out=180,in=255] (u1);
    }
\end{center}
The core subdigraphs are obtained by deleting the two edges $T \rightarrow v_1$ and $T \rightarrow v_{t+1}$. In general, $\mathbb{L}^l_t$ is obtained from $\mathbb{L}^u_t$ by reversing the arrows, relabelling each $u_i$ with $v_i$, relabelling each $v_i$ with $u_i$, and finally swapping the $T$ and $H$ labels.

The digraphs introduced above should be compared with the trapezohedron digraphs $\mathbb{T}_t$, which we recall from Definition~\eqref{def:Trapezohedron}. The trapezohedron element $T_t$ produced in equation~\eqref{eq:TrapezohedronElement} is the canonical generator of $\Omega_3(\mathbb{T}_t;R)\cong R$, and we now produce path chains with similar properties for the new digraphs.

\begin{definition}\label{def:LinearEelmetns}
    The elements
    \begin{align*}
        L_t &= \sum_{i=1}^{t} e_{T,u_i,v_i,H} - \sum_{i=1}^{t-1} e_{T,u_i,v_{i+1},H} \in \mathcal{A}_3(\mathbb{L}_t;R),
        \\
        L_t^u &= \sum_{i=1}^{t} e_{T,u_i,v_i,H} - e_{T,u_{i+1},v_i,H} \in \mathcal{A}_3(\mathbb{L}_t^u;R),
        \\
        L_t^l &= \sum_{i=1}^{t} e_{T,u_i,v_i,H} - e_{T,u_i,v_{i+1},H} \in \mathcal{A}_3(\mathbb{L}_t^l;R)
    \end{align*}
    for $t=1,2,\dots$ are called the \emph{asymmetric}, \emph{upper symmetric}, and \emph{lower symmetric} elements, respectively. 
    Collectively, $T_{t'}$, $L_t$, $L_t^l$, and $L_t^u$
    are called \emph{$3$-path elements} for $t=1,2,3\dots$ and $t'=2,3,\dots$.
\end{definition}

We also want to include the trapezohedra digraphs in the considerations below. In the case of the trapezohedra, for $t'=2,3,\dots$ we consider the core subdigraph $\overline{\mathbb{T}}_{t'}$ to be the whole digraph $\mathbb{T}_{t'}$.
\begin{definition}
The digraphs $\mathbb{L}_t$, $\mathbb{L}_t^u$, $\mathbb{L}_t^l$ for $t=1,2,\dots$, and $\mathbb{T}_{t'}$ for $t'=2,3,\dots$ are collectively called \emph{$3$-path sequence} digraphs.
\end{definition}

Similar to Proposition~\ref{prop:Trapezohedron} for trapezohedron elements, the other $3$-path sequences and $3$-path elements have the following properties.

\begin{proposition}\label{prop:OtherSpecialDim3Generators}
    Let $G = \mathbb{L}_t$, $\mathbb{L}_t^u$, or $\mathbb{L}_t^l$ for some $t=1,2,\dots$. Then the $R$-module $\Omega_3^{T,H}(G;R)$ is free of rank $1$ and is generated by the corresponding $3$-path element $L_t$, $L_t^u$, or $L_t^l$. Moreover,
    \[
        H_n^P(G;R)
        \cong
        \begin{cases}
            R & \text{if} \; n = 0 \\
            0 & \text{otherwise.}
        \end{cases}
    \]
\end{proposition}

\begin{proof}
    We give a detailed proof for $G = \mathbb{L}_t$ and leave it to the reader to consider the other cases, which are proved similarly.
    
    To begin, note that $\Omega_0(\mathbb{L}_t;R)$ has a basis of inductive elements consisting of $e_T$, $e_{u_i}$, $e_{v_i}$, and $e_H$, and by part (1) of Proposition~\ref{prop:LowDimInductiveElements} $\Omega_1(\mathbb{L}_t;R)$ has a basis of inductive elements consisting of $e_{T,u_i}$, $e_{T,v_1}$, $e_{u_i, v_i}$, $e_{u_j, v_{j+1}}$, $e_{v_i,H}$, and $e_{u_t,H}$ for each $i=1,\dots,t$ and $j=1,\dots,t-1$.
    Furthermore, $\mathbb{L}_t$ contains directed triangles
    \[
        T_l = e_{T,u_1,v_1}, \;\;\; \text{and} \;\;\; T_u = e_{u_t,v_t,H},
    \]
    and directed squares
    \[
        S_i^l = e_{T,u_{i+1},v_{i+1}} - e_{T,u_i,v_{i+1}},
        \;\;\;
        S_i^u = e_{u_i,v_i,H} - e_{u_i,v_{i+1},H},
        \;\;\; \text{and} \;\;\;
        S = e_{T,u_t,H} - e_{T,v_1,H},
    \]
    for $i=1,\dots,t-1$, but no double edges. Since $\mathbb{L}_t$ contains no multisquares, by Propositions~\ref{prop:Dim2Base}~and~\ref{prop:LowDimInductiveElements}, the generators listed above form an inductive basis of $\Omega_2(\mathbb{L}_t;R)$.

    Since the digraph $\mathbb{L}_t$ is connected, we have $H_0^P(\mathbb{L}_t;R)\cong R$. Moreover, $H^P_1(\mathbb{L}_t;R) = 0$, as a basis of the cycles in $\Omega_1(\mathbb{L}_t;R)$ coincides precisely with the boundaries of directed squares and directed triangles $S_i^l$, $S_i^u$, $S$, $T_l$, and $T_u$.

    Turning to $\Omega_3(\mathbb{L}_t;R)$, we claim that it is free of rank $1$ generated by the element $L_T$. Indeed, $L_T$ is obtained as an upper extension over the face multigraph
    \begin{center}
        \tikz {
            \node (1) at (0,0) {$T_l$};
            \node (2) at (3,0) {$S_1^l$};
            \node (3) at (6,0) {$\;\;\;\cdots\;\;\;$};
            \node (4) at (9,0) {$S_{t-1}^l$.};
            \draw[-] (1) -- (2) node[midway,above] {$e_{T,u_1}$};
            \draw[-] (2) -- (3) node[midway,above] {$e_{T,u_2}$};
            \draw[-] (3) -- (4) node[midway,above] {$e_{T,u_{t-1}}$};
        }
    \end{center}
    By Corollary~\ref{cor:FieldGeneratorBasis}, up to multiplication by a unit of $R$, the element $L_t$ is a unique generator
    of $\Omega_3(\mathbb{L}_t;R)$, as no other upper inductive structures can be formed using the inductive bases of $\Omega_2(\mathbb{L}_t;R)$ and $\Omega_1(\mathbb{L}_t;R)$ described above.
    As the maximal path length in $\mathbb{L}_t$ is $3$, with all paths beginning at $T$ and ending at $H$, we have $\Omega_n(\mathbb{L}_t;R) = 0$
    for $n \geq 4$ and $\Omega_3^{T,H}(\mathbb{L}_t;R) = \Omega_3(\mathbb{L}_t;R)$.

    Since $\partial_3^P(L_t) = T_u - T_l + S + \sum_{i=1}^{t-1} S_i^u - \sum_{i=1}^{t-1} S_i^l$ is nonzero, we have $H^P_3(\mathbb{L}_t;R) = 0$. Moreover, we claim that $\partial_3^P(L_t)$ generates $\ker(\partial^P_2\colon\Omega_2(\mathbb{L}_t;R)\rightarrow\Omega_1(\mathbb{L}_t;R))$ and hence that $H^P_2(\mathbb{L}_t;R) = 0$. This is best verified by direct computation: simply note that since every edge in $\mathbb{L}_t$ is the intersection of a unique pair of directed triangles or squares, there is an essentially unique way in which the edges must cancel when evaluating the $\partial^P_2$-image of any element of $\ker(\partial^P_2)$. \qed
\end{proof}

We will use the $3$-path elements to explicitly describe inductive elements of $\Omega_3(G;R)$ for any digraph $G$. Before doing so, we need to establish some notation that will be used in the definitions and proofs below. For a $3$-path sequence digraph, introduce sets $\mathcal{I}$ indexing its $u_i$ vertices and $\mathcal{J}$ indexing its $v_j$ vertices. These sets will depend on the digraph in question, so let us agree on the following notation.
\begin{align}\label{eq:uIndexSets}
    & \text{For} \; \mathbb{L}{\vphantom{\mathbb{L}}}_{t}^u \; \text{with}\; t\geq1\; \text{let} \; \mathcal{I} = \{1,\dots,t+1\}, \; \text{and}
    \nonumber
    \\
    & \text{for} \;\mathbb{T}_t, \; \mathbb{L}_t \; or \; \mathbb{L}_t^l \; \text{with} \; t\geq1 \; \text{let} \; \mathcal{I} = \{1,\dots,t\}.
\end{align}
Similarly,
\begin{align}\label{eq:vIndexSets}
    & \text{for} \; \mathbb{L}{\vphantom{\mathbb{L}}}_{t}^l \; \text{with}\; t\geq1\; \text{let} \; \mathcal{J} = \{1,\dots,t+1\}, \; \text{and}
    \nonumber
    \\
    & \text{for} \;\mathbb{T}_t, \; \mathbb{L}_t \; or \; \mathbb{L}_t^u \; \text{with}\;t\geq1\; \text{let} \; \mathcal{J} = \{1,\dots,t\}.
\end{align}

With this notation established, we turn to describing the class of digraph maps which will be used to produce inductive elements in $\Omega_3(G;R)$.
\begin{definition}\label{def:3PathStructureMap}
    Let $\mathbb{H}$ be a $3$-path sequence digraph. A digraph map $\varphi \colon \mathbb{H} \to G$ is called a \emph{$3$-path structure map} if the following conditions are satisfied.
    \begin{enumerate}[(1)]
        \item 
        The map $\varphi$ sends any directed square of $\mathbb{H}$ contained in the core subdigraph $\overline{\mathbb{H}}$ to a directed square in $G$.
        \item 
        If $\mathbb{H}=\mathbb{L}_1$, then $\varphi$ is a strong digraph map on $\overline{\mathbb{L}}_1$.
        \item
        If $\mathbb{H}=\mathbb{L}_1^u$, then $\varphi(v_1) \neq \varphi(H)$, and if $\mathbb{H}=\mathbb{L}_1^l$ then $\varphi(T) \neq \varphi(u_1)$.
        \item
        If $\mathbb{H} = \mathbb{L}^u_t$, then $\varphi(u_{i_1})\neq\varphi(u_{i_2})$ for all $(i_1,i_2) \in (\mathcal{I} \times \mathcal{I}) \setminus \{ (1,t+1), (t+1,1) \}$ such that $i_1 \neq i_2$. 
        \item 
        If $\mathbb{H} = \mathbb{L}^l_t$, then $\varphi(v_{j_1})\neq\varphi(v_{j_2})$ for all $(j_1,j_2) \in (\mathcal{J} \times \mathcal{J}) \setminus \{ (1,t+1), (t+1,1) \}$ such that $j_1 \neq j_2$.
        \item 
        Otherwise, $\varphi(u_{i_1})\neq\varphi(u_{i_2})$ and $\varphi(v_{j_1})\neq\varphi(v_{j_2})$ for all $i_1 \neq i_2 \in \mathcal{I}$ and $j_1 \neq j_2 \in \mathcal{J}$.
    \end{enumerate}
\end{definition}

The following important statement, which follows easily from parts (1), (2), (3) of Definition~\ref{def:3PathStructureMap}, will be applied frequently in the remainder of the section.

\begin{lemma}\label{lem:3PathStructureStraongMap}
    A $3$-path structure map $\varphi \colon \mathbb{H}\to G$ becomes a strong digraph map when restricted to the core subdigraph $\overline{\mathbb{H}}$. 
    Consequently, it maps any directed triangle in $\mathbb{H}$ to either a directed triangle or double edge in $G$.
\end{lemma}

\begin{proof}
    If $\mathbb{H}=\mathbb{L}_1$, then the statements are ensured by parts (2) of Definition~\ref{def:3PathStructureMap}.
    If $\mathbb{H}=\mathbb{L}_1^u$ or $\mathbb{H}=\mathbb{L}_1^l$, then the statements are implied by parts (1) and (3) of Definition~\ref{def:3PathStructureMap}. In all other cases, the statements are a consequence of part (3) of Definition~\ref{def:3PathStructureMap}.
    \qed
\end{proof}

We show next that the $3$-path structure image of a $3$-path element is an inductive element in $\Omega_3(G;R)$. Later, we concern ourselves with the construction of $3$-path structure maps realising any given inductive element of $\Omega_3(G;R)$ in this way, and eventually with the uniqueness of these constructions.

\begin{theorem}\label{thm:SructureMapsAreInductive}
    Let $\mathbb{H}$ be a $3$-path sequence digraph and $\textfrak{H}$ the $3$-path element generating $\Omega_3(\mathbb{H};R)$. If $\varphi \colon \mathbb{H} \to G$ is a $3$-path structure map, then $\varphi_{\#}(\mathfrak{H})$ is an inductive element of $\Omega_3(G;R)$. Furthermore, the inductive element $\varphi_{\#}(\mathfrak{H})$ has a unique inductive structure.
\end{theorem}

\begin{proof}
    To show that $\varphi_\#(\textfrak{H}) \in \Omega_3(G;R)$ is an inductive element, we will construct an inductive structure $F_2^\varphi$ on $\varphi_\#(\textfrak{H})$ and prove that $\varphi_\#(\textfrak{H})$ is strongly connected.

    We start by constructing the face multigraph $F^\varphi_2$. To this end, observe that for $j\in\mathcal{J}$, the element $\delta^h_{3,v_j}(\mathfrak{H})$ is the unique directed square or directed triangle in $\mathbb{H}$ with head vertex $v_j$ and tail vertex $T$. By part (1) of Definition~\ref{def:3PathStructureMap} and Lemma~\ref{lem:3PathStructureStraongMap}, $\varphi$ sends each of these elements to a directed square, directed triangle, or double edge in $G$. Thus for $j\in\mathcal{J}$ put
    \[
        x_j = \varphi_{\#}(\delta^h_{3,v_j}(\mathfrak{H}))\in\Omega_2(G;R)
    \]
    and take $\{x_j\mid j\in\mathcal{J}\}$ as the vertex set of $F_2^\varphi$.
    
    Now, $\varphi$ is a strong digraph map on $\overline{\mathbb{H}}$ by Lemma~\ref{lem:3PathStructureStraongMap}.
    Since each edge of the form $(T,u_i)$ lies in $\overline{\mathbb{H}}$, for any $j\in \mathcal{J}\setminus \{\max(\mathcal{J})\}$ we have
    \[
        \delta^h_{2,\varphi(u_{i})}(x_j)
        =
        \pm e_{\varphi(T),\varphi(u_{i})}
        =
        - \delta^h_{2,\varphi(u_{i})}(x_{j+1}),
    \]
    where $i= j$ for $\mathbb{H} = \mathbb{T}_{t'}$, $\mathbb{L}_t$, or $\mathbb{L}_t^l$, and $i = j+1$ for $\mathbb{H} = \mathbb{L}_t^u$.
    Therefore, we can add an edge to $F_2^\varphi$ labelled $(\{ \varphi_{\#}(e_{T,u_{i}}), -\varphi_{\#}(e_{T,u_{i}}) \}, u_{i})$ between $x_j$ and $x_{j+1}$ for each $j\in \mathcal{J}\setminus \{\max(\mathcal{J}) \}$. In addition, when $\mathbb{H} = \mathbb{T}_{t'}$ we also have
    \[
        \delta^h_{2,\varphi(u_{\max(\mathcal{J})})}(x_{\max(\mathcal{J})})
        =
        \pm e_{\varphi(T),\varphi(u_{\max(\mathcal{J})})}
        =
        - \delta^h_{2,\varphi(u_{\max(\mathcal{J})})}(x_{1}).
    \]
    Hence, in this case, we add to the face multigraph $F_2^\varphi$ an additional edge labelled by $(\{ \varphi_{\#}(e_{T,u_{\max(\mathcal{J})}}), -\varphi_{\#}(e_{T,u_{\max(\mathcal{J})}}) \}, u_{\max(\mathcal{J})})$ between $x_{\max(\mathcal{J})}$ and $x_{1}$.
    Note that, by Lemma~\ref{lem:3PathStructureStraongMap}, $\varphi_{\#}(e_{v_j,H}) = e_{\varphi(v_j),\varphi(H)}\in \Omega_1(G;R)$ for each $j \in \mathcal{J}$.

    When $\mathbb{H} \neq \mathbb{L}^l_t$, the injectivity conditions on the $\varphi(v_j)$ from parts (4),(6) of Definition~\ref{def:3PathStructureMap} ensure that the vertex labels $\{ x_j \}_{j \in \mathcal{J}}$ are distinct.
    Moreover, by Lemma~\ref{lem:3PathStructureStraongMap}, $\varphi$ sends directed triangles to directed triangles or double edges, and because $\varphi$ is a $3$-path structure map it preserves the core directed squares $e_{u_{i},v_{i},H}-e_{u_{i},v_{i+1},H}$ for $i \in \mathcal{I}$ modulo $t$ when $\mathbb{H} = \mathbb{T}_t$, for $i \in \mathcal{I}\setminus \{ t \}$ when $\mathbb{H} = \mathbb{L}_t$, and for $i \in \mathcal{I}\setminus \{ 1,t+1 \}$ when $\mathbb{H} = \mathbb{L}_t^u$. Therefore,
    for $j\in\mathcal{J}$, we have $\delta^h_{2,u}(x_j) \neq 0$, $(u,h(I)) \notin E_G$, and that $u \neq h(I)$ if and only if there is an edge in $F_2^\varphi$ labelled by $u$ incident to $x_j$. It follows that in these cases, the labelled multigraph $F_2^\varphi$ is an $h(I)$-complete face multigraph.
    
    When $\mathbb{H} = \mathbb{L}^l_t$, the injectivity conditions in parts (5),(6) of Definition~\ref{def:3PathStructureMap} imply that all the $\varphi(v_j)$ are distinct except for $\varphi(v_1) = \varphi(v_{t+1})$. However, the conditions also imply that the two directed triangles $e_{T,u_1,v_1}$ and $e_{T,u_t,v_{t+1}}$ have distinct images. Thus again we have $\delta^h_{2,u}(x_a) \neq 0$, $(u,h(I)) \notin E_G$, and that $u \neq h(I)$ if and only if there is an edge in $F_2^\varphi$ labelled by $u$ incident to $x_j$. Hence in this case also, $F^\varphi_2$ is an $h(I)$-complete face multigraph.
     
    Using Lemma~\ref{lem:3PathStructureStraongMap}, in particular the existence of the edges $\varphi(e_{v_j,H}) = e_{\varphi(v_j),\varphi(H)}\in \Omega_1(G;R)$ for $j \in \mathcal{J}$, the fact that $\varphi$ sends directed triangles to directed triangles or double edges, 
    and the construction and functorality of $\varphi_{\#} \colon \Omega_n(\mathbb{H};R) \to \Omega_n(G;R)$, we have $\varphi_\#(\textfrak{H}) = \left[\sum_{a \in \mathcal{A}} x_a\right]^{\varphi(H)}$.
    Therefore, since by Proposition~\ref{prop:LowDimInductiveElements} the elements $\{ x_j \}_{j\in \mathcal{J}}$ and edges $e_{\varphi(T),\varphi(u_i)}$ for $i \in \mathcal{I}$ are inductive elements, $F_2^\varphi$ is an upper inductive structure on $\varphi_\#(\textfrak{H})$.
    
    It remains to show that $\varphi_\#(\textfrak{H})$ is strongly connected.
    Again using the injectivity conditions on the $\varphi(v_j)$ 
    from part (6) of Definition~\ref{def:3PathStructureMap}, no two directed squares in $\{ x_j \}_{j\in \mathcal{J}}$ sit within the same multisquare in $G$.
    By Proposition~\ref{prop:LowDimInductiveElements}, the only inductive elements in $\Omega_2(G;R)$ are directed squares, directed triangles, and double edges up to sign. This implies that for any inductive elements $\bar{x}_1,\dots,\bar{x}_{m} \in \Omega_2(G;R)$ such that $\bar{x}_1+\cdots+\bar{x}_{m} = \sum_{j\in{\mathcal{J}}} x_j$, we have that $m=|\mathcal{J}|$ and $\bar{x}_1,\dots,\bar{x}_{m}$ is a permutation of $x_1,\dots,x_{\max{(\mathcal{J})}}$.
    Hence, the set of vertex labels in the inductive structure for $\varphi_\#(\textfrak{H})$ is uniquely determined.
    Furthermore, the injectivity conditions on $\varphi(u_i)$ from 
    Definition~\ref{def:3PathStructureMap} 
    imply that each edge also has a uniquely determined label containing either $u_i$ or $H$. Hence, the edge labels in $F_2^\varphi$ are uniquely determined. Therefore, the inductive structure $F_2^\varphi$ on $\varphi_\#(\textfrak{H})$ is unique, which implies that $\varphi_\#(\textfrak{H})$ is strongly connected. In particular, $\varphi_\#(\textfrak{H})$ is an inductive element, as required. \qed
\end{proof}

Given an inductive element $I \in \Omega_{3}(G;R)$, the next theorem provides an explicit construction of a $3$-path structure map $\varphi_I\colon\mathbb{H}_I\rightarrow G$ realising $I$ as the induced image of the $3$-path element $\mathfrak{H}_I\in\Omega_3(\mathbb{H}_I;R)$.
The $3$-path sequence $\mathbb{H}_I$ which is the domain of $\varphi_I$ is carefully constructed in Table~\ref{table:HIdef} below using as input a choice of inductive structure $F_2$ for $I$. Both $\mathbb{H}_I$ and the map $\varphi_I$ depend on a choice of vertices $w_i^I \in V_G$, which is fixed first in Table~\ref{table:wiIconditions}. Although the constructions make explicit use of $F_2$, we show later in Theorem~\ref{thm:3PathStructureConstruction} that $\varphi_I$ is essentially independent of the particular choice of inductive structure, thereby justifying the choice of notation.

To begin the construction, fix an inductive element $I \in \Omega_{3}(G;R)$ and realise its inductive structure by means of a face multigraph $F_2 = F_2(x_1,\dots,x_m)$ whose vertex labels $x_1,\dots,x_m \in \Omega_2(G;R)$ are inductive elements and whose edges are labelled using inductive elements in $\Omega_1(G;R)$. By Proposition~\ref{prop:LowDimInductiveElements}, inductive elements in $\Omega_2(G;R)$ are either double edges, directed triangles, or directed squares, and inductive elements in $\Omega_1(G;R)$ are edges of $G$ up to sign. Moreover, by part (3) of Proposition~\ref{prop:exsumm}, the underlying multigraph of $F_2$ is a line or a cycle.

When $F_2$ is a line, we may assume that the indices $1,\dots,m$ are such that there is exactly one edge between $x_j$ and $x_{j+1}$ for each $j=1,\dots,m-1$. When $F_2$ is a cycle and $m\geq3$, we treat the indices modulo $m$ and assume that there is exactly one edge between $x_j$ and $x_{j+1}$ for each $j=1,\dots,m$. In each of these cases, denote the unique edge between $x_j$ and $x_{j+1}$ by $(\{ y_j, -y_j \}, v_j)$ for some $y_j \in \Omega_{2}(G;R)$ and $v_j \in V_G$. Finally, when $F_2$ is a cycle and $m=2$, we assume that $F_2$ consists of two vertices $x_1,x_2$ and two edges between them labelled $(\{ y_1, -y_1 \}, v_1),(\{ y_2, -y_2 \}, v_2)$ for some $y_1, y_2 \in \Omega_{2}(G;R)$ and $v_1,v_2\in V_G$ with $v_1\neq v_2$.

For each $j=1,\dots,m$ associate to $F_2$ the set of vertices
\begin{equation}\label{eq:HeadDeltaAtVertices}
    \delta_{x_j}
    =
    \{ v \in V_G \: | \: \delta^h_{2,v}(x_j) \neq 0 \}.
\end{equation}
These sets are used in Table~\ref{table:wiIconditions} below to define an indexing set $\mathcal{W}_I$ and vertices $w_i^I \in V_G$ for $i \in \mathcal{W}_I$.
\begin{center} 
\captionof{table}{Definition of $\mathcal{W}_I$ and the $w_i^I$ for $i\in\mathcal{W}_I$.}\label{table:wiIconditions}
\begin{tabular}{ | c | c | c | c | c | c | } 
  \hline
  \makecell{$F_2$} &  & \makecell{\bf $x_1$ \\ square} & \makecell{\bf $x_m$ \\ square} & {\bf $\mathcal{W}_I$} & {\bf $w_i^I$} \\ 
  \hline
  cycle & $m\geq1$ & - & - & $\{1,\dots,m\}$ & $w_i^I = v_i$ \\ 
  line & $m=1$ & \cmark & \cmark & $\{1,2\}$ & $\{ w_1^I,w_2^I \} = \delta_{x_1}$ \\
  line & $m=1$ & \xmark & \xmark & $\{1\}$ & $\{ w_1^I \} = \delta_{x_1}$ \\
  line & $m\geq2$ & \xmark & \xmark & $\{1,\dots,m-1\}$ & $w_i^I = v_i$ \\ 
  line & $m\geq2$ & \cmark & \xmark & $\{1,\dots,m\}$ & $w_i^I = v_{i-1}$ for $i\neq 1$, $w_1^I \in \delta_{x_1}$ and $w_1^I \neq v_1$  \\ 
  line & $m\geq2$ & \xmark & \cmark & $\{1,\dots,m\}$ & $w_i^I = v_{i}$ for $i\neq m$, $w_m^I \in \delta_{x_m}$ and $w_m^I \neq v_{m-1}$ \\ 
  line & $m\geq2$ & \cmark & \cmark & $\{1,\dots,m+1\}$ & \makecell{$w_i^I = v_{i-1}$ for $i\neq 1,m+1$, $w_1^I \in \delta_{x_1}$, \\ $w_{m+1}^I \in \delta_{x_m}$ and $w_1^I \neq v_1$, $w_{m+1}^I \neq v_{m-1}$} \\ 
  \hline
\end{tabular}
\end{center}
In this table, the first two column are the conditions on the structure of the face multigraph $F_2$ and the number $m$ of its vertices. The third and fourth columns are the conditions that the vertices $x_1$ or $x_m$ of $F_2$ as elements of $\Omega_2(G;R)$ are directed squares, with a tick indicating that the condition in the column is met and a cross that it is not. The final two columns define $\mathcal{W}_I$ and $w_i^I\in V_G$ for $i\in \mathcal{W}_I$ under the conditions determined by the first four columns.
For reasons of convenience, in the definition of $\mathbb{H}_I$ next, there is a single exceptional case in the second line of Table~\ref{table:wiIconditions} in the situation when $(w_1^I,H)\in E_G$ and $w_2^I = h(I)$. Under these conditions, we assume the indices $1$ and $2$ on $w_1^I$ and $w^I_2$ are switched, the original choice having been arbitrarily, and this will be required to make the definition of $\mathbb{H}_I$ in line $3$ of Table~\ref{table:HIdef} correct.

We also define a $3$-path sequence digraph $\mathbb{H}_I$, which depends on the structure of $I$ and $F_2$, using the conditions detailed in Table~\ref{table:HIdef} below. For the purpose of defining $\varphi_I$ later, it will be convenient in some cases to reverse the order of the indices on the vertex labels $x_1,\dots,x_m$ of $F_2$, and the conditions on when this is to be done are indicated in the final column of Table~\ref{table:HIdef}.

\begin{center}
\captionof{table}{Definition of $\mathbb{H}_I$.}\label{table:HIdef}
\begin{tabular}{ | c | c | c | c | c | c | c | c | c | c |} 
  \hline
  $F_2$ & $m = 1$ & $x_1$ square & $x_m$ square & $\mathbb{H}_I$ & reverse $x_1,\dots,x_m$ \\ 
  \hline
  cycle & $m\geq1$ & - & - &  $\mathbb{T}_m$ & \xmark \\ 
  line & $m=1$ & \xmark & - & $\mathbb{L}_1$ & \xmark \\
  line & $m=1$ & \cmark & - & $\mathbb{L}_1^u$ & \xmark \\
  line & $m\geq2$ & \xmark & \xmark & $\mathbb{L}_{m-1}^l$ & \xmark \\ 
  line & $m\geq2$ & \xmark & \cmark & $\mathbb{L}_m$ & \xmark \\ 
  line & $m\geq2$ & \cmark & \xmark & $\mathbb{L}_m$ & \cmark \\ 
  line & $m\geq2$ & \cmark & \cmark & $\mathbb{L}_m^u$ & \xmark \\
  \hline
\end{tabular}
\end{center}

In this table, the first four columns are conditions on $I$ and $F_2$, a tick indicating that the condition in the column is met, and a cross that it is not. The second to last column defines the digraph $\mathbb{H}_I$ under the conditions determined by the first four columns.
In the last column, a tick indicates that the additional step of reversing the order of the indices on $x_1,\dots,x_m$ in $F_2$ should be performed by replacing $i$ by $m-i+1$ for each $i = 1,\dots,m$.
It is easy to see that Table~\ref{table:HIdef} assigns to any inductive structure $F_2$ on $I$ a unique $3$-path sequence digraph $\mathbb{H}_I$. 

\begin{theorem}\label{thm:3PathStructureConstruction}
    Let $I\in \Omega_3(G;R)$ be an inductive element and $F_2$ an inductive structure on $I$ with vertices $x_1,\dots,x_m$. Let $w_j^I\in V_G$ be the vertices defined in Table~\ref{table:wiIconditions}, and let $\mathbb{H}_I$ be the $3$-path sequence digraph defined in Table~\ref{table:HIdef}, with $\textfrak{H}_I\in H_3(\mathbb{H}_I;R)$ denoting the associated $3$-path element.
    Then a digraph map $\varphi_I \colon \mathbb{H}_I \to G$ is determined by setting
    \begin{equation}\label{eq:VetexImages}
        \varphi_I(H) = h(I), \;\;\;
        \varphi_I(v_j) = h(x_j), \;\;\;
        \varphi_I(u_i) = w_i^I
        \;\;\; \text{and} \;\;\;\
        \varphi_I(T) = t(I),
    \end{equation}
    for $i \in \mathcal{I}$ and $j\in \mathcal{J}$. Furthermore: 
    \begin{enumerate}[(i)]
        \item
        The induced image of $\textfrak{H}_I$ under $\varphi_I$ is $\pm I$.
        \item 
        The map $\varphi_I$ is a well defined $3$-path structure map.
    \end{enumerate}
    In fact, $\mathbb{H}_I$ and $\varphi_I$ are assigned to $I$ independently of the choice of $F_2$.
\end{theorem}

Before passing to the proof of Theorem~\ref{thm:3PathStructureConstruction}, we pause to formalise a simple argument which will be used repeatedly.

\begin{lemma}\label{lem:LocalStructureHitsInductive}
    In the notation of Theorem~\ref{thm:3PathStructureConstruction}, assume that $\varphi_I$ is a well defined digraph map and suppose that each $x_1,\dots,x_m$ lies in the image of $\varphi_{I\#} \colon \Omega_2(\mathbb{H}_I;R) \to \Omega_2(G;R)$. Then $\varphi_{I\#} \colon \Omega_3(\mathbb{H}_I;R) \to \Omega_3(G;R)$ satisfies  $\varphi_{I\#}(\textfrak{H}_I) = \pm I$.
\end{lemma}

\begin{proof}
    By Proposition~\ref{prop:LowDimInductiveElements}, the labels 
    $x_j \in \Omega_2(G;R)$ are either double edges, directed triangles, or directed squares, implying that the sets $\delta_{x_j}$ from equation~\eqref{eq:HeadDeltaAtVertices} satisfy $1 \leq |\delta_{x_j}| \leq 2$ for each $1 \leq j \leq m$. 
    Since $\varphi_{I \#}\colon \Omega_2(\mathbb{H}_I;R) \to \Omega_2(G;R)$ contains each $x_j$ in its image, by the definition of the vertices $w_i^I$, it must be the case that
    \[
        \delta_{x_j} = \{w^I_{j}\}, \:
        \delta_{x_j} = \{w^I_{j-1}\}, \:
        \delta_{x_j} = \{w^I_{j+1}\}, \:
        \delta_{x_j} = \{w^I_{j-1},w^I_{j}\}, \: \text{or} \;
        \delta_{x_j} = \{w^I_{j},w^I_{j+1}\}.
    \] 
    Therefore, as $I = [x_1+\cdots+x_m]^{h}$ by definition, $I$ can be written as a sum of elementary $3$-paths 
    which either have the form $\pm e_{t(I),w_{j}^I,h(x_j),h(I)}$ and $\pm e_{t(I),w_{j-1}^I,h(x_{j}),h(I)}$, or $\pm e_{t(I),w_{j}^I,h(x_j),h(I)}$ and $\pm e_{t(I),w_{j+1}^I,h(x_{j}),h(I)}$, 
    where we treat indices modulo $m$ when $F_2$ is a cycle.
    By the construction of $\varphi_I$, any of these elementary $3$-paths appearing in $I$ is contained in the image of $\varphi_I$, while no other $3$-path of $G$ lies in the image of $\varphi_I$.
    Thus the image of the element $\textfrak{H}_I$ under $\varphi_{I\#} \colon \Omega_3(\mathbb{T}_m;R) \to \Omega_3(G;R)$ is $\pm I$.
    \qed
\end{proof}

\begin{proof}[Proof of Theorem~\ref{thm:3PathStructureConstruction}]
    We begin by showing that $\varphi_I$ is a well defined digraph map
    satisfying conditions (1), (2), and (3) of Definition~\ref{def:3PathStructureMap} and
    $\varphi_{I\#}(\textfrak{H})= \pm I$.
    Achieving this proves part $(i)$ of the theorem and begins the proof of part $(ii)$. The proof of part $(ii)$ is then completed by verifying that $\varphi_I$ satisfies the injectivity conditions (4), (5), and (6) from Definition~\ref{def:3PathStructureMap}. Finally, with $(i),(ii)$ in hand, the last statement follows by applying the second part of Theorem~\ref{thm:SructureMapsAreInductive}, which implies that $F_2$ is uniquely determined by $I$.
    
    We deal first with the case that  $F_2$ is a cycle and $\mathbb{H}_I = \mathbb{T}_m$. Then each $x_i$ is a directed square, so the $\varphi_I$-images of the edges $T \to u_i$, $u_i \to v_i$, and $u_i \to v_{i+1}$ exist in $G$, where the indices $i \in \mathcal{I} = \{1,\dots,m\}$ are treated modulo $m$.
    The $\varphi_I$-images of the remaining edges $v_i \to H$ also exist in $G$ by the $h(I)$-completeness of $F_2$. Hence, $\varphi_I$ is a strong digraph map.
    In addition, $\varphi_I$ preserves directed squares of the form $e_{u_i,v_{i},H}-e_{u_{i},v_{i+1},H}$ by $h(I)$-completeness of $F_2$, and directed squares of the form $e_{T,u_{i},v_{i}}-e_{T,u_{i+1},v_{{i}}}$, as each $x_1,\dots,x_m$ is a directed square, where we treat indices $i \in \mathcal{I} = \{1,\dots,m\}$ modulo $m$. Since $\overline{\mathbb{T}}_m = \mathbb{T}_m$, it therefore follows that $\varphi_I$ preserves directed squares in $\overline{\mathbb{T}}_m$. In particular, condition (1) of Definition~\ref{def:3PathStructureMap} is satisfied.
    Lastly, each $x_1,\dots,x_m$ is up to sign the $\varphi_{I}$-image of some $e_{T,u_{i},v_{i}}-e_{T,u_{i+1},v_{{i}}}$. Hence, the fact that $\varphi_{I\#}(\textfrak{H}_I)=\pm I$ follows from Lemma~\ref{lem:LocalStructureHitsInductive}.
    
    We now consider the various cases when $F_2$ is a line. As the arguments are similar to those when $F_2$ is a cycle above, some details will be omitted.
    We deal first with the case $m=1$, so that $F_2$ consists of a single vertex $x_1$, which may be either a directed square, directed triangle, or double edge. 
    Assuming that $x_1$ is a directed square, we have $x_1 = \pm (e_{h(x_1),w_1^I,t(I)} - e_{h(x_1),w_2^I,t(I)})$.
    Evidently, $w^I_1\neq w_2^I$, so by the conventions on $w_1^I,w_2^I$ given under Table~\ref{table:wiIconditions}, there are only two cases to deal with: either $h(I)\not\in\{w_1^I,w_2^I\}$ and $(w_1^I,h(I)),(w_2^I,h(I))\in E_G$ by the $h(I)$-completeness of $F_2$, or $h(I)=w_1^I$ and $(w_2^I,h(I))\in E_G$. 
    Table~\ref{table:HIdef} assigns to both these situations the digraph $\mathbb{L}_1^u$, so we must check clause (3) of Definition~\ref{def:3PathStructureMap}. But this is clear, as we must have $\varphi_I(v_1) \neq \varphi_I(H)$, or else $F_2$ would not be a well defined inductive structure. It's now easy to check that $\varphi_I$ is well-defined and preserves directed squares in $\overline{\mathbb{L}}{\vphantom{\mathbb{L}}}_1^u$, and finally to use Lemma~\ref{lem:LocalStructureHitsInductive} to show that $\varphi_{I\#}(\textfrak{H}_I)=\pm I$.

    The other possibilities for $x_1$ as a directed triangle or double edge are contained in row $2$ of Table~\ref{table:HIdef}. In particular, $I$ must consist of a single non-trivial $3$-path summand $e_{t(I),v_1,v_2,h(I)}$, where each edge $t(I) \to v_1$, $v_1 \to v_2$, and $v_2 \to h(I)$ is the image of a unique edge in $\mathbb{H}_I$ under $\varphi_I$. This makes $\varphi_I$ a strong digraph map on $\overline{\mathbb{L}}_1$, verifying part (2) of Definition~\ref{def:3PathStructureMap}. The remaining details are straightforward and left to the reader.
    
    Now assume that $m \geq 2$. As explained in the paragraph above Table~\ref{table:wiIconditions}, we may assume that the vertices $x_1,x_m$ have valence $1$ and the vertices $x_2,\dots,x_{m-1}$ have valence $2$. In this situation, $x_1$ and $x_m$ may be either directed squares, directed triangles, or double edges, while $x_2,\dots,x_{m-1}$ must all be directed squares.

    When both $x_1$ and $x_m$ are directed squares, $x_1 = \pm (e_{h(x_1),w_1^I,t(I)} - e_{h(x_1),w_2^I,t(I)})$ and $x_m = \pm (e_{h(x_m),w_{\max\mathcal{W}_I}^I,t(I)} - e_{h(x_m),w_{\max\mathcal{W}_I-1}^I,t(I)})$.
    Since $x_1$ has valence $1$ and $F_2$ is $h(I)$-complete, either $w_1^I=h(I)$ or $(w_1^I,h(I))\in E_G$.
    Similarly, $x_m$ having valence $1$ implies that either $w_{\max\mathcal{W}_I}^I=h(I)$ or $(w_{\max\mathcal{W}_I}^I,h(I))\in E_G$.
    There are four possibilities in total, and these are all covered by the last row of Table~\ref{table:HIdef}. The details in each case are straightforward and similar to those given above, so will be omitted.
    
    The remainder of this part of the proof includes only the cases in which one or both of $x_1$ and $x_m$ is a directed triangle or double edge. The relevant $3$-path sequence digraphs are found in rows $4$, $5$, and $6$ of Table~\ref{table:HIdef}, and the details to be checked are completely analogous to those discussed above. However, there is one exceptional case, which is that it is not possible for both $x_1$ and $x_m$ to be double edges when $m=2$, as there is no way to construct a connected face multigraph on these two vertices.

    At this stage, we have shown that $\varphi_I$ is a well defined digraph map satisfying conditions (1), (2), and (3) of Definition~\ref{def:3PathStructureMap} as well as $\varphi_{I\#}(\mathfrak{H}_I)=\pm I$. To complete the proof, it remains only to show that $\varphi_I$ is a $3$-path structure map. In particular, we must show that $\varphi_I$ satisfies the injectivity conditions in parts (4), (5), and (6) of Definition~\ref{def:3PathStructureMap}. We drop the assumptions on $F_2$ introduced above, so below $F_2$ can be either a line or a cycle. In the latter case, we always treat its indices modulo $m$.

    Firstly, if $\varphi_I$ is not injective on the vertices $u_i$, then there are $i_1,i_2 \in \mathcal{I}$ such that $i_1 < i_2$ and $\varphi_I(u_{i_1}) = \varphi_I(u_{i_2}) = w \in V_G$. We will use this to construct a disconnected inductive structure $\bar{F}_2$ on $I$, so contradicting its strong connectedness.
    Before starting, note that $|\mathcal{I}| \geq 2$ by assumption, so we may so we may discount $\mathbb{H}_I$ from being $\mathbb{L}_1$ or $\mathbb{L}_1^l$ in the following considerations.
    
    To proceed, initially assume that at least one of $u_{i_1} \to v_{1} \to H$ or $u_{i_2} \to v_{\max(\mathcal{J})} \to H$ is a directed triangle. Note that this implies that $i_1=1$ in the first case and that $i_2=\max(\mathcal{I})$ in the second. In particular, when both are directed triangles, $\mathbb{H}_I=\mathbb{L}^u_t$ and we are in the exceptional case allowed for in part (4) of Definition~\ref{def:3PathStructureMap}. Therefore, assume that just one of them is a directed triangle, in which case $\mathbb{H}_I=\mathbb{L}_t$, or $\mathbb{H}_I=\mathbb{L}_t^u$ for $t\geq 2$.
    
    We treat the case that $u_{i_1} \to v_{1} \to H$ is a directed triangle, the other cases being essentially the same. In this situation, $\mathbb{H}_I=\mathbb{L}^u_t$ for some $t\geq 2$, so 
    $(u_{i_1}, H) \in E_{\mathbb{H}_I}$ and there is a directed square in $\mathbb{H}_I$ of the form $x=\pm (e_{u_{i_2},v_j, H}-e_{u_{i_2},v_{j+1}, H})$ for some $j \in \mathcal{J}$.
    However, this implies that there is an edge $(\varphi_I(u_{i_2}), \varphi_I(H)) \in E_G$, which cannot be the case, as $\varphi_I$ preserves the directed squares contained in $\overline{\mathbb{H}}_I$, meaning that it sends $x$ to a directed square satisfying $t(\varphi_I(x)) = w = \varphi_I(u_{i_2})$ and $h(\varphi_I(x)) = \varphi_I(H)$. By definition, directed squares cannot have an edge between their head and tail vertices, so there is a contradiction.
    
    Thus we reduce to the case that neither $u_{i_1} \to v_{1} \to H$ nor $u_{i_2} \to v_{\max(\mathcal{J})} \to H$ is a directed triangle. In particular, $\mathbb{H}_I \neq \mathbb{L}^u_1$. Since the digraphs $\mathbb{L}_1$ and $\mathbb{L}^l_1$ were eliminated previously, from now on we may assume that $m\geq 2$.

    Now, if $i_1>1$, then $u_{i-1}$ exists, and by the constructions of $\varphi_I$ and $F_2$, the edge $T \to w$ is contained in at least two directed squares, whose images under $\varphi_I$ are vertices of $F_2$ with the form
    \[
        x_{k} = \pm (e_{t(I),\varphi_I(u_{i_1-1}),h(x_{k})} - e_{t(I),w,h(x_{k})})
        \;\;\; \text{and} \;\;\;
        x_{k+1} = \pm (e_{t(I),\varphi_I(u_{i_1 +1}),h(x_{k+1})} - e_{t(I),w,h(x_{k+1})})
    \]
    for some $k = 1,\dots,m-1$.
    On the other hand, if $i_1 = 1$, then $u_{i_{1}-1}$ does not exist. 
    Given the previous reductions, this is the case if and only if 
    $T \to u_{i_1} \to v_{1}$ is a directed triangle, which implies that either $\mathbb{H}_I = \mathbb{L}_t$ or $\mathbb{H}_I = \mathbb{L}_t^l$ for some $t \geq2$.
    In this case, we may take the directed triangle or double edge
    \[
         x_{k} = \pm e_{t(I),w,\varphi_I(v_{1})}
    \]
    and let $x_{k+1}$ be the directed square defined above, this being certain to exist
    as $i_1 \leq m-1$.
        
    In any case, by the construction of $\varphi_I$ there is an edge between $x_{k}$ and $x_{k+1}$ in $F_2$ labelled $(\{e_{t(I), w},-e_{t(I),w }\}, w)$.
    By an argument symmetric to that above, we obtain using $u_{i_2}$ instead of $u_{i_1}$, vertices $x_{k'},x_{k'+1}$ of $F_2$ which are joined by an edge labelled $(\{e_{t(I), w},-e_{t(I),w }\}, w)$.
    As $i_1 < i_2$, it follows that $k < k' $. Hence, the two edges in $\bar{F}_2$ labelled $(\{e_{t(I), w},-e_{t(I),w }\}, w)$ between $x_{k}$, $x_{k+1}$ and $x_{k'}$, $x_{k'+1}$ determined above are distinct. Therefore, a new inductive structure $\bar{F}_2$ on $I$ can be formed by replacing the edges between $x_{k}$, $x_{k+1}$ and $x_{k'}$, $x_{k'+1}$ by edges between $x_{k}$, $x_{k'}$ and $x_{k+1}$, $x_{k'+1}$ or $x_{k}$, $x_{k'+1}$ and $x_{k+1}$, $x_{k'}$ with labels $(\{e_{t(I), w},-e_{t(I),w }\}, w)$.
    However, as $F_2$ was a cycle or a line, $\bar{F}_2$ must be disconnected, which contradicts the strong connectedness of the inductive element $I$. This completes the proof of the injectivity condition on the $\varphi(u_i)$ for $i \in \mathcal{I}$.

    Turning now to the remaining injectivity conditions, if $\varphi_I$ is not injective on the vertices $v_j$, then there are $j_1,j_2 \in \mathcal{J}$ such that $j_1 < j_2$ and $\varphi_I(v_{j_1}) = \varphi_I(v_{j_2}) = w \in V_G$.
    As before, we will use the assumption to construct a disconnected inductive structure $\bar{F}_2$ on $I$, contradicting the strong connectedness of $I$. 
    Noting that in the present situation $m\geq2$ is clearly implied by the fact that $|\mathcal{J}|\geq2$, the face multigraph $\bar{F}_2$ will be obtained from $F_2$ by replacing two of its vertices and defining a new edge structure. The new vertices will depend on elements $y_1,y_2\in\Omega_2(\mathbb{H}_I;R)$, to be defined as follows.

    If $j_1 =1$ and $\mathbb{H}_I = \mathbb{L}_t$ or $\mathbb{L}_t^l$, then take $y_1$ to be whichever of the directed triangle $\pm e_{T, u_{1}, v_{j_1}}$ is required for $\varphi_I(y_1)$ to be a vertex of $F_2$. Otherwise, there is $i_1 \in\mathcal{I}$ such that
    $\Omega_2(\mathbb{H}_I;R)$ contains a directed square
    \begin{equation}\label{eq:squares1}
        y_1 = \pm (e_{T,u_{i_1},v_{j_1}}-e_{T,u_{i_1+1},v_{j_1}})
    \end{equation}
    and $\varphi_I(y_1)$ is a vertex of $F_2$. Similarly, if $j_2 =\max(\mathcal{J})$ and $\mathbb{H}_I = \mathbb{L}_t^l$, then take $y_2$ to be whichever of the directed triangles $\pm (e_{T,u_{\max(\mathcal{I})},v_{j_2}})$ is required for $\varphi_I(y_2)$ to be a vertex of $F_2$.
    Otherwise, there is an $i_2 \in\mathcal{I}$ such that
    $\Omega_2(\mathbb{H}_I;R)$ contains a directed square
    \begin{equation}\label{eq:squares2}
        y_2 = \pm (e_{T,u_{i_2},v_{j_2}}-e_{T,u_{i_2+1},v_{j_2}})
    \end{equation}
    and $\varphi_I(y_2)$ is a vertex of $F_2$. 
    Crucially, in each situation $j_1 < j_2$ implies that $y_1 \neq y_2$.

    If either $y_1$ or $y_2$ is a directed triangle, then either $(\varphi_I(T) = t(I),w) \in E_G$, or $\varphi_I(T) = t(I) = w$. Hence, as $\varphi_I$ preserves directed squares on $\overline{\mathbb{H}}_I$,
    both $y_1$ and $y_2$ must be directed triangles and we are in the case concerning $\mathbb{L}^l_t$ allowed for in part (5) of Definition~\ref{def:3PathStructureMap}.
    
    Otherwise, both $y_1$ and $y_2$ are directed squares. Since $i_1 < i_2$ and $\varphi_I(y_1)$, $\varphi_I(y_2)$ are joined by a path in $F_2$, the two directed squares from equations~\eqref{eq:squares1}~and~\ref{eq:squares2} must have the same sign.
    We have $j_2 \neq j_1 +1$, as $\varphi_I$ must preserve the directed square $e_{u_{i_1+1}=u_{i_2},v_{j_1},H}-e_{u_{i_1+1}=u_{i_2},v_{j_2},H}$, and together with the proof above of parts (4) and (6) of Definition~\ref{def:3PathStructureMap} for $\varphi_I(u_i)$ and $i \in \mathcal{I}$, this implies that $\varphi(u_{i_1}), \varphi(u_{i_1+1}), \varphi(u_{i_2}), \varphi(u_{i_2+1}) \in V_G$ are all distinct. It follows that
    $\varphi_I(y_1)$ and $\varphi_I(y_2)$ are contained within a multisquare in $G$. Therefore, we may take directed squares $y'_1,y'_2 \in \Omega_2(G;R)$ such that $\varphi_I(y_1)+\varphi_I(y_2)=y'_1+y'_2$ and
    \begin{align}\label{eq:SquareChoices}
        y'_1 = \pm (e_{\varphi_I(T),\varphi_I(u_{i_1}),w}-e_{\varphi_I(T),\varphi_I(u_{i_2+1}),w})
        \;\; \text{and} \;\;
        y'_2 = \pm (e_{\varphi_I(T),\varphi_I(u_{i_2}),w}-e_{\varphi_I(T),\varphi_I(u_{i_1+1}),w}).
    \end{align}

    We now form the inductive structure $\bar{F}_2$ on $I$ by replacing the vertices labelled $\varphi_I(y_1)$, $\varphi_I(y_2)$ by vertices labelled $y'_1$ and $y'_2$,
    retaining the edges and their labels incident to $\varphi_I(y_1)$ and $\varphi_I(y_2)$ and attaching them to an appropriate $y'_1$ and $y'_2$ labelled vertex instead.
    By assumption, $F_2$ is either a cycle or a line, and due to the construction of $y'_1$ and $y'_2$ in equation~\eqref{eq:SquareChoices}, $\bar{F}_2$ is disconnected. However, this contradicts the strong connectedness of the inductive element $I$. \qed
\end{proof}

Finally, we determine the extent to which the $3$-path structure maps constructed in the previous theorem are uniquely determined by the inductive element $I$. This programme is complicated, however, by a degree of redundancy which arises from the presence of symmetries of the underlying digraphs. In particular, for each $t'=2,3,\dots$, there is a natural action of the cyclic group $C_{t'}$ of order $t'$
on the trapezohedron digraph $\mathbb{T}_{t'}$ from Definition~\ref{def:Trapezohedron}, which is determined by the action of a generator $\sigma \in C_{t'}$ on vertices as
\[
    \sigma(T) = T,\;
    \sigma(u_i) = u_{i+1},\;
    \sigma(v_i) = v_{i+1},\;
    \text{and} \;
    \sigma(H) = H
\]
for $i=1,\dots,t'$, with indices read modulo $t'$. Moreover, it is straightforward to show that any group acting on $\mathbb{T}_{t'}$ by digraph maps must be a subgroup of $C_{t'}$. Therefore, we write $\mathrm{Sym}(\mathbb{T}_{t'}) = C_{t'}$. The $C_{t'}$-action on $\mathbb{T}_{t'}$ extends linearly to an action on $\mathcal{A}_3(\mathbb{T}_{t'};R)$, and subsequently descends to an action on $\Omega_3(\mathbb{T}_{t'};R)$.

Similarly, there are natural actions of $C_2=\{1_{c_2}, \sigma \}$ on the $3$-path sequence digraphs $\mathbb{L}_t^u$ and $\mathbb{L}_t^l$, for $t=1,2,\dots$. In each case, the $C_2$ action is determined by $\sigma(T)=T$, $\sigma(H)=H$ and
\begin{align*}
    & \text{for $\mathbb{L}_t^u$ by} 
    & \sigma(u_i) &= u_{t-i+2} 
    & \text{and} 
    && \sigma(v_i) &= v_{t-i+1} 
    &&\text{for $i=1,\dots,t+1$ and $j=1,\dots,t$;}
    \\
    & \text{for $\mathbb{L}_t^l$ by}
    & \sigma(u_i) &= u_{t-i+1}
    & \text{and}
    && \sigma(v_i) &= v_{t-i+2} 
    && \text{for $i=1,\dots,t$ and $j=1,\dots,t+1$; \;\:}
\end{align*}
Since the $C_2$-actions are effected through digraph maps, they define also linear actions on the module of path chains. In each case, this action switches the sign of the $3$-path element. Extending the notation introduced above, write $\mathrm{Sym}(\mathbb{L}_t^u) = \mathrm{Sym}(\mathbb{L}_t^l) = C_2$. 
Also, as there are no non-trivial group actions on $\mathbb{L}_t$ by digraph maps, write $\mathrm{Sym}(\mathbb{L}_t) = 1$.

The symmetry groups constructed above appear in the statement of Theorem~\ref{thm:Omega3Explicitly} below, which asserts that the 3-path structure maps constructed in Theorem~\ref{thm:SructureMapsAreInductive} are unique up to their actions. Before passing to the theorem, we pause to establish some technical facts that will be used in its proof.

\begin{lemma}\label{lem:3PathStructHitsDirTriAndDblEdg}
    Let $\varphi \colon \mathbb{H} \to G$ be a $3$-path structure map and denote by $\textfrak{H}\in\Omega_3(\mathbb{H};R)$ the $3$-path element associated to $\mathbb{H}$. The following statements hold: 
    \begin{enumerate}
        \item 
        If either
        \begin{enumerate}[(a)]
            \item 
            $\mathbb{H}=\mathbb{L}_2$ and $\varphi(u_2)=\varphi(v_1)$, or
            \item 
            $\mathbb{H}=\mathbb{L}_t$, $\mathbb{L}_t^u$, or $\mathbb{L}_t^l$ for some $t\geq 1$ and $\varphi(T)=\varphi(H)$,
        \end{enumerate}
        then there is a directed triangle or double edge $x\in\Omega_2(\varphi(\mathbb{H});R)$ that is not the $\varphi_\#$-image of any directed triangle in $\mathbb{H}$.
        Moreover, this $x$ cannot be written as $\pm \delta^h_{3,u}(\varphi_{\#}(\textfrak{H}))$ or $\pm \delta^t_{3,u}(\varphi_{\#}(\textfrak{H}))$ for any $u \in V_G$.
        \item
        Otherwise, any directed triangle or double edge in $\varphi(\mathbb{H})$ is the $\varphi_\#$-image of some directed triangle in $\mathbb{H}$.
    \end{enumerate}
\end{lemma}

When reading the proof of the lemma, the reader may find it useful to refer to Proposition~\ref{prop:3PathStructHitsDirTriAndDblEdg} below for a summary of the non-injectivity conditions for $3$-path structure maps, and to refer to Example~\ref{ex:3PathInjectivityFailures} for a visual component.

\begin{proof}
    Without loss of generality, we assume that $\varphi$ is surjective. Hence, $\varphi(\mathbb{H}) = G$. Then applying Lemma~\ref{lem:3PathStructureStraongMap}, as $V_{\mathbb{H}}=V_{\overline{\mathbb{H}}}$, a directed triangle or double edge in $G$ that is not the $\varphi$-image of a directed triangle in $\mathbb{H}$, can only exist if $\varphi$ is not injective on vertices.
    We will consider each possibility from the following exhaustive list in turn.
    \begin{enumerate}[(i)]
        \item 
        We have $\varphi(T)=\varphi(H)$.
        \item 
        We have $\varphi(u_i) = \varphi(T)$ for some $i\in \mathcal{I}$, or $\varphi(v_j) = \varphi(H)$ for some $j \in \mathcal{J}$.
        \item 
        We have $\varphi(u_i) = \varphi(H)$ for some $i\in \mathcal{I}$, or $\varphi(v_j) = \varphi(T)$ for some $j \in \mathcal{J}$.
        \item 
        We have $\varphi(u_{i_1}) = \varphi(u_{i_2})$ for some $i_1\neq i_2 \in \mathcal{I}$, or $\varphi(v_{j_1}) = \varphi(v_{j_2})$ for some $j_1 \neq j_2 \in \mathcal{J}$.
        \item 
        We have $\varphi(u_i) = \varphi(v_j)$ for some $i\in \mathcal{I}$ and $j \in \mathcal{J}$.
    \end{enumerate}
    We show that any combination of possible cases does not conflict with the statement of the theorem.

    We first note that condition (i) may hold for any $\mathbb{H}$, and if $\varphi$ is injective on $V_{\mathbb{H}} \setminus \{ T,H \}$, then it is straightforward to see that the statement of the lemma holds (refer to Example~\ref{ex:3PathInjectivityFailures} for illustrations). In particular, regarding the second part of (1), for $x\in\Omega_2(\varphi(\mathbb{H});R)$ that is not in the $\varphi_\#$-image of any directed triangle in $\mathbb{H}$, there are no paths $e_{w_0,w_1,w_2,w_3}\in\Omega_2(\mathbb{H};R)$ such that $e_{\varphi(w_0),\varphi(w_1),\varphi(w_2)}$ or $e_{\varphi(w_1),\varphi(w_2),\varphi(w_3)}$ is equal to $x$. Hence, $x \neq  \pm \delta^h_{3,u}(\varphi_{\#}(\textfrak{H}))$ and $x \neq \pm \delta^t_{3,u}(\varphi_{\#}(\textfrak{H}))$ for any $u \in V_G$.
    
    Case (ii) cannot occur, as $(T,u_i), (v_j,H)  \in E_{\overline{\mathbb{H}}}$ and for all $i \in \mathcal{I}$ and $j \in \mathcal{J}$, and $\varphi$ is a strong digraph map on $\overline{\mathbb{H}}$ by Lemma~\ref{lem:3PathStructureStraongMap}.

    Consider now case (iii). If $\varphi(u_i)=\varphi(H)$, then $\varphi$ would not preserve any directed square of the form $\pm (e_{u_i,v,H}-e_{u_i,v',H})$ for $v \neq v' \in E_{\mathbb{H}}$. Since this square is contained in $\overline{\mathbb{H}}$, this would contradict the assumption that $\varphi$ is a $3$-path structure map. Similarly, if $\varphi(v_j)=\varphi(H)$, then $\varphi$ would not preserve any directed square of the form $\pm (e_{T,u,v_j}-e_{T,u',v_j})$ for $u \neq u' \in E_{\mathbb{H}}$, again contradicting the assumption that $\varphi$ is a $3$-path structure map.
    Thus the only possibilities are that $\mathbb{H}=\mathbb{L}_t$ with $i=t$, $\mathbb{H}=\mathbb{L}_t^u$ with $i=1$, $\mathbb{H}=\mathbb{L}_t^u$ with $i=t+1$, $\mathbb{H}=\mathbb{L}_t$ with $j=1$, $\mathbb{H}=\mathbb{L}_t^l$ with $j=1$, or $\mathbb{H}=\mathbb{L}_t^l$ with $j=t+1$ for some $t = 1,2,\dots$.
    However, note that if $\mathbb{H}=\mathbb{L}_1^u$, then it's not possible that $\varphi(u_1)=\varphi(H)=\varphi(u_2)$, as in this case $\varphi$ would not preserve the directed square $e_{T,u_1,v_1}-e_{T,u_2,v_1}$. Similarly, if $\mathbb{H}=\mathbb{L}_1^l$, then it's not possible that $\varphi(v_1)=\varphi(H)=\varphi(v_2)$, as in this case $\varphi$ would not preserve the directed square $e_{u_1,v_1,H}-e_{u_1,v_2,H}$.

    Now, conditions (i) and (iii) cannot hold simultaneously as this would result in $\varphi$ failing to be a strong digraph map on $\overline{\mathbb{H}}$. Conditions (iii) and (iv) can hold simultaneously, but only when $\mathbb{H} = \mathbb{L}_t^u$ or $\mathbb{L}_t^l$ and condition (iii) is met twice, in which case condition (iv) is implied by condition (iii).
    When condition (v) is not met, it is straightforward to see that the statement of the lemma holds when condition (iii) does (refer to Example~\ref{ex:3PathInjectivityFailures} for illustrations). 

    For case (iv), the injectivity conditions from parts (4), (5), (6) of Definition~\ref{def:3PathStructureMap} imply that, up to symmetry, the only possibilities are either $\mathbb{H}=\mathbb{L}_t^u$ with $i_1 = 1$ and $i_2 = t+1$, or $\mathbb{H}=\mathbb{L}_t^u$ with $j_1 = 1$ and $j_2 = t+1$ for some $t\geq1$. However, note that the case $t=1$ cannot occur, as the directed squares $\pm(e_{T,u_1,H}-e_{T,u_{2},H})$ and $\pm(e_{T,v_1,H}-e_{T,v_{2},H})$ are contained in $\overline{\mathbb{H}}$ and would not be preserved by $\varphi$.
    In any case, it is straightforward to see that the statement of the lemma holds for either of the possibilities above under conditions (i) or (iii) (refer to Example~\ref{ex:3PathInjectivityFailures} for illustrations).

    In the case of part (v), no directed square of the form $\pm (e_{u_i,v,H}-e_{u_i,v',H})$ for $v \neq v' \in E_{G}$ can be preserved by $\varphi$, as $(\varphi(v_j)=\varphi(u_i), \varphi(H))\in E_{G}$.
    Similarly, no directed square of the form $\pm (e_{T,u,v_j}-e_{T,u',v_j})$ for $u \neq u' \in E_{G}$ can be preserved by $\varphi$, as $(\varphi(T), \varphi(v_j)=\varphi(u_i))\in E_{G}$. 
    Therefore, combining the restrictions derived above, the only possibility is that $\mathbb{H}=\mathbb{L}_t$ with $i=t$ and $j=1$.
    In particular, conditions (v) and (iv) cannot hold simultaneous, as they can only occur for different domains $\mathbb{H}$.
    Furthermore, it cannot be the case that condition (v) holds when $t=1$, or that condition (iii) holds alongside condition (v), as either would result in $\varphi$ failing to be a strong digraph map on $\overline{\mathbb{H}}$ contradicting Lemma~\ref{lem:3PathStructureStraongMap}.
    Finally, in the settings when condition (v) holds, even when combined with condition (i), the statement of the lemma is easily verified (refer to Example~\ref{ex:3PathInjectivityFailures} for illustrations).
    \qed
\end{proof}

It is worth noting that during the last proof, a full description of the conditions under which a $3$-path structure map may fail to be injective on vertices was given. Although they will not be used in the remainder of the section, it seems worthwhile to record these results, as they eventually provide a complete geometric understanding of the path chains in $\Omega_3(G;R)$.

\begin{proposition}\label{prop:3PathStructHitsDirTriAndDblEdg}
    Let $\varphi \colon \mathbb{H} \to G$ be a $3$-path structure map with $3$-path element $\textfrak{H}\in\Omega_3(\mathbb{H};R)$. If $\varphi$ fails to be injective on vertices, then this may occur only as one or more of the following possibilities.
    \begin{enumerate}[(a)]
        \item 
            $\mathbb{H}$ is any $3$-path sequence digraph and $\varphi(T)=\varphi(H)$.
        \item
            $\mathbb{H}=\mathbb{L}_t$ for some $t\geq1$ and at least one of $\varphi(u_t) = H$ or $\varphi(v_1) = T$.
        \item
            $\mathbb{H}=\mathbb{L}_t^u$ for some $t\geq1$ and at least one of $\varphi(u_1) = \varphi(H)$ or $\varphi(u_{t+1}) = \varphi(H)$, but not both if $t=1$.
        \item
            $\mathbb{H}=\mathbb{L}_t^l$ for some $t\geq1$ and at least one of $\varphi(v_1) = \varphi(T)$ or $\varphi(v_{t+1}) = \varphi(T)$, but not both if $t=1$.
        \item 
            $\mathbb{H}=\mathbb{L}_{t'}^u$ for some $t'\geq2$ and $\varphi(u_1) = \varphi(u_{t'+1})$.
        \item 
            $\mathbb{H}=\mathbb{L}_{t'}^l$ for some $t'\geq2$ and $\varphi(v_1) = \varphi(v_{t'+1})$.
        \item
            $\mathbb{H}=\mathbb{L}_{t'}$ for some $t'\geq2$ and $\varphi(u_{t'}) = \varphi(v_1)$.
    \end{enumerate}
    The possible simultaneous combinations of these conditions that $\varphi$ may realise are summarised in the following tables.
    \begin{center}
    \begin{tabular}{c|ccc}
        $\mathbb{H}=\mathbb{L}_t$ & (b) & (g)\\
         \hline
         (a) & \xmark & \cmark \\
         (b) & \cellcolor{shade} & \xmark \\
    \end{tabular}
    \;\;\;\;\;\;\;\;\;\;\;\;
    \begin{tabular}{c|ccc}
        $\mathbb{H}=\mathbb{L}_t^u$ & (c) & (e) \\
         \hline
         (a) & \cmark & \cmark \\
         (c) & \cellcolor{shade} & \cmark \\
    \end{tabular}
    \;\;\;\;\;\;\;\;\;\;\;\;
    \begin{tabular}{c|ccc}
        $\mathbb{H}=\mathbb{L}_t^l$ & (d) & (f) \\
         \hline
         (a) & \cmark & \cmark \\
         (d) & \cellcolor{shade} & \cmark \\
    \end{tabular}
    \end{center}
    In particular, neither the combination
    (a),(c),(e), nor the combination (a),(d),(f) can hold simultaneously.
    \qed
\end{proposition}

We include some illustrations of the non-injectivity conditions given in Proposition~\ref{prop:3PathStructHitsDirTriAndDblEdg} in the following example. Although tangential to the main developments of this section, the diagrams below may serve as a useful reference while reading the proof of Theorem~\ref{thm:Omega3Explicitly}.

\begin{exmp}\label{ex:3PathInjectivityFailures}
    The following diagrams give examples of the non-injectivity conditions from part (a) of Proposition~\ref{prop:3PathStructHitsDirTriAndDblEdg} when $\mathbb{H}=\mathbb{T}_2$, $\mathbb{T}_3$, $\mathbb{L}_1$, $\mathbb{L}_2$, $\mathbb{L}_1^u$, and $\mathbb{L}_2^u$:
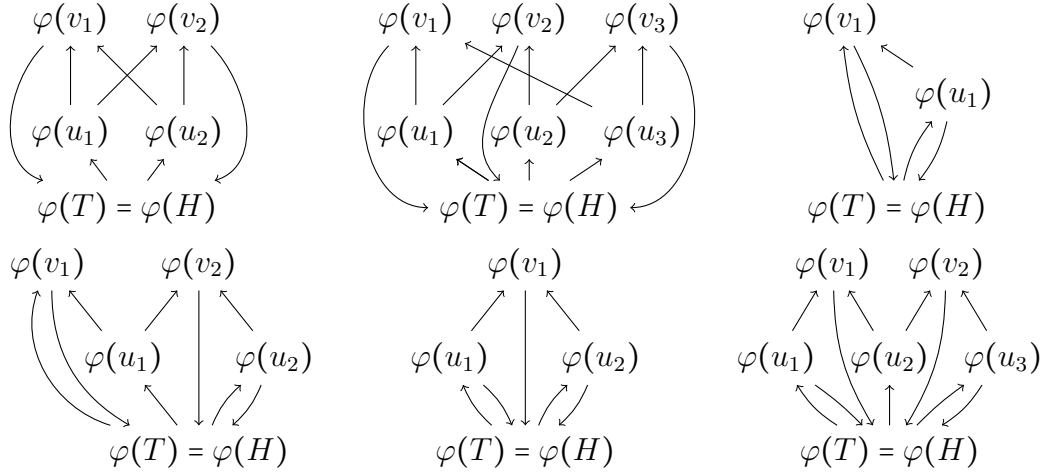
    \begin{center}
        \tikz {
            \node (T) at (0,0) {$\varphi(T)=\varphi(H)$};
            \node (u1) at (-0.75,1) {$\varphi(u_1)$};
            \node (u2) at (0.75,1) {$\varphi(u_2)$};
            \node (v1) at (-0.75,2.5) {$\varphi(v_1)$};
            \node (v2) at (0.75,2.5) {$\varphi(v_2)$};
            \draw[->] (T) -- (u1);
            \draw[->] (T) -- (u2);
            \draw[->] (u1) -- (v1);
            \draw[->] (u1) -- (v2);
            \draw[->] (u2) -- (v2);
            \draw[->] (u2) -- (v1);
            \draw[->] (v1) to [out=225,in=162.5] (T);
            \draw[->] (v2) to [out=-45,in=17.5] (T);
        }
        \;\;\;\;\;\;
        \tikz {
            \node (T) at (0.75,0) {$\varphi(T)=\varphi(H)$};
            \node (u1) at (-0.75,1) {$\varphi(u_1)$};
            \node (u2) at (0.75,1) {$\varphi(u_2)$};
            \node (u3) at (2.25,1) {$\varphi(u_3)$};
            \node (v1) at (-0.75,2.5) {$\varphi(v_1)$};
            \node (v2) at (0.75,2.5) {$\varphi(v_2)$};
            \node (v3) at (2.25,2.5) {$\varphi(v_3)$};
            \draw[->] (T) -- (u1);
            \draw[->] (T) -- (u2);
            \draw[->] (T) -- (u1);
            \draw[->] (T) -- (u3);
            \draw[->] (u1) -- (v1);
            \draw[->] (u1) -- (v2);
            \draw[->] (u2) -- (v2);
            \draw[->] (u2) -- (v3);
            \draw[->] (u3) -- (v3);
            \draw[->] (u3) -- (v1);
            \draw[->] (v1) to [out=225,in=180] (T);
            \draw[->] (v3) to [out=-45,in=0] (T);
            \draw[->] (v2) to [out=-112.5,in=137.5] (T);
        }
        \;\;\;\;\;\;
        \tikz {
            \node (T) at (0,0) {$\varphi(T)=\varphi(H)$};
            \node (u1) at (0.75,1.5) {$\varphi(u_1)$};
            \node (v1) at (-0.75,2.5) {$\varphi(v_1)$};
            \draw[->] (T) to[bend left=12.5] (u1);
            \draw[->] (u1) to[bend left=12.5] (T);
            \draw[->] (T) to[bend left=10] (v1);
            \draw[->] (v1) to[bend left=10] (T);
            \draw[->] (u1) -- (v1);
        }
        \;\;\;\;\;\;
        \tikz {
            \node (T) at (2,0) {$\varphi(T)=\varphi(H)$};
            \node (u1) at (1,1.25) {$\varphi(u_1)$};
            \node (u2) at (3,1.25) {$\varphi(u_2)$};
            \node (v1) at (0,2.5) {$\varphi(v_1)$};
            \node (v2) at (2,2.5) {$\varphi(v_2)$};
            \draw[->] (T) -- (u1);
            \draw[->] (T) to [out=163,in=-110] (v1);
            \draw[->] (v1) to[bend right=30] (T);
            \draw[->] (u1) -- (v1);
            \draw[->] (u1) -- (v2);
            \draw[->] (u2) -- (v2);
            \draw[->] (v2) -- (T);
            \draw[->] (T) to[bend left=12.5] (u2);
            \draw[->] (u2) to[bend left=12.5] (T);
        }
        \;\;\;\;\;\;
        \tikz{
            \node (T) at (0,0) {$\varphi(T)=\varphi(H)$};
            \node (u1) at (-1,1.25) {$\varphi(u_1)$};
            \node (u2) at (1,1.25) {$\varphi(u_2)$};
            \node (v1) at (0,2.5) {$\varphi(v_1)$};
            \draw[->] (u1) -- (v1);
            \draw[->] (u2) -- (v1);
            \draw[->] (u1) to[bend left=12.5] (T);
            \draw[->] (T) to[bend left=12.5] (u1);
            \draw[->] (u2) to[bend left=12.5] (T);
            \draw[->] (T) to[bend left=12.5] (u2);
            \draw[->] (v1) -- (T);
        }
        \;\;\;\;\;\;
        \tikz{
            \node (T) at (2.25,0) {$\varphi(T)=\varphi(H)$};
            \node (u1) at (0.75,1.25) {$\varphi(u_1)$};
            \node (u2) at (2.25,1.25) {$\varphi(u_2)$};
            \node (u3) at (3.75,1.25) {$\varphi(u_3)$};
            \node (v1) at (1.5,2.5) {$\varphi(v_1)$};
            \node (v2) at (3,2.5) {$\varphi(v_2)$};
            \draw[->] (T) -- (u2);
            \draw[->] (u1) -- (v1);
            \draw[->] (u2) -- (v1);
            \draw[->] (u2) -- (v2);
            \draw[->] (u3) -- (v2);
            \draw[->] (u1) to[bend left=5] (T);
            \draw[->] (T) to[bend left=12.5] (u1);
            \draw[->] (u3) to[bend left=12.5] (T);
            \draw[->] (T) to[bend left=5] (u3);
            \draw[->] (v1) to[bend right=15] (T);
            \draw[->] (v2) to[bend left=15] (T);
        }
    \end{center}

    The following are examples of the conditions in parts (b), (c), and (e) of Proposition~\ref{prop:3PathStructHitsDirTriAndDblEdg} when $(i)$ $\mathbb{H} = \mathbb{L}_3$ and $\varphi(v_1) = \varphi(T)$, $(ii)$ $\mathbb{H} = \mathbb{L}_3$, $\varphi(u_1) = \varphi(H)$, and $\varphi(v_3) = \varphi(T)$, and $(iii)$ $\mathbb{H} = \mathbb{L}_3^u$, and $\varphi(u_1) = \varphi(u_3) = H$.
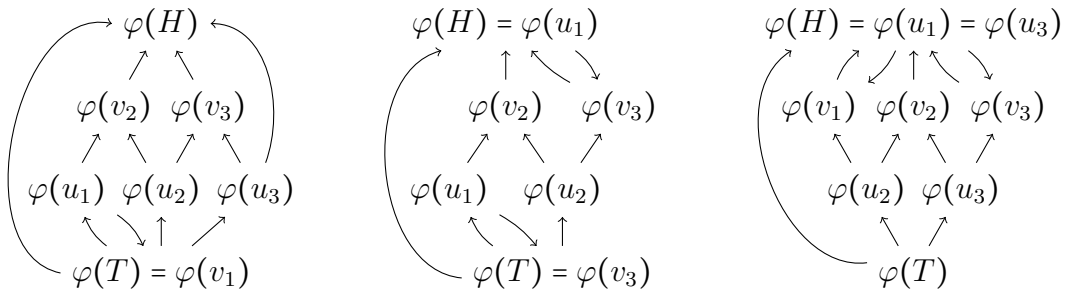
    \begin{center}
        \tikz{
            \node (T) at (1.25,0) {$\varphi(T)=\varphi(v_1)$};
            \node (u1) at (0,1.1) {$\varphi(u_1)$};
            \node (u2) at (1.25,1.1) {$\varphi(u_2)$};
            \node (u3) at (2.5,1.1) {$\varphi(u_3)$};
            \node (v2) at (0.625,2.2) {$\varphi(v_2)$};
            \node (v3) at (1.875,2.2) {$\varphi(v_3)$};
            \node (H) at (1.25,3.3) {$\varphi(H)$};
            \draw[->] (u1) to[bend left=15] (T);
            \draw[->] (T) to[bend left=15] (u1);
            \draw[->] (T) -- (u2);
            \draw[->] (T) -- (u3);
            \draw[->] (u1) -- (v2);
            \draw[->] (u2) -- (v2);
            \draw[->] (u2) -- (v3);
            \draw[->] (u3) -- (v3);
            \draw[->] (v2) -- (H);
            \draw[->] (v3) -- (H);
            \draw[->] (u3) to [out=70,in=0] (H);
            \draw[->] (T) to[bend left=90] (H);
        }
        \;\;\;
        \tikz{
            \node (T) at (2.25,0) {$\varphi(T)=\varphi(v_3)$};
            \node (u1) at (0.75,1.1) {$\varphi(u_1)$};
            \node (u2) at (2.25,1.1) {$\varphi(u_2)$};
            \node (v2) at (1.5,2.2) {$\varphi(v_2)$};
            \node (v3) at (3,2.2) {$\varphi(v_3)$};
            \node (H) at (1.5,3.3) {$\varphi(H)=\varphi(u_1)$};
            \draw[->] (u1) to[bend left=10] (T);
            \draw[->] (T) to[bend left=15] (u1);
            \draw[->] (T) -- (u2);
            \draw[->] (u1) -- (v2);
            \draw[->] (u2) -- (v2);
            \draw[->] (u2) -- (v3);
            \draw[->] (v3) to[bend left=10] (H);
            \draw[->] (H) to[bend left=15] (v3);
            \draw[->] (T) to[bend left=80] (H);
            \draw[->] (v2) -- (H);
        }
        \;\;\;
        \tikz {
            \node (T) at (1.25,0) {$\varphi(T)$};
            \node (u2) at (0.625,1.1) {$\varphi(u_2)$};
            \node (u3) at (1.875,1.1) {$\varphi(u_3)$};
            \node (v1) at (0,2.2) {$\varphi(v_1)$};
            \node (v2) at (1.25,2.2) {$\varphi(v_2)$};
            \node (v3) at (2.5,2.2) {$\varphi(v_3)$};
            \node (H) at (1.25,3.3) {$\varphi(H)=\varphi(u_1)=\varphi(u_3)$};
            \draw[->] (T) -- (u2);
            \draw[->] (T) -- (u3);
            \draw[->] (u2) -- (v1);
            \draw[->] (u2) -- (v2);
            \draw[->] (u3) -- (v2);
            \draw[->] (u3) -- (v3);
            \draw[->] (v2) -- (H);
            \draw[->] (v1) to[bend left=15] (H);
            \draw[->] (H) to[bend left=15] (v1);
            \draw[->] (v3) to[bend left=15] (H);
            \draw[->] (H) to[bend left=15] (v3);
            \draw[->] (T) to[bend left=77.5] (H);
        }
    \end{center}
    
    The following diagrams show the combination of the conditions in part (e) and part (a) or (c) of Proposition~\ref{prop:3PathStructHitsDirTriAndDblEdg}.
    In the first two images, $\mathbb{H} = \mathbb{L}_2^u$ and $\mathbb{H} = \mathbb{L}_3^u$, with (e) holding. In the third and fourth images, $\mathbb{H} = \mathbb{L}_2^u$ and $\mathbb{H} = \mathbb{L}_3^u$, with (e) and (a) holding. In the last two images, $\mathbb{H} = \mathbb{L}_2^u$ and $\mathbb{H} = \mathbb{L}_3^u$, with (e) and (c) holding.
    
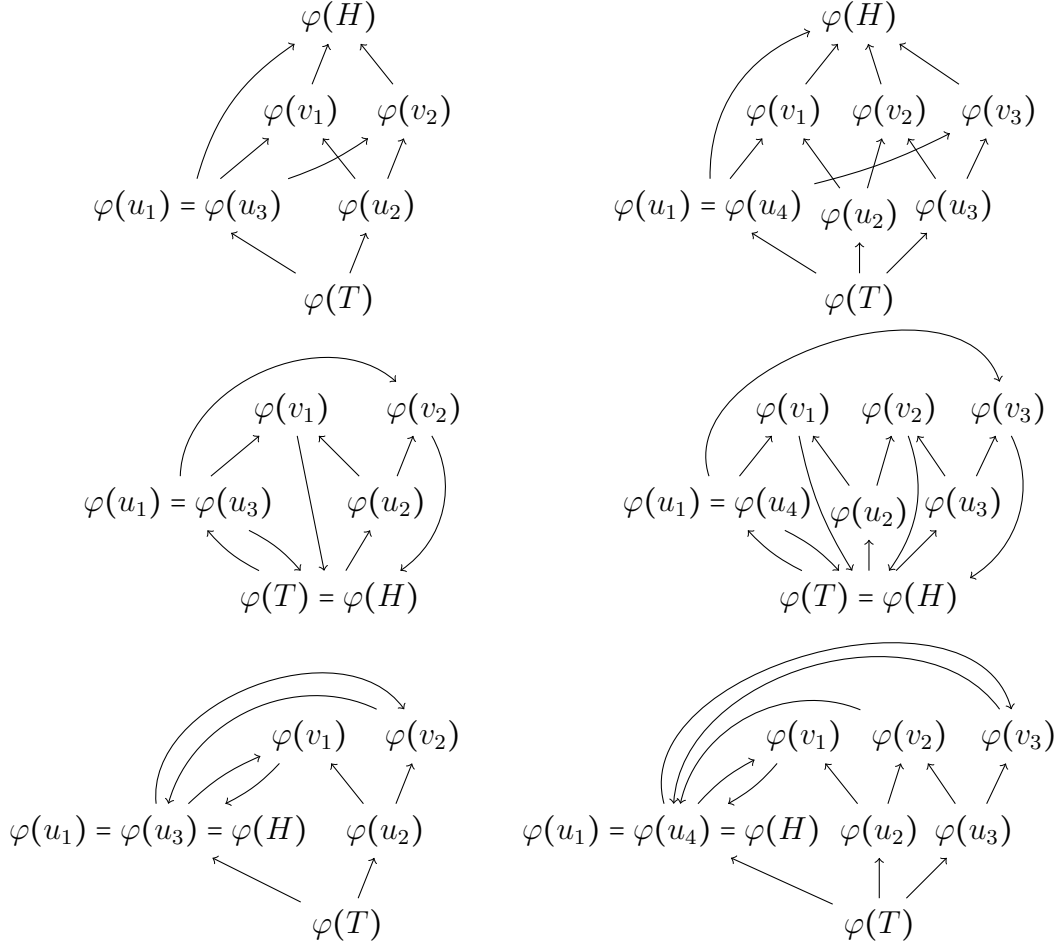
\begin{center}
        \tikz {
            \node (T) at (1.5,0) {$\varphi(T)$};
            \node (u1) at (-0.5,1.25) {$\varphi(u_1)=\varphi(u_3)$};
            \node (u2) at (2,1.25) {$\varphi(u_2)$};
            \node (v1) at (1,2.5) {$\varphi(v_1)$};
            \node (v2) at (2.5,2.5) {$\varphi(v_2)$};
            \node (H) at (1.5,3.75) {$\varphi(H)$};
            \draw[->] (T) -- (u1);
            \draw[->] (T) -- (u2);
            \draw[->] (u1) -- (v1);
            \draw[->] (u2) -- (v1);
            \draw[->] (u2) -- (v2);
            \draw[->] (v1) -- (H);
            \draw[->] (v2) -- (H);
            \draw[->] (u1) to[bend left=20] (H);
            \draw[->] (u1) to[bend right=7.5] (v2);
        }
        \;\;\;\;\;\; \;\;\;\;\;\;
        \tikz {
            \node (T) at (1,0) {$\varphi(T)$};
            \node (u1) at (-1,1.25) {$\varphi(u_1)=\varphi(u_4)$};
            \node (u2) at (1,1.125) {$\varphi(u_2)$};
            \node (u3) at (2.25,1.25) {$\varphi(u_3)$};
            \node (v1) at (0,2.5) {$\varphi(v_1)$};
            \node (v2) at (1.4,2.5) {$\varphi(v_2)$};
            \node (v3) at (2.845,2.5) {$\varphi(v_3)$};
            \node (H) at (1,3.75) {$\varphi(H)$};
            \draw[->] (T) -- (u1);
            \draw[->] (T) -- (u2);
            \draw[->] (T) -- (u3);
            \draw[->] (u1) -- (v1);
            \draw[->] (u2) -- (v1);
            \draw[->] (u2) -- (v2);
            \draw[->] (u3) -- (v2);
            \draw[->] (u3) -- (v3);
            \draw[->] (v1) -- (H);
            \draw[->] (v2) -- (H);
            \draw[->] (v3) -- (H);
            \draw[->] (u1) to[bend left=35] (H);
            \draw[->] (u1) to[bend right=6.25] (v3);
        }
    \end{center}
    \vspace{-1.6cm}
    \begin{center}
        \tikz {
            \node (T) at (1.5,0) {$\varphi(T)=\varphi(H)$};
            \node (u1) at (-0.5,1.25) {$\varphi(u_1)=\varphi(u_3)$};
            \node (u2) at (2.25,1.25) {$\varphi(u_2)$};
            \node (v1) at (1,2.5) {$\varphi(v_1)$};
            \node (v2) at (2.75,2.5) {$\varphi(v_2)$};
            \draw[->] (T) -- (u2);
            \draw[->] (u1) -- (v1);
            \draw[->] (u2) -- (v1);
            \draw[->] (u2) -- (v2);
            \draw[->] (v1) -- (T);
            \draw[->] (u1) to[bend left=11] (T);
            \draw[->] (T) to[bend left=11] (u1);
            \draw[->] (u1) to[bend left=65] (v2);
            \draw[->] (v2) to[bend left=42.5] (T);
        }
        \;\;\;\;\;\; \;\;\;\;\;\;
        \tikz {
            \node (T) at (1,0) {$\varphi(T)=\varphi(H)$};
            \node (u1) at (-1,1.25) {$\varphi(u_1)=\varphi(u_4)$};
            \node (u2) at (1,1.125) {$\varphi(u_2)$};
            \node (u3) at (2.25,1.25) {$\varphi(u_3)$};
            \node (v1) at (0,2.5) {$\varphi(v_1)$};
            \node (v2) at (1.4,2.5) {$\varphi(v_2)$};
            \node (v3) at (2.845,2.5) {$\varphi(v_3)$};
            \draw[->] (T) -- (u2);
            \draw[->] (T) -- (u3);
            \draw[->] (u1) -- (v1);
            \draw[->] (u2) -- (v1);
            \draw[->] (u2) -- (v2);
            \draw[->] (u3) -- (v2);
            \draw[->] (u3) -- (v3);
            \draw[->] (v1) to[bend right=10] (T);
            \draw[->] (v2) to[bend left=27.5] (T);
            \draw[->] (v3) to[bend left=42.5] (T);
            \draw[->] (u1) to[bend left=10] (T);
            \draw[->] (T) to[bend left=10] (u1);
            \draw[->] (u1) to[bend left=90] (v3);
        }
    \end{center}
    \vspace{-1.6cm}
    \begin{center}
        \tikz {
            \node (T) at (1.5,0) {$\varphi(T)$};
            \node (u1) at (-1,1.25) {$\varphi(u_1)=\varphi(u_3) = \varphi(H)$};
            \node (u2) at (2,1.25) {$\varphi(u_2)$};
            \node (v1) at (1,2.5) {$\varphi(v_1)$};
            \node (v2) at (2.5,2.5) {$\varphi(v_2)$};
            \draw[->] (T) -- (u1);
            \draw[->] (T) -- (u2);
            \draw[->] (u1) to[bend left=10] (v1);
            \draw[->] (v1) to[bend left=10] (u1);
            \draw[->] (u2) -- (v1);
            \draw[->] (u2) -- (v2);
            \draw[->] (u1) to[bend left=75] (v2);
            \draw[->] (v2) to[bend right=50] (u1);
        }
        \;\;
        \tikz {
            \node (T) at (1,0) {$\varphi(T)$};
            \node (u1) at (-1.75,1.25) {$\varphi(u_1)=\varphi(u_4)=\varphi(H)$};
            \node (u2) at (1,1.25) {$\varphi(u_2)$};
            \node (u3) at (2.25,1.25) {$\varphi(u_3)$};
            \node (v1) at (0,2.5) {$\varphi(v_1)$};
            \node (v2) at (1.4,2.5) {$\varphi(v_2)$};
            \node (v3) at (2.845,2.5) {$\varphi(v_3)$};
            \draw[->] (T) -- (u1);
            \draw[->] (T) -- (u2);
            \draw[->] (T) -- (u3);
            \draw[->] (u1) to[bend left=10] (v1);
            \draw[->] (v1) to[bend left=10] (u1);
            \draw[->] (u2) -- (v1);
            \draw[->] (u2) -- (v2);
            \draw[->] (u3) -- (v2);
            \draw[->] (u3) -- (v3);
            \draw[->] (v2) to[bend right=50] (u1);
            \draw[->] (u1) to[bend left=90] (v3);
            \draw[->] (v3) to[bend right=70] (u1);
        }
        \;\;\;\;\;\;
    \end{center}
    
    The following examples show the combination of the conditions in parts (g) and (a) of Proposition~\ref{prop:3PathStructHitsDirTriAndDblEdg}. Here the diagrams show $(i)$ $\mathbb{H} = \mathbb{L}_2$ with (g) holding, $(ii)$ $\mathbb{H} = \mathbb{L}_3$ with (g) holding, $(iii)$ $\mathbb{H} = \mathbb{L}_2$ with (a) and (g) holding, and $(iv)$ $\mathbb{H} = \mathbb{L}_3$ with (a) and (g) holding.
    
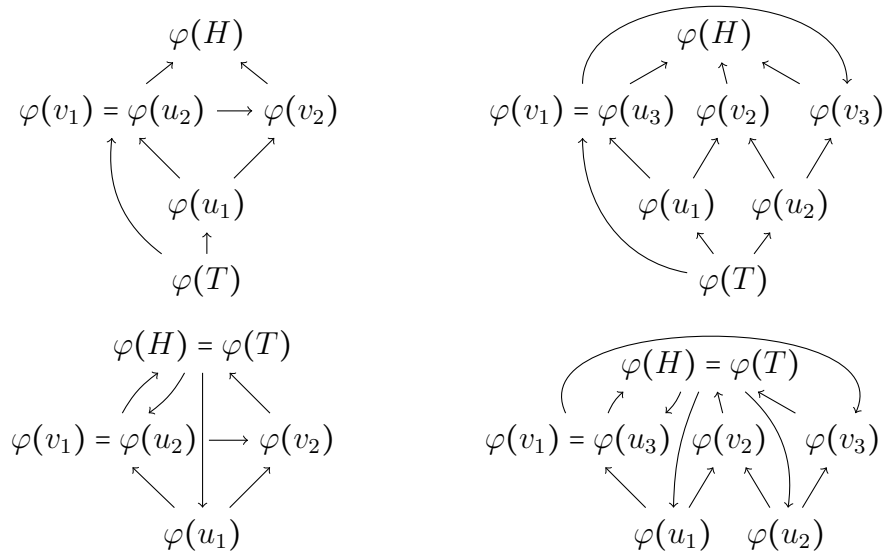
\begin{center}
        \tikz {
            \node (T) at (1,0.25) {$\varphi(T)$};
            \node (u1) at (1,1.25) {$\varphi(u_1)$};
            \node (v1) at (-0.25,2.5) {$\varphi(v_1)=\varphi(u_2)$};
            \node (v2) at (2.25,2.5) {$\varphi(v_2)$};
            \node (H) at (1,3.5) {$\varphi(H)$};
            \draw[->] (T) -- (u1);
            \draw[->] (T) to[bend left=30] (v1);
            \draw[->] (u1) -- (v1);
            \draw[->] (u1) -- (v2);
            \draw[->] (v1) -- (H);
            \draw[->] (v2) -- (H);
            \draw[->] (v1) -- (v2);
        }
        \;\;\;\;\;\;\;\;\;\;\;\;
        \tikz {
            \node (T) at (1.5,0.25) {$\varphi(T)$};
            \node (u1) at (0.75,1.25) {$\varphi(u_1)$};
            \node (u2) at (2.25,1.25) {$\varphi(u_2)$};
            \node (v1) at (-0.5,2.5) {$\varphi(v_1)=\varphi(u_3)$};
            \node (v2) at (1.5,2.5) {$\varphi(v_2)$};
            \node (v3) at (3,2.5) {$\varphi(v_3)$};
            \node (H) at (1.25,3.5) {$\varphi(H)$};
            \draw[->] (T) -- (u1);
            \draw[->] (T) -- (u2);
            \draw[->] (T) to [out=170,in=270] (v1);
            \draw[->] (u1) -- (v1);
            \draw[->] (u1) -- (v2);
            \draw[->] (u2) -- (v2);
            \draw[->] (u2) -- (v3);
            \draw[->] (v1) -- (H);
            \draw[->] (v2) -- (H);
            \draw[->] (v3) -- (H);
            \draw[->] (v1) to[bend left=90] (v3);
        }
        \\
        \tikz {
            \node (u1) at (1,1.25) {$\varphi(u_1)$};
            \node (v1) at (-0.3,2.5) {$\varphi(v_1)=\varphi(u_2)$};
            \node (v2) at (2.25,2.5) {$\varphi(v_2)$};
            \node (H) at (1,3.75) {$\varphi(H)=\varphi(T)$};
            \draw[->] (u1) -- (v1);
            \draw[->] (u1) -- (v2);
            \draw[->] (v1) to[bend left=12.5] (H);
            \draw[->] (H) to[bend left=12.5] (v1);
            \draw[->] (v2) -- (H);
            \draw[->] (v1) -- (v2);
            \draw[->] (H) -- (u1);
        }
        \;\;\;\;\;\;\;\;\;\;\;\;
        \tikz {
            \node (u1) at (0.75,1.25) {$\varphi(u_1)$};
            \node (u2) at (2.25,1.25) {$\varphi(u_2)$};
            \node (v1) at (-0.5,2.5) {$\varphi(v_1)=\varphi(u_3)$};
            \node (v2) at (1.5,2.5) {$\varphi(v_2)$};
            \node (v3) at (3,2.5) {$\varphi(v_3)$};
            \node (H) at (1.25,3.5) {$\varphi(H)=\varphi(T)$};
            \draw[->] (u1) -- (v1);
            \draw[->] (u1) -- (v2);
            \draw[->] (u2) -- (v2);
            \draw[->] (u2) -- (v3);
            \draw[->] (v1) to[bend left=12.5] (H);
            \draw[->] (H) to[bend left=12.5] (v1);
            \draw[->] (v2) -- (H);
            \draw[->] (v3) -- (H);
            \draw[->] (v1) to[bend left=115] (v3);
            \draw[->] (H) to[bend left=28] (u2);
            \draw[->] (H) to[bend right=11] (u1);
        }
    \end{center}
\end{exmp}

We now state and prove the uniqueness theorem for $3$-path structure maps.

\begin{theorem}\label{thm:Omega3Explicitly}
    For any inductive element $I \in \Omega_3(G;R)$ there is a $3$-path structure map $\varphi\colon\mathbb{H}\rightarrow G$ such that $\varphi_{\#}(\mathfrak{H})=\pm I$. Moreover, this $3$-path structure map is unique up to the $\mathrm{Sym}(\mathbb{H})$ action on its domain.
\end{theorem}

\begin{proof}
    We want to show that in all cases the $3$-path structure maps $\varphi_I$ constructed in Theorem~\ref{thm:3PathStructureConstruction} are unique up to sign and the specified group actions.
    We begin by showing that the type of $3$-path sequence digraph $\mathbb{H}_I$ in the domain of $\varphi_I$
    is uniquely determined up to the specified conditions and then verify that there is a unique value of $t$ or $t'$ for which a $3$-path structure map can exist in the case of each type of $3$-path sequence digraph. Justifying that $\varphi_I$ is unique up to the $\mathrm{Sym}(\mathbb{H}_I)$ action on its domain will then complete the proof. Throughout, we assume that $I$ has an inductive structure $F_2$ with $m$ vertices. 
    
    An observation we use below is that the $3$-path sequence digraphs can be characterised by their directed triangles. The trapezohedra contain no directed triangles, while any other $3$-path sequence digraph
    contains $(i)$ either a unique directed triangle $e_{T,u_1,v_1}$ or $e_{u_1,v_1,H}$, and $(ii)$ either a unique directed triangle $e_{T,u_{\max(\mathcal{I})},v_{\max(\mathcal{J})}}$ or $e_{u_{\max(\mathcal{I})},v_{\max(\mathcal{J})},H}$.
    
    To put this observation into effect, suppose that $\mathbb{H}_I = \mathbb{T}_{t'}$ for some $t'\geq2$. Then as $\varphi_I$ is a $3$-path structure map, by Lemma~\ref{lem:3PathStructHitsDirTriAndDblEdg}
    the image of $\varphi_I$ cannot contain any directed triangles or double edges. Since each $3$-path sequence digraph which is not a trapezohedron contains a directed triangle, no $3$-path structure map can be constructed from any of these digraphs onto the image of $\varphi_I$. 
    
    We claim that the same conclusion is true also when $\mathbb{H}_I$ is not a trapezohedron. Thus recall that $\varphi_I$ was constructed in Theorem~\ref{thm:3PathStructureConstruction} so as to satisfy $\varphi_{I\#}(\textfrak{H}_I) = \pm I$, where $\textfrak{H}_I\in\Omega_3(\mathbb{H}_I;R)$ is the $3$-path element of $\mathbb{H}_I$ and assume that a second $3$-path sequence map $\varphi\colon\mathbb{H}\rightarrow G$ satisfying $\varphi_\#(\mathfrak{H}) = \pm I$ is given, where $\mathfrak{H}\in\Omega_3(\mathbb{H};R)$ is the $3$-path element of $\mathbb{H}$.
    By switching the sign of one of the 3-path elements, it will be no loss of generality to assume that $\varphi_\#(\mathfrak{H}) = \varphi_{I\#}(\mathfrak{H}_I)$. We will show that assuming that the $3$-path sequence digraph $\mathbb{H}$ is not of the same type as $\mathbb{H}_I$ leads to a contradiction.
    
    To begin, note that since $\varphi_I,\varphi$ have the same image and $\mathbb{H}_I$ is not a trapezohedron, neither is $\mathbb{H}$. 
    
    Next, note that any directed triangle $x \in \Omega_2(\mathbb{H}_I;R)$ satisfies $x = \pm\delta^h_{3,w}(\textfrak{H}_I)$ or $x = \pm\delta^t_{3,w}(\textfrak{H}_I)$ for some $w \in V_{\mathbb{H}_I}$, so applying Lemma~\ref{lem:3PathStructureStraongMap}, any directed triangle or double edge in $G$ that is the $\varphi_I$-image of a directed triangle in $\mathbb{H}_I$ must be equal to $\pm\delta^h_{3,u}(\varphi_{I\#}(\textfrak{H}_I))$ or $\pm\delta^t_{3,u}(\varphi_{I\#}(\textfrak{H}_I))$ for some $u\in V_G$. But the same argument can be made for $\varphi\colon\mathbb{H}\rightarrow G$, so any directed triangle or double edge in $G$ that is the $\varphi$-image of a directed triangle in $\mathbb{H}$ must be equal to $\pm\delta^h_{3,u}(\varphi_{\#}(\textfrak{H}))$ or $\pm\delta^t_{3,u}(\varphi_{\#}(\textfrak{H}))$ for some $u \in V_G$.
    Therefore, since $\varphi_{I\#}(\textfrak{H}_I)=\varphi_{\#}(\textfrak{H})$ and any $3$-path sequence digraph that is not a trapezohedron contains precisely two directed triangles, Lemma~\ref{lem:3PathStructHitsDirTriAndDblEdg} implies that 
    the $\varphi_{I\#}$-image of each directed triangle in $\mathbb{H}_I$ must coincide with the $\varphi_\#$-image of a directed triangle in $\mathbb{H}$, and vice versa. 
    Moreover, it implies that the images of these directed triangles are either:
    \begin{enumerate}
        \item 
        one double edge, or
        \item
        two directed triangles, two double edges, or one directed triangle and one double edge.
    \end{enumerate}
    
    Now, as per the discussion above, the digraphs $\mathbb{H}_I$ and $\mathbb{H}$ each contain two directed triangles from the following list: 
    \begin{align}\label{eq:directtrianglelist}
    &  e_{T,u_1,v_1}, &&  e_{u_1,v_1,H}, &&  e_{T,u_{\max(\mathcal{I})},v_{\max(\mathcal{J})}}, && e_{u_{\max(\mathcal{I})},v_{\max(\mathcal{J})},H},
    \end{align}
    although we are abusing notation by using the same indexing of the vertices in both digraphs.
    These directed triangles are obtainable as $\delta^h_{3,w}(\textfrak{H}_I)$ for some $w \in V_{\mathbb{H}_I}$ or $\delta^h_{3,w}(\textfrak{H})$ for some $w \in V_{\mathbb{H}}$, 
    and $\mathbb{H}_I$ and $\mathbb{H}$ are distinguished by the unique combination which they contain.
     
    Our task is to show that for some directed triangle $x\in \Omega_2(\mathbb{H}_I;R)$, there is no directed triangle in $y \in \Omega_2(\mathbb{H};R)$ such that ${\varphi_I}_{\#}(x) = \varphi_{\#}(y)$.
    Thus we assume that for each directed triangle $x\in\Omega_2(\mathbb{H}_I;R)$, there is a directed triangle $y\in\Omega_2(\mathbb{H};R)$ such that $\varphi_{I\#}(x)=\varphi_\#(y)$.
    
    We choose directed triangles $x\in \Omega_2(\mathbb{H}_I;R)$ and $y\in \Omega_2(\mathbb{H};R)$ using the following steps. Initially, let $x\in \Omega_2(\mathbb{H}_I;R)$ be either one of the two directed triangles. In the case that $x =  \pm e_{T,u_1,v_1}$, there is a directed triangle $y \in \Omega_2(\mathbb{H};R)$ such that ${\varphi_I}_{\#}(x) = \varphi_{\#}(y)$. If $y =  \pm e_{T,u_1,v_1}$ also, then select $x$ instead as the other directed triangle in $\mathbb{H}_I$ and begin again.
    It is now not possible for $x$ and $y$ to be the same directed triangle from the list~\eqref{eq:directtrianglelist}, since $\mathbb{H}_I$ and $\mathbb{H}$ are assumed to be $3$-path sequences of different types.
    
    Otherwise, continuing to assume that $x =  \pm  e_{T,u_1,v_1}$, if $y = \pm  e_{T, u_{\max(\mathcal{I})}, v_{\max(\mathcal{J})}}$, then again select $x$ instead as the other directed triangle in $\mathbb{H}_I$. Checking all possible cases, we find that $x$ and $y$ are now distinct members of the list~\eqref{eq:directtrianglelist}, and that they are neither of the pairs
    \begin{align*}
        e_{T,u_1,v_1} 
        \; &\text{and} \;
         e_{T,u_{\max(\mathcal{I})},v_{\max(\mathcal{J})}},&
        \;\text{or}&&
        e_{u_1,v_1,H} 
        \; &\text{and} \;
        e_{u_{\max(\mathcal{I})},v_{\max(\mathcal{J})},H},
    \end{align*}
    in any order, with any signs.

    Using the above, we may conclude that the remaining possibility left to rule out is that $\varphi_I$ or $\varphi$ sends a directed triangle of the form $\pm e_{T,u_1,v_1}$ or $\pm e_{T,u_{\max(\mathcal{I})},v_{\max(\mathcal{J})}}$ and a directed triangle of the form $\pm e_{u_1,v_1,H}$ or $\pm e_{u_{\max(\mathcal{I})},v_{\max(\mathcal{J})},H}$ to the same element $z \in \Omega_2(G;R)$.

    Firstly, consider the case when $x = \pm e_{T,u_1,v_1}$, $y=\pm e_{u_1,v_1,H}$, and $\varphi_{I\#}(x) = \varphi_{\#}(y)$. Then $\varphi_I(T)=\varphi(u_1)$ and, since the inductive element $I$ is connected, also $\varphi_I(T) = t(I) = \varphi(T)$. Together, these equations imply that $\varphi(T) = \varphi(u_1)$, and it follows that $\varphi$ sends the edge $T \to u_1$ in $\mathbb{H}$ to a vertex. 
    However, $\mathbb{H}$ is a $3$-path sequence and by Lemma~\ref{lem:3PathStructureStraongMap} $\varphi$ is a strong digraph map on $\overline{\mathbb{H}}$, so there is a contradiction.
    The other possibility above is ruled out similarly, and we leave this task to reader.

    In summary, at this stage we have shown that $\mathbb{H}_I$ and $\mathbb{H}$ must be $3$-path sequences of the same type. It remains to show that $\mathbb{H}_I$ and $\mathbb{H}$ must also have the same number of vertices.
    
    For reference, $\mathbb{H}_I$ and $\mathbb{H}$ belong to one of the following families of digraphs
    \[
    \mathbb{T}_{t'}, \: \mathbb{L}_t, \: \mathbb{L}_t^u, \: \text{or} \: \mathbb{L}_{t}^l
    \]
    where the indices run $t\geq1$ and $t'\geq2$. Let $\mathbb{H}_I$ have index $m_I$ and $\mathbb{H}$ have index $m$. We will assume that $m\neq m_I$.

    Recall that $\textfrak{H}_I\in\Omega_3(\mathbb{H}_I;R)$ and $\mathfrak{H}\in\Omega_3(\mathbb{H};R)$ denote the $3$-path elements described in Definition~\ref{def:LinearEelmetns}. It follows from the injectivity conditions (4), (5), (6) of Definition~\ref{def:3PathStructureMap} that the number of nonzero path summands in $\textfrak{H}_I$ is precisely the same as in $I$. If $m< m_I$, then the number of nonzero path summands in $\varphi_{\#}(\textfrak{H})$ is strictly less than the number $m_I$ of nonzero path summands in $\textfrak{H}_I$ and $I$. Therefore, $\varphi_{\#}(\textfrak{H}) \neq \pm I$.

    On the other hand, suppose that $m> m_I$.
    As there are now strictly more nonzero path summands in $\textfrak{H}$ than in $I$, there are at least two distinct nonzero path summands in $\textfrak{H}$ which $\varphi$ sends to identical paths with opposite signs. However, all nonzero path summands in $\textfrak{H}$ are unique and determined by the vertices $u_i$ and $v_j$ which they contain. Hence it must be the case that $\varphi$ is not injective on $\{ u_i \}_{i \in \mathcal{I}}$ or $\{ v_j \}_{j \in \mathcal{J}}$. This contradicts the assumptions on vertex injectivity from Definition~\ref{def:3PathStructureMap}, even when $\mathbb{H}$ is one of $\mathbb{L}^u_t$ or $\mathbb{L}^l_t$, as here $\varphi$ would be required to be non-injective on an additional vertex to be non-injective on nonzero path summands.

    Finally, we turn to address the uniqueness of $\mathbb{H}_I$ and the $\mathrm{Sym}(\mathbb{H}_I)$ action. Referring to the definition of $\varphi_I$ in equation~\eqref{eq:VetexImages}, when $\mathbb{H}_I = \mathbb{T}_m$, the vertices 
    $\varphi_I(T),\varphi_I(u_i),\varphi_I(v_i)$ and $\varphi_I(H)$ appear in the $3$-paths 
    $e_{\varphi_I(T),\varphi_I(u_i),\varphi_I(v_i),\varphi_I(H)}$ and $e_{\varphi(T),\varphi_I(u_i),\varphi_I(v_{i+1}),\varphi_I(H)}$
    for $i=1,\dots,t$ modulo $t$.
    However, these paths are precisely the nonzero summands of $\pm I$ as the image of $\varphi_{I\#}\colon \Omega_3(\mathbb{T}_m;R) \to \Omega_3(G;R)$.
    Hence, due to the cyclic structure of $F_2$ and the vertex injectivity conditions from Definition~\ref{def:3PathStructureMap}, up to pre-composition with the cyclic group action on the indices, $\varphi_I$ is the unique digraph map whose induced map $\Omega_3(\mathbb{T}_m;R) \to \Omega_3(G;R)$ satisfies ${\varphi_I}_{\#}(T_m)= \pm I$. A similar argument can be made for the uniqueness of $\varphi_I$ up the $C_2$ action when $\mathbb{H}_I \neq \mathbb{T}_m$, which completes the proof.
    \qed
\end{proof}

In Corollaries~\ref{cor:FieldGeneratorBasis}~and~\ref{cor:IductiveBasisIntegral} it was shown that inductive elements generate $\Omega_3(G;R)$ when $R$ has characteristic $0$ or $2$. Thus Theorem~\ref{thm:SructureMapsAreInductive} has the following consequence.

\begin{corollary}\label{cor:3PathStructureGeneration}
    Let $R$ be a ring of characteristic $0$ or $2$. Then the induced images of $3$-path elements under $3$-path structure maps generate $\Omega_3(G;R)$. \qed
\end{corollary}

Using Corollary~\ref{cor:3PathStructureGeneration} we can derive the following less explicit description of $\Omega_3(G;R)$ that directly generalises Grigor'yan's Theorem~\ref{thm:Dim3BasisNoDoubleNoMulti}.

\begin{corollary}\label{cor:Omega3Explicitly}
    Let $R$ be a ring of characteristic $0$ or $2$ and $G$ a digraph. Then there is a generating set for $\Omega_3(G;R)$, the elements of which are obtained as the induced images of a trapezohedron element $T_t$ under a digraph map $\mathbb{T}_t \to G$ for some integer $t \geq 2$.
\end{corollary}

\begin{proof}
    There are digraph maps $m_{as}\colon \mathbb{T}_t \to \mathbb{L}_{t-1}$, $m_s^u\colon \mathbb{T}_t \to \mathbb{L}_{t-1}^u$, $m_s^l\colon \mathbb{T}_t \to \mathbb{L}_{t-1}^l$ determined by 
    \begin{align*}
        m_{as}(T) &= T,&m_{as}(H) &= H, &m_{as}(u_i) &= u_i, &m_{as}(v_i) &= v_i, &m_{as}(u_t) &= T, & m_{as}(v_t) &= u_{t-1},
        \\
        m_s^u(T) &= T, &m_s^u(H) &= H, & m_s^u(u_i) &= u_{i+1}, &m_s^u(v_i) &= v_i, &m_s^u(u_t) &= u_1, & m_s^u(v_t) &= H,
        \\
        m_s^l(T) &= T, &m_s^l(H) &= H, &m_s^l(u_i) &= u_i, &m_s^l(v_i)& = v_{i+1}, &m_s^l(v_t)& = v_1, &m_s^l(u_t)&= T
    \end{align*}
    for $i=1,\dots,t-1$ and one checks that 
    \[
    {m_{as}}_{\#}(T_t) = L_{t-1},\qquad {m_s^u}_{\#}(T_t) = L_{t-1}^u,\qquad {m_s^l}_{\#}(T_t) = L_{t-1}^l.
    \]
    The corollary follows by composing the above maps with the $3$-path structure maps supplied by Corollary~\ref{cor:3PathStructureGeneration}. \qed
\end{proof}

Corollary~\ref{cor:Omega3Explicitly} does not generally produce a basis for $\Omega_3(G;R)$, even up to the cyclic group actions and signs, as $G$ may contain multisquares (see Example~\ref{ex:ExtendedMultisquare}). However, once a representative from each $\mathrm{Sym}(\mathbb{H}_I)$ orbit is selected and a sign chosen on each of these elements, both the generating sets from Corollaries~\ref{cor:3PathStructureGeneration}~and~\ref{cor:Omega3Explicitly} coincide, and agree with Grigor'yan's basis when the digraph $G$ contains neither double edges or multisquares.

To end the section, we use the results above and the theory developed in Section~\ref{sec:Coefficients} to prove a universal coefficient theorem for path homology in dimensions less than $3$. As it turns out, dimension $3$ is exactly where one can no longer expect a classical universal coefficient theorem to be valid for path homology. This failure is quite explicit in dimension $3$, and we illustrate this in Example~\ref{ex:UCTfails} below by studying the digraph $\mathcal{G}$ originally defined in~\cite{Fu2024}.

Recall that a chain homomorphism $\mu_*\colon\Omega_*(G;\mathbb{Z})\otimes R\rightarrow \Omega_*(G;R)$ was constructed in Section~\ref{sec:Coefficients} for any digraph $G$ and ring $R$. Letting $Z_n(G;R)\subseteq \Omega_n(G;R)$ be the submodule of cycles and $B_n(G;R)\subseteq Z_n(G;R)$ the submodule of boundaries, there are induced maps 
\[
\mu_n^Z\colon Z_n(G;\mathbb{Z})\otimes R\rightarrow Z_n(G;R),\qquad \mu_n^B\colon B_n(G;\mathbb{Z})\otimes R\rightarrow B_n(G;R). 
\]
Similarly, there is an induced map in homology $\bar{\mu}_n\colon H_n^P(G;\mathbb{Z})\otimes R\rightarrow H_n^P(G;R)$.
\begin{theorem}\label{th:Dim3Char0Char2invareince}
Let $G$ be a digraph and $R$ is a ring of characteristic $2$. Then $\mu_*\colon\Omega_*(G;\mathbb{Z})\otimes_{\mathbb{Z}} R \rightarrow\Omega_*(G;R)$ is a chain isomorphism in degrees $\leq 3$ and for each $0\leq n\leq 2$ there is an exact sequence
\begin{equation}\label{eq:unicofextseqdim3}
0\rightarrow H_n^P(G;\mathbb{Z})\otimes R \xrightarrow{\bar{\mu}_n} H_n^P(G;R)\rightarrow \mathrm{Tor}(H_{n-1}^P(G;\mathbb{Z}),R)\rightarrow 0.
\end{equation}
This sequence is natural and splits, but the splitting is not natural. 
In degree $3$ there is a natural exact sequence
\begin{equation}\label{eq:unicofdim3}
0\rightarrow \ker(\bar{\mu}_3)\rightarrow \mathcal{B}
_3(G;R)\rightarrow\mathrm{Tor}(H_2^P(G;\mathbb{Z}),R)\rightarrow\mathrm{coker}(\bar{\mu}_3)\rightarrow0
\end{equation}
where $\mathcal{B}_3(G;R)=\mathrm{coker}(\mu_3^B\colon B_3(G;\mathbb{Z})\otimes R\rightarrow B_3(G;R))$.
\end{theorem}
\begin{proof}
It follows from Corollary~\ref{cor:CoeffChange3} that the chain map $\mu_{*}\colon\Omega_*(G;\mathbb{Z})\otimes R \rightarrow\Omega_*(G;R)$ is injective in all degrees. It is surjective in degrees $<2$ by the discussion in Section~\ref{sec:LowDimBasis} surrounding equations~\eqref{eq:Dim0Basis}~and~\eqref{eq:Dim1Basis}, and in degree $2$ by Proposition~\ref{prop:Dim2Base}. Finally, Corollary~\ref{cor:3PathStructureGeneration} shows that $\mu_*$ is surjective also in degree $3$ if $R$ has characteristic $2$. In particular, $\mu_*$ is a chain isomorphism in degrees $\leq3$.

Since $\mu_*$ is a chain isomorphism in degrees $\leq3$, the proof of the universal coefficient theorem in \cite[Corollary 7.56, pg. 450]{Rotman2009} provides the exact sequences in degrees $\leq2$. 
In degree $3$ we do not get an exact sequence because $\mu_4$ need not necessarily be an isomorphism, so the map $H_3(\Omega_*(G;\mathbb{Z})\otimes R;\partial^P\otimes1)\rightarrow H^P_3(G;R)$ induced by $\mu_*$ may no longer be used to identify the two groups.
To proceed, consider the diagram
\begin{equation}\label{eq:KernelsBoundariesDim3}
\begin{gathered}
\xymatrix{\cdots\ar[r]&B_3(G;\mathbb{Z})\otimes R\ar[r]\ar[d]^-{\mu_3^B}&Z_3(G;\mathbb{Z})\otimes R\ar[r]\ar[d]^-{\mu_3^Z}&H_3^P(G;\mathbb{Z})\otimes R\ar[d]^-{\bar{\mu}_3}\ar[r]&0\\
0\ar[r]&B_3(G;R)\ar[r]&Z_3(G;R)\ar[r]&H_3^P(G;R)\ar[r]& 0.}
\end{gathered}
\end{equation}
The diagram commutes, the bottom row is exact by construction, and the top row is exact by the right exactness of the tensor product.
Since $\mu_3\colon\Omega_3(G;\mathbb{Z})\otimes R\rightarrow\Omega_3(G;R)$ is bijective, $\mu_3^Z$ is injective, so the Snake Lemma~\cite[Corollary 6.12 and Exercise 6.5]{Rotman2009} gives an exact sequence
\[
0\rightarrow \ker(\bar{\mu}_3)\rightarrow\mathrm{coker}(\mu_3^B)\rightarrow\mathrm{coker}(\mu_3^Z)\rightarrow\mathrm{coker}(\bar{\mu}_3)\rightarrow0.
\]

The module $\mathrm{coker}(\mu_3^B)=\mathcal{B}_3(G;R)$ has already been defined, so to obtain the sequence~\eqref{eq:unicofdim3} it remains only to study $\mathrm{coker}(\mu_3^Z)$.
To this end, consider the commutative diagram
\[
\xymatrix{0\ar[r]&Z_3(G;\mathbb{Z})\otimes R\ar[r]\ar[d]^-{\mu_3^Z}&\Omega_3(G;\mathbb{Z})\otimes R\ar[d]^-{\mu_3}\ar[r]^-{\partial^P\otimes 1}&B_2(G;\mathbb{Z})\otimes R\ar[d]^-{\mu_2^B}\ar[r]&0\\
0\ar[r]&Z_3(G;R)\ar[r]&\Omega_3(G;R)\ar[r]^-{\partial^P}& B_2(G;R)\ar[r]&0.}
\]
The bottom row is exact by construction, and the top row is exact because $B_2(G;\mathbb{Z})$ is free abelian. Since $\mu_3$ is bijective, the Snake Lemma gives an isomorphism 
\[
\mathrm{coker}(\mu_3^Z)\cong \ker(\mu_2^B).
\]

Finally, similar to diagram~\eqref{eq:KernelsBoundariesDim3}, consider the diagram
\[
\xymatrix{\cdots\ar[r]&B_2(G;\mathbb{Z})\otimes R\ar[r]^-{i\otimes1}\ar[d]^-{\mu_2^B}&Z_2(G;\mathbb{Z})\otimes R\ar[r]\ar[d]^-{\mu_2^Z}&H_2^P(G;\mathbb{Z})\otimes R\ar[d]^-{\bar{\mu}_2}\ar[r]&0\\
0\ar[r]&B_2(G;R)\ar[r]^-{\bar{i}}&Z_2(G;R)\ar[r]&H_2^P(G;R)\ar[r]& 0}
\]
where $i,\bar{i}$ are the inclusions. The injectivity of $\bar{i}$ and the commutativity of the first square yield
\[
\ker(\mu_2^B)=\ker(\bar i\circ\mu_2^B)=\ker(\mu_2^Z\circ(i\otimes1)).
\]
Since $\mu_2\colon\Omega_2(G;\mathbb{Z})\otimes R\rightarrow\Omega_2(G;R)$ is injective, so is $\mu_2^Z$. Consequently,
\[
\ker(\mu_2^Z\circ(i\otimes1))=\ker(i\otimes1)\cong\mathrm{Tor}(H_2^P(G;\mathbb{Z}),R). 
\]
This completes the construction of the sequence. 
Since its construction at each stage is natural, the sequence is natural. \qed
\end{proof}
\begin{remark}
Inspection of the proof shows that the map $\mathcal{B}_3(G;R)\rightarrow\mathrm{Tor}(H_2^P(G;\mathbb{Z}),R)$ appearing in the sequence~\eqref{eq:unicofdim3} is induced by the composite
\[
B_3(G;R)\subseteq \Omega_3(G;R)\xrightarrow{\mu_3^{-1}}
\Omega_3(G;\mathbb{Z})\otimes R\xrightarrow{\partial^P\otimes1}B_2(G;\mathbb{Z})\otimes R.
\]
This can be useful to know, e.g. when its injectivity or surjectivity can be verified directly.
\end{remark}
The module $\mathcal{B}_3(G;R)=\mathrm{coker}(\mu_3^B\colon B_3(G;\mathbb{Z})\otimes R\rightarrow B_3(G;R))$ is the sole obstruction to $G$ having a universal coefficient exact sequence of the form~\eqref{eq:unicofextseqdim3} also for $n=3$. While the condition $\mathcal{B}_3(G;R)=0$ is generally difficult to check, it is guaranteed to hold in many cases, for instance if $\mu_4\colon\Omega_4(G;\mathbb{Z})\otimes R\to\Omega_4(G;R)$ is an isomorphism, and in particular when $G$ has no paths of length greater than $3$.
\begin{corollary}\label{cor:Dim3Char0Char2UCT}
Let $G$ be a digraph and $R$ a ring of characteristic $2$.
\begin{enumerate}
\item
If $\mathcal{B}_3(G;R)=0$, then there is a splitting exact sequence
\[
0\rightarrow H_3^P(G;\mathbb{Z})\otimes R \xrightarrow{\bar{\mu}_3} H_3^P(G;R)\rightarrow \mathrm{Tor}(H_{2}^P(G;\mathbb{Z}), R)\rightarrow 0.
\]
A sufficient condition for this to occur is that $\Omega_4(G;R)=0$.
\item
If $\mathrm{Tor}(H_2^P(G;\mathbb{Z}),R)=0$, then there is an exact sequence
\[
0\rightarrow \mathcal{B}_3(G;R)\rightarrow H_3^P(G;\mathbb{Z})\otimes R\xrightarrow{\bar{\mu}_3} H_3^P(G;R)\rightarrow 0. \qed
\]
\end{enumerate}
\end{corollary}

The situation for a universal coefficient theorem involving rings of odd prime characteristic is more complicated. Proposition~\ref{prop:univcoeffs} gives a universal coefficient theorem in this case for digraphs containing no multisquares, but unpublished examples due to the present authors~\cite[Example 6.2]{BurfittCutler2024} show that in general there is no universal coefficient exact sequence of the form~\eqref{eq:unicofextseqdim3} when $n=3$ and $R$ has odd prime characteristic. Still, this leaves open the possibility that the sequence~\eqref{eq:unicofextseqdim3} may be exact when $n=2$ and $R$ is a ring of odd prime characteristic. A positive answer to question $(2)$ from the introduction would solve also this problem in the affirmative.
\begin{exmp}\normalfont\label{ex:UCTfails}
We compute the integral path homology of the digraph $\mathcal{G}$ defined in~\cite[\S 5.4, equation (5.9)]{Fu2024} and drawn below, and show that the obstruction module $\mathcal{B}_3(\mathcal{G};\mathbb{Z}/2)$ is nontrivial.
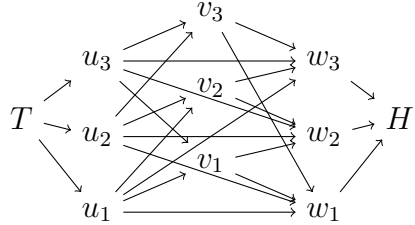
\begin{center}
        \tikz {
            \node (T) at (0.5,0) {$T$};
            \node (u1) at (1.5,-1.25) {$u_1$};
            \node (u2) at (1.5,-0.25) {$u_2$};
            \node (u3) at (1.5,0.75) {$u_3$};
            \node (v1) at (3,-0.6) {$v_1$};
            \node (v2) at (3,0.4) {$v_2$};
            \node (v3) at (3,1.4) {$v_3$};
            \node (w1) at (4.5,-1.25) {$w_1$};
            \node (w2) at (4.5,-0.25) {$w_2$};
            \node (w3) at (4.5,0.75) {$w_3$};
            \node (H) at (5.5,0) {$H$};
            \draw[->] (T) -- (u1);
            \draw[->] (T) -- (u2);
            \draw[->] (T) -- (u3);
            \draw[->] (u1) -- (v1);
            \draw[->] (u1) -- (v2);
            \draw[->] (u2) -- (v2);
            \draw[->] (u2) -- (v3);
            \draw[->] (u3) -- (v3);
            \draw[->] (u3) -- (v1);
            \draw[->] (v1) -- (w1);
            \draw[->] (v1) -- (w2);
            \draw[->] (v2) -- (w2);
            \draw[->] (v2) -- (w3);
            \draw[->] (v3) -- (w3);
            \draw[->] (v3) -- (w1);
            \draw[->] (u1) -- (w1);
            \draw[->] (u1) -- (w3);
            \draw[->] (u2) -- (w2);
            \draw[->] (u2) -- (w1);
            \draw[->] (u3) -- (w3);
            \draw[->] (u3) -- (w2);
            \draw[->] (w1) -- (H);
            \draw[->] (w2) -- (H);
            \draw[->] (w3) -- (H);
        }
    \end{center}
The content of~\cite[Proposition 5.3]{Fu2024} is that the universal coefficient theorem must fail somewhere for $\mathcal{G}$, but the methods of \cite{Fu2024} do not permit a deeper analysis. Knowing that $\mathcal{B}_3(\mathcal{G};\mathbb{Z}/2)\neq0$ provides a precise explanation for this failure.

The path homology of $\mathcal{G}$ with field coefficients is easily calculated directly by computer. We have used code written by the first author and made available at \cite{Burfitt2024} to compute
\[
H^P_*(\mathcal{G};\mathbb{Z}/2)\cong\begin{cases}\mathbb{Z}/2&*=0\\
\mathbb{Z}/2&*=2\\
0&\text{otherwise}\end{cases}
\qquad
H^P_*(\mathcal{G};\mathbb{Z}/p)\cong\begin{cases}\mathbb{Z}/p&*=0\\
0&\text{otherwise}\end{cases}
\]
for $p$ any odd prime.
Using Theorem~\ref{th:Dim3Char0Char2invareince}, we will explain that the universal coefficient theorem fails for the mod $2$ path homology of $\mathcal{G}$ in degree $3$. 

To proceed, note that $G$ has no multisquares, so Proposition~\ref{prop:univcoeffs} gives for each odd prime $p$ and each $n\geq0$ a universal coefficient exact sequence
\[
0\rightarrow H_n^P(\mathcal{G};\mathbb{Z})\otimes\mathbb{Z}/p \rightarrow H_n^P(\mathcal{G};\mathbb{Z}/p)\rightarrow \mathrm{Tor}(H_{n-1}^P(\mathcal{G};\mathbb{Z}),\mathbb{Z}/p)\rightarrow 0.
\]
Since $H_n^P(\mathcal{G};\mathbb{Z}/p)=0$ for $n\geq1$, we have $H_n^P(\mathcal{G};\mathbb{Z})\otimes\mathbb{Z}/p=0$ for all $n\geq1$, and hence that $H_n^P(\mathcal{G};\mathbb{Z})$ is $2$-torsion for all $n\geq1$. On the other hand, for $n\leq 2$ Theorem~\ref{th:Dim3Char0Char2invareince} gives the universal coefficient exact sequence
\[
0\rightarrow H_n^P(\mathcal{G};\mathbb{Z})\otimes\mathbb{Z}/2 \rightarrow H_n^P(\mathcal{G};\mathbb{Z}/2)\rightarrow \mathrm{Tor}(H_{n-1}^P(\mathcal{G};\mathbb{Z}),\mathbb{Z}/2)\rightarrow 0.
\]
Since $H_1^P(\mathcal{G};\mathbb{Z}/2)=0$, we have $H_1^P(\mathcal{G};\mathbb{Z})\otimes\mathbb{Z}/2=0$, and hence $H_1^P(\mathcal{G};\mathbb{Z})=0$ follows from the remarks above. 
This leaves $H_2^P(\mathcal{G};\mathbb{Z})\otimes\mathbb{Z}/2\cong H_2^P(\mathcal{G};\mathbb{Z}/2)\cong\mathbb{Z}/2$. Consequently, $H_2^P(\mathcal{G};\mathbb{Z})\cong\mathbb{Z}/2^r$ for some $r\geq1$.

Now, it is explained in~\cite[pg.17]{Fu2024} that $\Omega_4(\mathcal{G};\mathbb{Q})=0$. Since $\Omega_4(\mathcal{G};\mathbb{Z})$ is free abelian and $\Omega_4(\mathcal{G};\mathbb{Z})\otimes\mathbb{Q}\cong \Omega_4(\mathcal{G};\mathbb{Q})=0$, we have $\Omega_4(\mathcal{G};\mathbb{Z})=0$. Consequently, $H_4^P(\mathcal{G};\mathbb{Z})=0$ and $H_3^P(\mathcal{G};\mathbb{Z})=Z_3(\mathcal{G};\mathbb{Z})$, which is free abelian. But it was shown above that $H_3^P(\mathcal{G};\mathbb{Z})$ is $2$-torsion, so it must be that $H_3^P(\mathcal{G};\mathbb{Z})=0$. Finally, $H_0^P(\mathcal{G};\mathbb{Z})\cong\mathbb{Z}$ follows from the connectedness of $\mathcal{G}$. Thus we have
\[
H^P_*(\mathcal{G};\mathbb{Z})\cong\begin{cases}\mathbb{Z}&*=0\\
\mathbb{Z}/2^r&*=2\\
0&\text{otherwise}\end{cases}
\]
for some $r\geq1$. Moreover, applying Proposition~\ref{Pr:coefisopid} shows that $H^P_*(\mathcal{G};\mathbb{Q})\cong H^P_*(\mathcal{G};\mathbb{Z})\otimes\mathbb{Q}$ is trivial in all positive degrees.

We do not know of a quick way to determine the value of $r$, but direct computation yields $r=1$. In fact, a lengthy calculation yields $H^P_2(\mathcal{G};\mathbb{Z})\cong\mathbb{Z}/2$ generated by the inductive element
\[
C = \sum_{j=1}^3 S_j - S_j'
\]
where $S_j$ and $S_j'$ for $j=1,2,3$ are the directed squares
\[
S_j = e_{v_j,w_j,H} - e_{v_j,w_{j+1},H},\qquad
S'_j = e_{T,u_j,v_{j+1}} - e_{T,u_{j+1},v_{j+1}}
\]
with all indices treated modulo $3$.

To conclude the discussion, consider the second exact sequence from Theorem~\ref{th:Dim3Char0Char2invareince}
\[
0\rightarrow \ker(\bar{\mu}_3)\rightarrow\mathcal{B}_3(\mathcal{G};\mathbb{Z}/2)\rightarrow\mathrm{Tor}(H_2^P(\mathcal{G};\mathbb{Z}),\mathbb{Z}/2)\rightarrow\mathrm{coker}(\bar{\mu}_3)\rightarrow0.
\]
From the computations above, $\ker(\bar{\mu}_3)=0=\mathrm{coker}(\bar{\mu}_3)$. Since $\Omega_4(\mathcal{G};\mathbb{Z})=0$, also $B_3(\mathcal{G};\mathbb{Z})=0$, so the sequence above degenerates to a chain of isomorphisms
\[
B_3(\mathcal{G};\mathbb{Z}/2)\cong\mathcal{B}_3(\mathcal{G};\mathbb{Z}/2)\cong\mathrm{Tor}(H_2^P(\mathcal{G};\mathbb{Z}),\mathbb{Z}/2)\cong\mathbb{Z}/2.
\]
\end{exmp}

\section{Important examples}\label{sec:ImportantExamples}

We now construct a number of explicit examples demonstrating that whilst defining face multigraphs in Definition~\ref{def:FaceMultiGraph}, it is too restrictive to assume that the ridge decomposition in equation~\eqref{def:FaceMultiGraph} consists of a single nonzero term with coefficient $\pm 1$.
Throughout the section, $G$ denotes a digraph and $R$ a commutative, unital ring, unless otherwise stated.

\subsection{Decomposition of a \texorpdfstring{$\delta_{n,v}^h$}{head vertex face maps} image as distinct inductive elements}

The following example demonstrates that for $n \geq 4$ the image of an upper inductive element under $\delta_{n,v}^h$ for $v \in G$ need not be an upper inductive element in $\Omega_{n-1}(G;R)$.

\begin{exmp}\label{exam:SingleFaceDecomposition}\normalfont
    Consider the following digraph $G$.
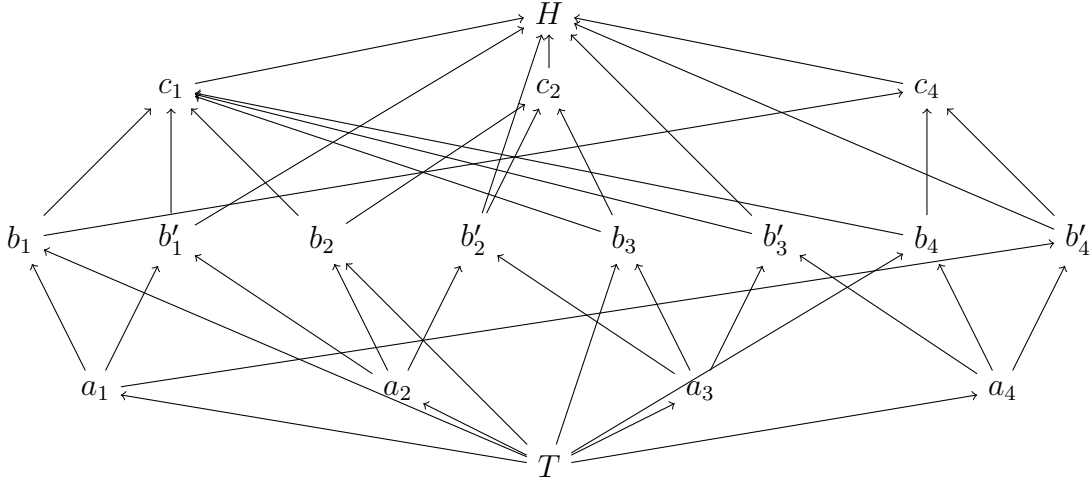
    \begin{center}
        \tikz {
            \node (H) at (0,7) {$H$};
            \node (T) at (0,1) {$T$};
            \node (a1) at (-6,2) {$a_1$};
            \node (a2) at (-2,2) {$a_2$};
            \node (a3) at (2,2) {$a_3$};
            \node (a4) at (6,2) {$a_4$};
            \node (b1) at (-7,4) {$b_1$};
            \node (b2) at (-3,4) {$b_2$};
            \node (b3) at (1,4) {$b_3$};
            \node (b4) at (5,4) {$b_4$};
            \node (b'1) at (-5,4) {$b'_1$};
            \node (b'2) at (-1,4) {$b'_2$};
            \node (b'3) at (3,4) {$b'_3$};
            \node (b'4) at (7,4) {$b'_4$};
            \node (c1) at (-5,6) {$c_1$};
            \node (c2) at (0,6) {$c_2$};
            \node (c4) at (5,6) {$c_4$};
            \draw[->] (T) -- (a1);
            \draw[->] (T) -- (a2);
            \draw[->] (T) -- (a3);
            \draw[->] (T) -- (a4);
            \draw[->] (T) -- (b1);
            \draw[->] (T) -- (b2);
            \draw[->] (T) -- (b3);
            \draw[->] (T) -- (b4);
            \draw[->] (a1) -- (b1);
            \draw[->] (a1) -- (b'1);
            \draw[->] (a1) -- (b'4);
            \draw[->] (a2) -- (b2);
            \draw[->] (a2) -- (b'2);
            \draw[->] (a2) -- (b'1);
            \draw[->] (a3) -- (b3);
            \draw[->] (a3) -- (b'3);
            \draw[->] (a3) -- (b'2);
            \draw[->] (a4) -- (b4);
            \draw[->] (a4) -- (b'4);
            \draw[->] (a4) -- (b'3);
            \draw[->] (b1) -- (c1);
            \draw[->] (b'1) -- (c1);
            \draw[->] (b2) -- (c1);
            \draw[->] (b2) -- (c2);
            \draw[->] (b'2) -- (c2);
            \draw[->] (b3) -- (c2);
            \draw[->] (b3) -- (c1);
            \draw[->] (b'3) -- (c1);
            \draw[->] (b4) -- (c1);
            \draw[->] (b4) -- (c4);
            \draw[->] (b'4) -- (c4);
            \draw[->] (b1) -- (c4);
            \draw[->] (c1) -- (H);
            \draw[->] (c2) -- (H);
            \draw[->] (c4) -- (H);
            \draw[->] (b'1) -- (H);
            \draw[->] (b'2) -- (H);
            \draw[->] (b'3) -- (H);
            \draw[->] (b'4) -- (H);
        }
    \end{center}
    We demonstrate that $\Omega_{4}(G;R)$ contains a single inductive element $I_4$ up to sign and that $\delta^h_{4,c_1}(I_4)$ is not an upper inductive element.
    Throughout the example, all index values on $a_i$, $b_i$, and $b'_i$ are assumed to be integers modulo $4$, and we set $c_3 = c_1$ and $c_0 = c_4$.
    
    We first show that there are four $3$-dimensional upper inductive elements up to sign whose head vertex is not $H$, that is
    \[
        L_i = e_{T,a_i,b_i,c_i} - e_{T,a_i,b'_i,c_i} + e_{T,a_{i+1},b'_i,c_i} - e_{T,a_{i+1},b_{i+1},c_i}
    \]
    for $i=1,2,3,4$. 
    
    To this end, notice that each of the elements $L_i$ can be described as an element $L_{i}\in\Omega_3(\mathbb{L}_i;R)$ where $\mathbb{L}_i$ are the $G$ full subdigraphs below.
    
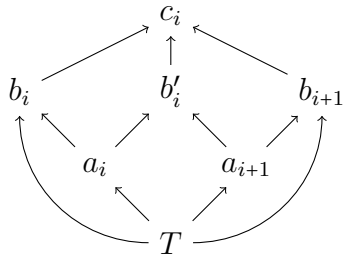
\begin{center}
        \tikz {
            \node (T) at (0,2) {$T$};
            \node (ai) at (-1,3) {$a_i$};
            \node (ai+1) at (1,3) {$a_{i+1}$};
            \node (bi) at (-2,4) {$b_i$};
            \node (bi+1) at (2,4) {$b_{i+1}$};
            \node (b'i) at (0,4) {$b'_i$};
            \node (ci) at (0,5) {$c_i$};
            \draw[->] (T) -- (ai);
            \draw[->] (T) -- (ai+1);
            \draw[->] (T) to [out=180,in=270] (bi);
            \draw[->] (T) to [out=0,in=270] (bi+1);
            \draw[->] (ai) -- (bi);
            \draw[->] (ai) -- (b'i);
            \draw[->] (ai+1) -- (bi+1);
            \draw[->] (ai+1) -- (b'i);
            \draw[->] (bi) -- (ci);
            \draw[->] (b'i) -- (ci);
            \draw[->] (bi+1) -- (ci);
        }
    \end{center}
    The above subdigraphs are isomorphic to $\mathbb{L}^l_2$ from Definition~\ref{def:LowerLinearElements} with $L_i$ the corresponding $3$-path element. Hence, the $L_i$ are inductive elements as demonstrated in the proof of Theorem~\ref{thm:Omega3Explicitly}.
    As $G$ contains no directed cycles and the maximal path length is $4$, the $L_i$ are the only upper inductive elements in $\Omega_3(G;R)$ without head vertex $H$.

    Recall that by Proposition~\ref{prop:LowDimInductiveElements}, directed triangles are inductive elements in $\Omega_2(G;R)$ up to sign. Denote by $T_i$, the directed triangle $e_{T,a_i,b_i}$.
    Inductive element $I_4\in\Omega_4(G;R)$ can now be obtained as the upper extension $[L_1,L_2,L_3,L_4]^H$ using the following inductive structure.
    
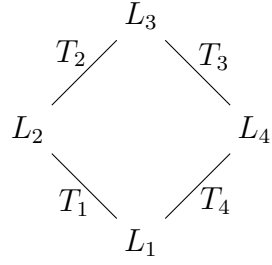
\begin{center}
        \tikz {
            \node (T1) at (0,0) {$L_1$};
            \node (T2) at (-1.5,1.5) {$L_2$};
            \node (T3) at (0,3) {$L_3$};
            \node (T4) at (1.5,1.5) {$L_4$};
            \draw[-] (T1) -- (T2) node[pos=0.25,left] {$T_1$};
            \draw[-] (T2) -- (T3) node[pos=0.8,left] {$T_2\:$};
            \draw[-] (T3) -- (T4) node[pos=0.25,right] {$\:T_3$};
            \draw[-] (T4) -- (T1) node[pos=0.725,right] {$\:T_4$};
        }
    \end{center}
    This connected complete face multigraph is unique, as it includes all inductive elements with head vertex not $H$ and all edges are uniquely determined by their labels.
    Therefore, $I_4$ is up to sign the unique upper inductive generator of $\Omega_4(G;R)$. Since $h(L_1)=h(L_2)=c_1$ we have
    \[
        \delta_{4,c_1}^h(I_4)=L_1+L_2
    \]
    demonstrating that in general the image of $\delta^h_{n,v}$ when applied to upper inductive elements need not yield an upper inductive element.
    \end{exmp}

\subsection{Boundary matrix multiplicities}\label{sec:DiffMultiplicityCounter}

The next example demonstrates that the path homology differential with respect to an inductive basis can contain arbitrary multiplicities in its boundary matrix.

\begin{exmp}\label{exam:ArbitraryMultiplicities}\normalfont
    Let $t\geq 2$ be an integer.
    Throughout this example, all index values are assumed to be integers modulo $2t$.
    We construct a digraph $\mathbb{M}_t$ with vertices
    \[
        V_{\mathbb{M}_t} = 
        \{
        T,
        u^A_1,u^A_2,u^{B},
        v_1^A,v_2^A,v_1^{B},\dots,v_{2t}^{B},
        w^A,w^{B}_1,\dots,w^{B}_{2t},
        H
        \}
    \]
    and edges
    \begin{align*}
        &
        T \to u_1^A, \:
        T \to u_2^A, \:
        T \to u^{B}, 
        \\ &
        u_1^A \to v_1^A, \:
        u_1^A \to v_2^A, \:
        u_2^A \to v_1^A, \:
        u_2^A \to v_2^A, \:
        u^{B} \to v_i^{B}, \:
        u_1^A \to v_{2i+1}^{B}, \:
        u_2^A \to v_{2i}^{B},
        \\ &
        v_1^A \to w^A, \:
        v_2^A \to w^A, \:
        v_i^{B} \to w_i^{B}, \:
        v_i^{B} \to w_{i+1}^{B}, \:
        v_1^A \to w_{2i}^{B}, \:
        v_2^A \to w_{2i+1}^{B},
        \\ &
        w^A \to H, \:
        w^{B}_i \to H
    \end{align*}
    for $i=1,\dots,2t$.
    In the case $t=2$, we obtain the following digraph $\mathbb{M}_2$.
    
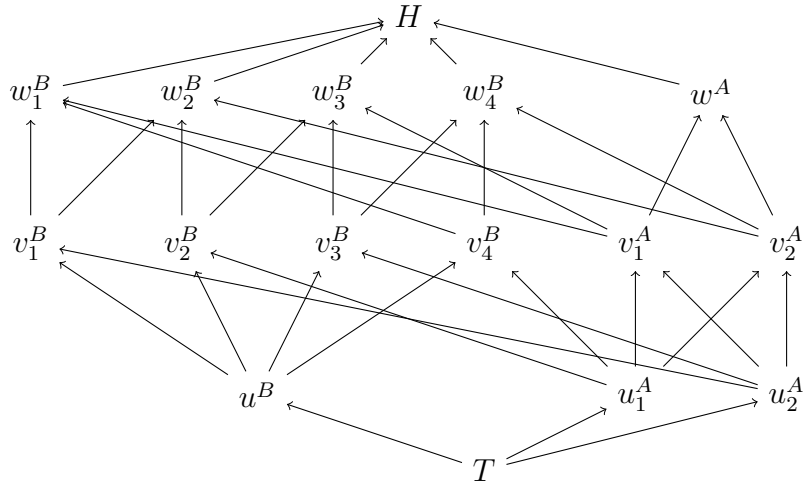
\begin{center}
        \tikz {
            \node (H) at (2,7) {$H$};
            \node (T) at (3,1) {$T$};
            \node (xA) at (5,2) {$u_1^A$};
            \node (x'A) at (7,2) {$u_2^A$};
            \node (x) at (0,2) {$u^{B}$};
            \node (yA) at (5,4) {$v_1^A$};
            \node (y'A) at (7,4) {$v_2^A$};
            \node (yBC) at (-3,4) {$v^{B}_1$};
            \node (yCD) at (-1,4) {$v^{B}_2$};
            \node (yDE) at (1,4) {$v^{B}_3$};
            \node (yBE) at (3,4) {$v^{B}_4$};
            \node (hA) at (6,6) {$w^A$};
            \node (hB) at (-3,6) {$w^{B}_1$};
            \node (hC) at (-1,6) {$w^{B}_2$};
            \node (hD) at (1,6) {$w^{B}_3$};
            \node (hE) at (3,6) {$w^{B}_4$};
            \draw[->] (T) -- (xA);
            \draw[->] (T) -- (x'A);
            \draw[->] (T) -- (x);
            \draw[->] (x) -- (yBC);
            \draw[->] (x) -- (yCD);
            \draw[->] (x) -- (yDE);
            \draw[->] (x) -- (yBE);
            \draw[->] (xA) -- (yA);
            \draw[->] (xA) -- (y'A);
            \draw[->] (xA) -- (yCD);
            \draw[->] (xA) -- (yBE);
            \draw[->] (x'A) -- (yA);
            \draw[->] (x'A) -- (y'A);
            \draw[->] (x'A) -- (yBC);
            \draw[->] (x'A) -- (yDE);
            \draw[->] (yA) -- (hA);
            \draw[->] (yA) -- (hB);
            \draw[->] (yA) -- (hD);
            \draw[->] (y'A) -- (hA);
            \draw[->] (y'A) -- (hC);
            \draw[->] (y'A) -- (hE);
            \draw[->] (yBC) -- (hB);
            \draw[->] (yBC) -- (hC);
            \draw[->] (yCD) -- (hC);
            \draw[->] (yCD) -- (hD);
            \draw[->] (yDE) -- (hD);
            \draw[->] (yDE) -- (hE);
            \draw[->] (yBE) -- (hE);
            \draw[->] (yBE) -- (hB);
            \draw[->] (hA) -- (H);
            \draw[->] (hB) -- (H);
            \draw[->] (hC) -- (H);
            \draw[->] (hD) -- (H);
            \draw[->] (hE) -- (H);
        }
    \end{center}
    All maximal length paths in $\mathbb{M}_{t}$ have length $4$ and $\mathbb{M}_{t}$ contains no directed triangles. It follows that $\Omega_n(\mathbb{M}_{t};\mathbb{Z})=0$ for $n\geq5$.
    
    We now identify an element $I^t_4\in\Omega_4(\mathbb{M}_{t};\mathbb{Z})$ as an upper inductive element by describing all upper inductive structures on it.
    Using this construction, we show that $\partial_4^P(I^t_4) = t\cdot A + \cdots$ for a certain inductive element $A\in\Omega_3(\mathbb{M}_{t};\mathbb{Z})$. Although the example will be presented with integral coefficients, the conclusions are equally valid for coefficients in a field of characteristic $0$ or $p$, for a prime $p>t$.
    
    First we note that, if the vertex $H$ and all its incoming edges are removed from $\mathbb{M}_{t}$, the remaining subdigraph is the union of the full subdigraphs
    
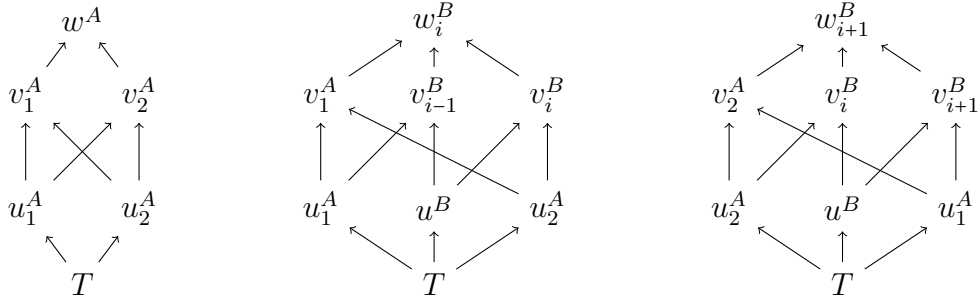
\begin{center}
        \tikz {
            \node (T) at (0,0) {$T$};
            \node (xA) at (-0.75,1) {$u_1^A$};
            \node (x-A) at (0.75,1) {$u_2^A$};
            \node (yA) at (-0.75,2.5) {$v_1^A$};
            \node (y-A) at (0.75,2.5) {$v_2^A$};
            \node (zA) at (0,3.5) {$w^A$};
            \draw[->] (T) -- (xA);
            \draw[->] (T) -- (x-A);
            \draw[->] (xA) -- (yA);
            \draw[->] (xA) -- (y-A);
            \draw[->] (x-A) -- (yA);
            \draw[->] (x-A) -- (y-A);
            \draw[->] (yA) -- (zA);
            \draw[->] (y-A) -- (zA);
        }
        \;\;\;\;\;\;\;\;\;\;\;\;
        \tikz {
            \node (T) at (0,0) {$T$};
            \node (xA) at (-1.5,1) {$u_1^A$};
            \node (x) at (0,1) {$u^{B}$};
            \node (x-A) at (1.5,1) {$u_2^A$};
            \node (yA) at (-1.5,2.5) {$v_1^A$};
            \node (yBE) at (0,2.5) {$v^{B}_{i-1}$};
            \node (yBC) at (1.5,2.5) {$v^{B}_i$};
            \node (zB) at (0,3.5) {$w^{B}_i$};
            \draw[->] (T) -- (xA);
            \draw[->] (T) -- (x);
            \draw[->] (T) -- (x-A);
            \draw[->] (xA) -- (yA);
            \draw[->] (xA) -- (yBE);
            \draw[->] (x) -- (yBC);
            \draw[->] (x) -- (yBE);
            \draw[->] (x-A) -- (yA);
            \draw[->] (x-A) -- (yBC);
            \draw[->] (yA) -- (zB);
            \draw[->] (yBC) -- (zB);
            \draw[->] (yBE) -- (zB);
        }
        \;\;\;\;\;\;\;\;\;\;\;\;
        \tikz {
            \node (T) at (0,0) {$T$};
            \node (x-A) at (-1.5,1) {$u_2^A$};
            \node (x) at (0,1) {$u^{B}$};
            \node (xA) at (1.5,1) {$u_1^A$};
            \node (y-A) at (-1.5,2.5) {$v_2^A$};
            \node (yBC) at (0,2.5) {$v_i^{B}$};
            \node (yCD) at (1.5,2.5) {$v_{i+1}^{B}$};
            \node (zC) at (0,3.5) {$w^{B}_{i+1}$};
            \draw[->] (T) -- (xA);
            \draw[->] (T) -- (x);
            \draw[->] (T) -- (x-A);
            \draw[->] (x-A) -- (y-A);
            \draw[->] (x-A) -- (yBC);
            \draw[->] (x) -- (yCD);
            \draw[->] (x) -- (yBC);
            \draw[->] (xA) -- (yA);
            \draw[->] (xA) -- (yCD);
            \draw[->] (y-A) -- (zC);
            \draw[->] (yCD) -- (zC);
            \draw[->] (yBC) -- (zC);
        }
    \end{center}
    where $i=1,3\dots,2t-1$ is an odd integer.
    In particular, by Definition~\ref{def:Trapezohedron}, the digraphs above are isomorphic to trapezohedra of orders $2$, $3$, and $3$, respectively. By Proposition~\ref{prop:Trapezohedron}, each of the trapezohedron digraphs contains a unique trapezohedron element generating the dimension $3$ path chains given in equation~\eqref{eq:TrapezohedronElement}.
    
    Denote these trapezohedron elements by $A \in \Omega_3^{T,w^A}(\mathbb{M}_{t};\mathbb{Z})$, $B_i \in \Omega_3^{T,w^B_i}(\mathbb{M}_{t};\mathbb{Z})$, and $B_{i+1}\in \Omega_3^{T,w^B_{i+1}}(\mathbb{M}_{t};\mathbb{Z})$ for $i=1,3\dots,2t-1$, respectively. Furthermore, write the directed squares $S_1^A = e_{T,u^A_2,v^A_1} - e_{T,u^A_1,v^A_1}$, $S_2^A = e_{T,u^A_1,v^A_2} - e_{T,u^A_2,v^A_2}$, $S^{B}_i = e_{T,u^B,v^B_i} - e_{T,u^A_2,v^B_i}$, and $S^{B}_{i+1} = e_{T,u^B,v^B_{i+1}} - e_{T,u^A_1,v^B_{i+1}}$. These directed squares form a basis of the submodule of $\Omega_{2}(\mathbb{M}_t;\mathbb{Z})$ whose elements have tail vertex $T$.
    
    We obtain an inductive element of $I_4^t \in \Omega_4(\mathbb{M}_t;\mathbb{Z})$ as the upper extension
    \[
        I_4^t = [\underbrace{A+\cdots+A}_{t}+B_1+\dots+B_{2t}]^H
    \]
    which is an extension over the following connected $H$-complete upper face multigraph.
    
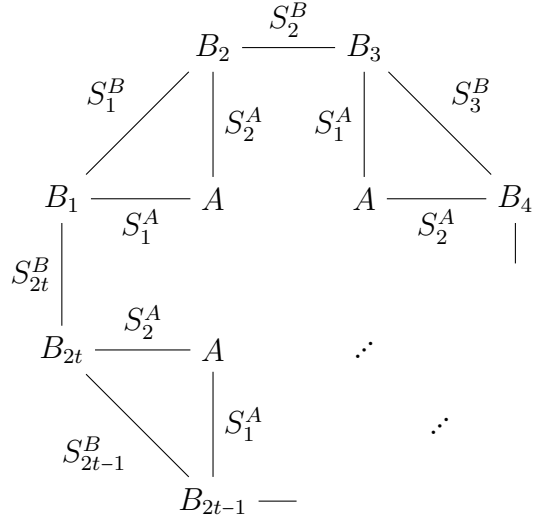
\begin{center}
        \tikz {
            \node (B1) at (0,4) {$B_1$};
            \node (B2) at (2,6) {$B_2$};
            \node (B3) at (4,6) {$B_3$};
            \node (B4) at (6,4) {$B_4$};
            \node (Bt-1) at (2,0) {$B_{2t-1}$};
            \node (Bt) at (0,2) {$B_{2t}$};
            \node (A2) at (2,4) {$A$};
            \node (A4) at (4,4) {$A$};
            \node (A2t) at (2,2) {$A$};
            \node (D1) at (4,2) {$\udots$};
            \node (D2) at (5,1) {$\udots$};
            \node (E1) at (6,3) {};
            \node (E2) at (3.25,0) {};
            \draw[-] (B1) -- (B2) node[midway,above left] {$S^B_{1}$};
            \draw[-] (B2) -- (B3) node[midway,above] {$S^B_{2}$};
            \draw[-] (B3) -- (B4) node[midway,above right] {$S^B_{3}$};
            \draw[-] (Bt-1) -- (Bt) node[midway,below left] {$S^B_{2t-1}$};
            \draw[-] (Bt) -- (B1) node[midway,left] {$S^B_{2t}$};
            \draw[-] (A2) -- (B1) node[midway,below] {$S^A_{1}$};
            \draw[-] (A2) -- (B2) node[midway,right] {$S^A_{2}$};
            \draw[-] (A4) -- (B3) node[midway,left] {$S^A_{1}$};
            \draw[-] (A4) -- (B4) node[midway,below] {$S^A_{2}$};
            \draw[-] (A2t) -- (Bt-1) node[midway,right] {$S^A_{1}$};
            \draw[-] (A2t) -- (Bt) node[midway,above] {$S^A_{2}$};
            \draw[-] (B4) -- (E1);
            \draw[-] (E2) -- (Bt-1);
        }
    \end{center}
    However, the upper face multigraph above is not unique.
    Nevertheless, all face multigraphs labelled by inductive elements over which $I_4^t$ upper extends are connected, containing differing placements of the edges labelled $S_1^A$ and $S_2^A$.
    Therefore, $I_4$ is indeed an inductive element.
    
    For example, when $t=2$ either of the following two face multigraphs
    would suffice as an upper inductive structure.
    
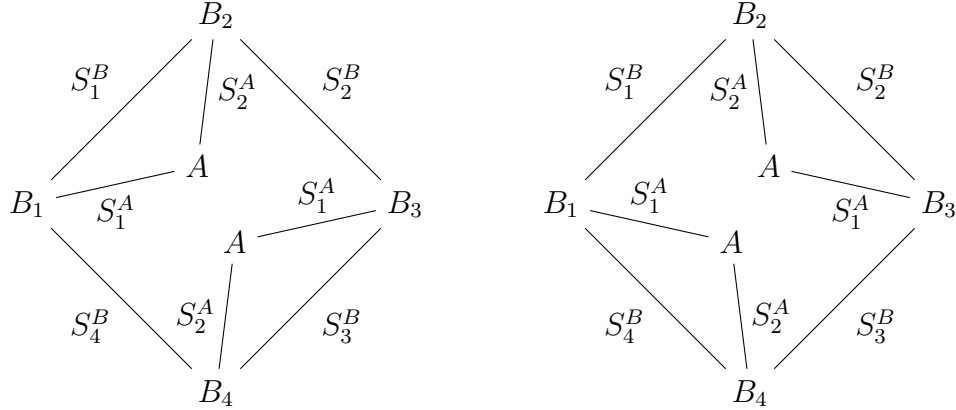
\begin{center}
        \tikz {
            \node (A1) at (-0.25,3) {$A$};
            \node (A2) at (0.25,2) {$A$};
            \node (B) at (-2.5,2.5) {$B_1$};
            \node (C) at (0,5) {$B_2$};
            \node (D) at (2.5,2.5) {$B_3$};
            \node (E) at (0,0) {$B_4$};
            \draw[-] (B) -- (C) node[midway,above left] {$S^B_{1}$};
            \draw[-] (C) -- (D) node[midway,above right] {$S^B_{2}$};
            \draw[-] (D) -- (E) node[midway,below right] {$S^B_{3}$};
            \draw[-] (E) -- (B) node[midway,below left] {$S^B_{4}$};
            \draw[-] (B) -- (A1) node[midway,below] {$S^A_1$};
            \draw[-] (C) -- (A1) node[midway,right] {$S^A_2$};
            \draw[-] (D) -- (A2) node[midway,above] {$S^A_1$};
            \draw[-] (E) -- (A2) node[midway,left] {$S^A_2$};
        }
        \;\;\;\;\;\;\;\;\;
        \tikz {
            \node (A1) at (-0.25,2) {$A$};
            \node (A2) at (0.25,3) {$A$};
            \node (B) at (-2.5,2.5) {$B_1$};
            \node (C) at (0,5) {$B_2$};
            \node (D) at (2.5,2.5) {$B_3$};
            \node (E) at (0,0) {$B_4$};
            \draw[-] (B) -- (C) node[midway,above left] {$S^B_{1}$};
            \draw[-] (C) -- (D) node[midway,above right] {$S^B_{2}$};
            \draw[-] (D) -- (E) node[midway,below right] {$S^B_{3}$};
            \draw[-] (E) -- (B) node[midway,below left] {$S^B_{4}$};
            \draw[-] (B) -- (A1) node[midway,above] {$S^A_1$};
            \draw[-] (E) -- (A1) node[midway,right] {$S^A_2$};
            \draw[-] (C) -- (A2) node[midway,left] {$S^A_2$};
            \draw[-] (D) -- (A2) node[midway,below] {$S^A_1$};
        }
    \end{center}
    
    Finally, $I_4^t \in \Omega_4(G;\mathbb{Z})$ is the unique element up to sign 
    extending only over connected $H$-complete face multigraph constructable using the inductive elements of $\Omega_3(\mathbb{M}_t;\mathbb{Z})$ with tail $T$ up to sign.
    Therefore, by Theorem~\ref{thm:BasisExtension}, $I_4^t$ is the unique generator of $\Omega_4(\mathbb{M}_t;\mathbb{Z})$ up to sign.
    Furthermore, the element $I_4^t$ is constructed using an inductive structure containing $t$ copies of the element $A\in \Omega_3(\mathbb{M}_t;\mathbb{Z})$. Hence, with respect to a $\Omega_*(\mathbb{M}_t;\mathbb{Z})$ basis extending the elements used to construct $I_4^t$, the image of $\delta^h_4$ consists of a vector containing an element of degree $t$. 
    Using the $\Omega_{*}^{*,*}(G;R)$ bigrading detailed in equation~\eqref{eq:Bigrading}, we see that the image of $\delta^h_4$ lies in a separate summand of $\Omega_3(\mathbb{M}_t;\mathbb{Z})$ from $\partial^P_{4}-\delta^h_4$. Therefore, the matrix representing $\partial^P_4$ with respect to any inductive bases of $\Omega_3(\mathbb{M}_t;\mathbb{Z})$ and $\Omega_4(\mathbb{M}_t;\mathbb{Z})$ containing the elements detailed above, has at least one entry of multiplicity $t$.
\end{exmp}

\bibliographystyle{amsplain}
\bibliography{References}

\end{document}